\UseRawInputEncoding
\documentclass{article}
\usepackage{float, graphicx}
\usepackage{amsmath, amsthm, amssymb,}
\usepackage{url}
\usepackage{mathrsfs}
\usepackage[colorlinks, bookmarks=true]{hyperref}
\usepackage{graphicx}
\usepackage{enumerate}
\usepackage[normalem]{ulem}
\usepackage{bm}
\usepackage{tikz}
\usepackage{cite}
\usepackage{color}
\usepackage{xcolor}
\usepackage{authblk}
\newcommand{\qq}[1]{[\![{#1}]\!]}
\makeatletter

\def\@settitle{\begin{center}%
    \bfseries
 \normalfont\LARGE\@title
  \end{center}%
}
\def\@setauthors{\begin{center}%
 \normalsize\@author
  \end{center}%
}

\numberwithin{equation}{section}

\newcommand{\const}{\mbox{const}}

\newcommand{\rd}{{\rm d}}

\newcommand{\bR}{{\mathbb R}}
\newcommand{\bM}{{\mathbb M}}

\def\cA{{\mathcal A}}
\def\cB{{\mathcal B}}
\def\cO{{\mathcal O}}

\def\cH{{\mathcal H}}
\def\cA{{\mathcal A}}
\def\cO{{\mathcal O}}

\newcommand{\fa}{{\frak a}}
\newcommand{\fb}{{\frak b}}
\newcommand{\fc}{{\frak c}}
\newcommand{\fC}{{\frak C}}

\newcommand{\cU}{{\mathcal U}}

\newcommand{\bmx}{{\bm{x}}}
\newcommand{\bmy}{{\bm{y}}}

\newcommand{\al}{\alpha}

\newcommand{\be}{\begin{equation}}
\newcommand{\ee}{\end{equation}}

\newcommand{\la}{\lambda}

\newcommand{\cC}{{\cal C}}

\newcommand{\cF}{{\cal F}}

\newcommand{\bC}{{\mathbb C}}

\newcommand{\bE}{{\mathbb E}}
\newcommand{\bP}{{\mathbb P}}

\renewcommand{\epsilon}{\varepsilon}
\renewcommand{\leq}{\leqslant}
\renewcommand{\geq}{\geqslant}

\renewcommand{\le}{\leq}
\renewcommand{\ge}{\geq}

\newcommand{\E}{\mathbb{E}}

\newcommand{\N}{\mathbb{N}}

\def\bR{{\mathbb R}}

\def\bB{{\mathbb B}}
\def\bC{{\mathbb C}}

\def\bP{{\mathbb P}}
\def\bE{{\mathbb E}}
\def\bP{{\mathbb P}}
\def\bQ{{\mathbb Q}}

\def\sfa{{\mathsf a}}
\def\sfb{{\mathsf b}}

\def\sfg{{\mathsf g}}
\def\sfh{{\mathsf h}}
\def\sfp{{\mathsf p}}

\newcommand{\cM}{{\cal M}}

\DeclareMathOperator{\diag}{diag}
\DeclareMathOperator{\tr}{Tr}

\DeclareMathOperator{\supp}{supp}

\DeclareMathOperator{\OO}{O}
\DeclareMathOperator{\oo}{o}

\newcommand{\del}{\partial}

\newcommand{\Tr}{\operatorname{Tr}}

\newcommand{\bma}{\bm{a}}
\newcommand{\bmb}{\bm{b}}

\renewcommand{\d}{\mathrm{d}}
\renewcommand{\Im}{{\mathrm{Im}}}
\renewcommand{\Re}{\mathrm{Re}}

\theoremstyle{plain} 
\newtheorem{theorem}{Theorem}[section]
\newtheorem*{theorem*}{Theorem}
\newtheorem{lemma}[theorem]{Lemma}
\newtheorem*{lemma*}{Lemma}
\newtheorem{corollary}[theorem]{Corollary}
\newtheorem*{corollary*}{Corollary}

\newtheorem{proposition}[theorem]{Proposition}
\newtheorem*{proposition*}{Proposition}
\newtheorem{assumption}{Assumption}
\newtheorem*{assumption*}{Assumption}

\newtheorem*{definition*}{Definition}

\newtheorem*{example*}{Example}
\newtheorem{remark}[theorem]{Remark}

\newtheorem*{remark*}{Remark}
\newtheorem*{remarks*}{Remarks}

\title{Large Deviations Asymptotics of Rectangular Spherical Integral}

  \author[1]{Alice Guionnet\thanks{aguionne@ens-lyon.fr}}
  \author[2]{Jiaoyang Huang\thanks{jh4427@nyu.edu}}

  \affil[1]{CNRS-ENS Lyon}
\affil[2]{New York University}

\date{}

\begin{document}

\maketitle

\abstract{
In this article we study the Dyson Bessel process, which describes  the evolution of singular values of rectangular matrix Brownian motions, and prove a large deviation principle for its empirical particle density. We then use it to obtain the asymptotics of the so-called rectangular spherical  integrals as $m,n$ go to infinity while $m/n$ converges.
  }

{
  \hypersetup{linkcolor=black}
  \tableofcontents
}

\newpage

\section{Introduction} In this article we shall study the asymptotics of the so-called rectangular spherical   integrals, also called Berezin-Karpelevich type  integrals in the literature.  This type of  integrals  arises when one studies rectangular matrices and is the natural counterpart of the  well known Harish-Chandra -Itzykson-Zuber (HCIZ) integral. 
The interest in  spherical integrals comes from different fields. Harish-Chandra was motivated by Fourier analysis in semi-simple Lie algebras. They appear in physics as the density in matrix models such as the Ising model  \cite{Eynard,Meh,CMM} or more generally matrix models with an external field \cite{BrHi}, including the famous Kontsevich matrix model \cite{kont}. Their uses in random matrix theory appeared more recently. First it was shown that spherical integral with  a rank one external field gives asymptotically the famous $R$-transform defined by Voiculescu in free probability \cite{GuMa05} as an analogue of Fourier transform. This approach was generalized to the rectangular-free convolution by using rectangular spherical integrals \cite{BG11} or to the multiplicative free convolution and   the $S$-transform \cite{BoPo,MePo}.

Knowing the  asymptotics of rank one spherical integrals allowed as well to investigate the large deviations for the extreme eigenvalues of random matrices. This approach was introduced in \cite{GuHu} where it was shown that the  probability that the largest eigenvalue of a Wigner matrix takes an unexpected value is the same when the entries are Rademacher or Gaussian. This universality phenomenon was shown to hold for random matrices with i.i.d. entries whose Laplace transform is bounded by the Laplace transform of a Gaussian variable with the same covariance. For more general sub-Gaussian entries, a transition appears in the rate function between large deviations towards a very large value with a heavy tail type rate function, and deviations close to the bulk which are governed by the Gaussian rate function. Such considerations were extended to unitary invariant ensembles \cite{GuMa20}, to the joint distribution of the largest eigenvalue and its eigenvector \cite{BiGu}, to sum of matrices, to  finitely many extreme eigenvalues \cite{GHmulti}. Indeed, the asymptotics of spherical integral could be extended to finite rank external fields \cite{GHmulti}. For small enough matrices, the same asymptotics were shown to extend to the case where the rank goes to infinity  more slowly than the dimension \cite{CoSn} and to full rank matrices \cite{CGM}.
  However, the limit differs when the rank of both matrices are of the same order and the matrices do not have small norms. Such a limit can as well be used to prove large deviation principles for the empirical measure of the eigenvalues of random matrices \cite{BGHldp} and more generally study the asymptotics of matrix models with an external field\cite{BrHi, MR2034487}.
  
 The formula for the asymptotics of  HCIZ integrals  was  foreseen  by Matytsin \cite{Matrig} and then proven rigorously in \cite{GZ3,GZei1add,MR2034487}. Matytsin used the description of Spherical integrals as invariant eigenfunctions of the Laplacian. The approach of \cite{GZ3} is kind of dual and based on a representation of spherical integrals as the density of a Dyson Brownian motion conditioned at time one, a representation which allows to use large deviations techniques and martingales. In this paper, we follow the same route for the rectangular case but prove a more general large deviation principle for conditioned Dyson Brownian motions. 
  In fact, the result in \cite{GZ3} relies on the matrix model, and only concerns the case $\beta=1$ or $2$ whereas we can deal in this paper with all cases $\beta\geq 1$. The extension of \cite{GZ3}  to the rectangular case is a natural step, which however posed significant difficulties for the proof of the lower bound if one uses the methods of \cite{GZ3}, due to additional singularity of the drifts. We should also mention the heuristics proposed in this setting in \cite{FoGr} following Matystin's arguments. One
   key idea of this paper is  to improve the large deviations lower bound by obtaining better criteria for the uniqueness of solutions to McKean-Vlasov equations with smooth fields inspired from 
\cite{LLX}, rather than the weaker approach developed in \cite{cabanal2003discussions}. Another novelty in this paper is a quantitative estimate for the convergence to Dyson Brownian motion with very general potential by a coupling argument, see Proposition \ref{p:lowerbound2}. Under more restricted assumptions, i.e. the limiting profile has square root behavior around the edge, such quantitative estimates for the convergence has been obtained in \cite{huang2019rigidity,adhikari2020dyson,huang2020edge} by using the characteristic method.
The quantitative estimate for the convergence allows us to efficiently control the locations of each particles and extend our result to Dyson Bessel processes which arises when one considers rectangular matrices and hence derive the limits of rectangular spherical integrals. We now state more precisely our main results.

The  rectangular spherical   integral is given by 
\begin{align}\label{e:UAVB} I_{n,m}(A_{n},B_{n})=
\iint e^{\beta n\Re[\Tr(A_n^*UB_nV^*)]}\rd U\rd V,
\end{align}
where  if $\beta=1$, $U\in \cO(n),V\in \cO(m)$ follow the Haar distribution over the orthogonal group, and $A_n\in \bR^{n\times m}, B_n\in \bR^{n\times m}$, whereas for $\beta=2$,  $U\in \cU(n),V\in \cU(m)$ follow the  Haar distribution over  the unitary group, $A_n\in \bC^{n\times m}, B_n\in \bC^{n\times m}$ for $\beta=2$. We call such integrals rectangular spherical integrals and shall study their asymptotic behavior when $m$ and $n$ go to infinity so that 
 the ratio $m/n$ converges towards some 
 $ 1+\alpha\in[1,\infty)$. This type of spherical integral  arises when one studies rectangular matrices and is the natural counterpart of the  well known Harish-Chandra -Itzykson-Zuber (HCIZ) integral defined  when $\beta=2$ and for two self-adjoint matrices $A_{n},B_{n}\in \bC^{n\times n}$  by
 $$I_{n}(A_{n},B_{n})=\int e^{n\Tr(A_{n}UB_{n}U^{*})}\rd U\,,$$
 where $U$ follows the Haar distribution over the unitary group.
 This integral was shown by Harish-Chandra \cite{HC57} and then Itzykson and Zuber \cite{Itzyksonzuber}
 to be equal to a determinant:
 \begin{equation}\label{tg}I_{n}(A_{n},B_{n})=c_{n}\frac{\det\left[e^{na_{i}b_{j}}\right]_{1\leq i,j\leq n}}{\Delta(\bma)\Delta(\bmb)},\end{equation}
 where $\bma=(a_1, a_2,\cdots, a_n), \bmb=(b_1, b_2,\cdots, b_n)$ are eigenvalues of $A_n$ and $B_n$ respectively, and $\Delta(\bma)=\prod_{i<j}(a_{i}-a_{j}), \Delta(\bmb)=\prod_{i<j}(b_{i}-b_{j})$ are Vandermonde determinants. 
 In 2003, Schlittgen and Wettig 
\cite{schlittgen2003generalizations} considered a generalization of the above rectangular spherical integral  given by 
\begin{align}\label{e:inegral0}
\iint \det[UV]^\nu e^{\tau\Tr(A_n^*UB_nV^*+VD_n^*U^*C_n)/2}\rd U\rd V,
\end{align}
where $U,V\in \cU(n)$ are $n\times n$ unitary matrices following Haar distribution, $A_n,B_n,C_n,D_n$ are deterministic $n\times n$ matrices, and $\nu$ is a non-negative integer. They showed that the generalization of the above integral to the case of unequal dimensions of $U,V$ leads to an integral which can be nonzero only if $\nu=0$, and predicted the following formula: for $m\geq n$
\begin{align}\label{e:integral1}
\iint e^{\tau \Tr(A_n^*UB_nV^*+VD_n^*U^*C_n)/2}\rd U\rd V
=\frac{\tau^{n(m-1)}\prod_{i=1}^n(m-i)!(n-i)!}{\Delta(\bmx^2)\Delta(\bmy^2)\prod_{i=1}^n (x_iy_i)^{m-n}}\det[I_{m-n}(2\tau x_iy_j)]_{1\leq i,j\leq n},
\end{align}
where $U\in \cU(n)$ is an $n\times n$ unitary matrix, $V\in \cU(m)$ is an $m\times m$ unitary matrix, both follow the Haar distribution, $ B_n, C_n$ are deterministic $n\times m$ matrices, and 
$A_n,D_n$ are deterministic $m\times n$ rectangular matrices,  $I_{m-n}(x)$ is the Bessel function
$$I_\kappa(2y)=y^{\kappa}\sum_{k=0}^\infty \frac{y^{2k}}{k!(k+\kappa)!}\,,$$
and $\bmx^2=(x_1^2, x_2^2,\cdots, x_n^2)$, $\bmy^2=(y_1^2, y_2^2,\cdots, y_n^2)$ are eigenvalues of the matrices $A_nC_n^*$, $B_nD_n^*$. This formula was proven in 
\cite{ghaderipoor2008generalization}.
We get the rectangular spherical integral \eqref{e:UAVB} from \eqref{e:integral1} by taking $A_n=C_n$ and $B_n=D_n$. Such formulas can be obtained by using
 the character expansion method.  Another approach is based on heat flows \cite{McS,BrHi}. Indeed, one can notice that Fourier functions $X\rightarrow e^{i\tr(AX)}$ 
 are the eigenfunctions of the Laplacian for any matrix $A$. Looking for eigenfunctions depending only on the eigenvalues of $X$ one gets the spherical integral $I_{n}(A_{n},X_{n})$, which in turns has to be an eigenfunction of the Laplace operator restricted to functions invariant under conjugation, namely the Dyson Laplace operator $L=-\Delta(X)^{-1}\sum_{i }\delta_{x_{i}}^{2}\Delta(X)$. Note however that \eqref{tg} and \eqref{e:integral1}  are not useful to derive asymptotics as they are given in terms of  a signed sum of diverging terms.

For a rectangular $n\times m$  matrix $A_{n}$ , $m\geq n$, with non trivial singular values  $(s_{i})_{{1\le i\le n}}$, we denote $\hat\nu_{A}^{n}$ its symmetrized empirical singular values 
$$\hat\nu^{n}_{A}=\frac{1}{2n}\sum_{i=1}^{n }\left(\delta_{s_{i}}+\delta_{-s_{i}}\right)\,.$$
We denote by $\Sigma$ the non commutative entropy
$$\Sigma(\nu)=\int \log |x-y|d\nu(x) d\nu(y)\,.$$
Then, we prove the following asymptotics for the rectangular spherical integrals:

\begin{theorem} 
\label{main1}
Let $A_n, B_n\in \bR^{n\times m}$ and  $U\in \cO(n),V\in \cO(m)$ following Haar distribution over orthogonal group for $\beta=1$;  $A_n, B_n\in \bC^{n\times m}$ and $U\in \cU(n),V\in \cU(m)$ following Haar distribution over unitary group, for $\beta=2$,  where $m\geq n$ and {{$m/n\rightarrow 1+\alpha, \alpha \ge 0$}}. We assume that  the symmetrized empirical singular values  $\hat\nu_A^n$ and $\hat \nu_B^n$ of $A_n$ and $B_n$ converge weakly  to $\hat\nu_A$ and $\hat\nu_B$ respectively. We moreover assume that for $C=A$ or $B$, we have  $\sup_{n}\hat\nu_{C}^{n}(x^{2})<\infty$, $\Sigma(\hat\nu_{C})>-\infty$ and, if $\alpha\neq 0$,
$
\int \ln |x|d\hat \nu_{C}>-\infty$. 
Then, the  following limit of the rectangular spherical integral exists
$$
\lim_{n}\frac{1}{n^{2}}\log I_{n,m}(A_n,B_n)=\frac{\beta}{2} I^{\alpha}(\hat \nu_{A},\hat\mu_{B}),\quad I_{n,m}(A_n,B_n)=\int e^{\beta  n\Re[\Tr(A_n^*UB_nV^*)]}\rd U\rd V.
$$
It is given explicitly by
\begin{align}\begin{split}\label{e:arate}
&I^{\alpha}(\hat\nu_{A},\hat\nu_{B})=-\inf_{\{\hat\rho_t\}_{0\leq t\leq1} }\left\{
\int_0^1 \int u_s^2\hat \rho_s \rd x\rd s+\frac{\pi^2}{3}\int_0^1\int  \hat \rho^3_s \rd x\rd s
+\frac{\alpha^2}{4}\int \frac{\hat \rho_s(x)}{x^2}\rd x\rd s\right\}\\
&+(\hat\nu_{A}(x^{2}-\al\log |x|)+\hat\nu_{B}(x^{2}-\al\log |x|))-(\Sigma(\hat\nu_A)+\Sigma(\hat\nu_B)) +\const,
\end{split}\end{align}
where $\const$ is a constant depending on {{$\al$}}.
The infimum is taken over  continuous symmetric measure valued processes  $(\hat\rho_{t}(x)\rd x)_{{0<t<1}}$ such that
\begin{align}\label{e:bbterm}
\lim_{t\rightarrow 0}\hat\rho_t(x)\rd x=\hat\nu_A,\quad
\lim_{t\rightarrow 1}\hat\rho_t(x)\rd x=\hat\nu_B.
\end{align}
Moreover, 
$u$ is  the weak solution of the following conservation of mass equation
\begin{align*}
\del_s\hat \rho_s+\del_x(\hat \rho_s u_s)=0.
\end{align*}
\end{theorem}
This theorem will be proved in Section \ref{sph-sec}. We show in Proposition \ref{p:uniquelimit} that in fact the non commutative law of $(A_n,UB_nV^{*})$ converges when$(U,V)$ follows  the Gibbs measure with free energy $ I_{n,m}(A_n,B_n)$. 
As in \cite{GZ3}, the main point is to derive a large deviation principle for the associated processes, namely Bessel Dyson processes. Indeed, let $G_{n}$ be  an $n\times m$ rectangular matrix with independent real ($\beta=1$) or complex ($\beta=2$) Gaussian entries and set

\begin{align*}
X_n=A_n+\frac{1}{\sqrt n}G_n.
\end{align*} then, we claim that the large deviation principle for the symmetrized empirical singular values 
of $X_{n}$ gives the asymptotics of spherical integrals. In fact,
denote the singular value decomposition of $X_n$ as $X_n=UB_nV^*$. Then the joint  law of $(B_n,U,V)$ is given by 
\begin{align}\begin{split}\label{e:lawXX0}
\frac{1}{Z_{n,m}}\prod_i b_i^{\beta(m-n+1)-1}
\prod_{i<j}|b^2_i-b^2_j|^\beta e^{-\frac{\beta n}{2}(\sum_i b^2_i+\sum a^2_i)+\beta n\Re[\Tr(A_n^*UB_nV^*)]}\rd U\rd V\rd B_n.
\end{split}\end{align}
Assume that we have proven a   large deviation principle  for $\hat\nu^{n}_{X}$ with a good rate function $I_{\hat\nu_{A}}$ so that for any symmetric probability measure $\hat\nu_{B}$

\begin{align}\label{conv10}
\lim_{n\rightarrow \infty}\frac{1}{n^2}\log \bP(\hat\nu^n_{B} \in \bB(\hat\nu_B,\delta))
= -I_{\hat\nu_{A}}(\hat\nu_{B})+o_{\delta(1)}
\end{align}
{{where $\oo_\delta(1)$ goes to zero as $\delta$ goes to zero.}}
By integrating \eqref{e:lawXX0} over the ball $\bB(\hat\nu_B,\delta)$, we have
\begin{align*}
&\phantom{{}={}}\int_{\hat\nu^n_{B}\in \bB(\hat\nu_B,\delta)}\frac{1}{Z_{n,m}}\prod_i b_i^{\beta(m-n+1)-1}
\prod_{i<j}|b^2_i-b^2_j|^\beta e^{-\frac{\beta n}{2}(\sum_i b^2_i+\sum a^2_i)+\beta n\Re[\Tr(A^*UBV^*)]}\rd U\rd V\rd B_n\\
&=\frac{1}{Z_{n,m}} 
e^{\frac{\beta n^2}{2} (2  \al \int \log |x|\rd \hat\nu_B+2\Sigma(\hat\nu_B)-(\hat\nu_A( x^2) +\hat\nu_B(x^2))+\oo_\delta(1))}
\int_{\hat\nu^n_{B}\in \bB(\hat\nu_B,\delta)} \int e^{\beta n\Re[\Tr(A^*UBV^*)]}\rd U\rd V\rd B_n.
\end{align*}

By rearranging, we obtain the following asymptotics of the spherical integral (following the standard arguments to prove large deviations for Beta-ensembles \cite{BAG}):
\begin{align*}\begin{split}
&\lim_{n\rightarrow \infty}\frac{1}{n^{2}}\log I_{ n,m}(A_n,B_n)
=-I_{\hat\nu_{A}}(\hat\nu_{B})\\
& -\frac{\beta}{2}\left(2\al\int \log x \rd\hat\nu_{B}(x)+2\Sigma(\hat\nu_B)-(\hat\nu_{A}(x^{2})+\hat\nu_{B}(x^{2})) \right)+\const.
\end{split}\end{align*}
To prove \eqref{conv10}, we see $X_{n}=H(1)$ as the matrix valued process $H(t)=A+G_{n}(t)/\sqrt n$ at time one, where $G_{n}(t)$ is field with independent Brownian motions. The singular values $s_1(t)\geq s_2(t)\geq \cdots \geq s_{n-1}(t)\geq |s_n(t)|$  of $H(t)$ follow the Dyson Bessel  process:
\begin{align}\begin{split}\label{e:DBPcopy}
\d s_i(t) 
&=\frac{\d W_{i}}{\sqrt{\beta n}}+\left(\frac{1}{2n}\sum_{j: j \neq i}\frac{1}{s_i(t)-s_j(t)}+\frac{1}{2n}\sum_{j: j\neq i}\frac{1}{s_i(t)+s_j(t)}+\frac{\al_n}{2 s_i(t)}\right)\d t,\quad 1\leq i\leq n,
\end{split}\end{align}
where $W_1, W_2,\cdots, W_n$ are independent Brownian motions and 
\begin{align*}
\alpha_{n}=\frac{m-n}{n}+(1-\frac{1}{\beta})\frac{1}{n}.
\end{align*} 
We denote the  empirical particle density of \eqref{e:DBPcopy} and its symmetrized version, which is also the symmetrized empirical singular values of $H(t)$, as
\begin{align*}
\nu_t^n=\frac{1}{n}\sum_{i=1}^n \delta_{s_i(t)}, \quad \hat\nu_t^n=\frac{1}{2n}\sum_{i=1}^n (\delta_{s_i(t)}+\delta_{-s_i(t)}),
\end{align*}
We
prove a large deviation principle for $\{\hat\nu_{t}^{n}\}_{0\leq t\leq 1}$, in Section \ref{secbessel}. The rate function is given by
\begin{align}
\label{e:ratt}
S^\al_{\hat\mu_{0}}(\{\hat\nu_t\}_{0\leq t\leq 1})=\sup_{f\in \cC^{2,1}_b} S^\al(\{\hat\nu_t,f_t\}_{0\leq t\leq 1}),
\end{align}
where
\begin{align*}
S^\al(\{\hat\nu_t,f_t\}_{0\leq t\leq 1})
&=
\hat\nu_1(f_1)-\hat\mu_0(f_0)-\int_0^1\int \del_s f_s(x) \rd\hat\nu_s(x)\rd s-\frac{1}{2}\int_0^1\int \frac{f_s'(x)-f_s'(y)}{x-y}\rd \hat\nu_s(x)\rd \hat\nu_s(y)\rd s \\
& -\frac{\alpha}{2}\int_0^1\int \frac{f_s'(x)}{x} \rd \hat\nu_s(x)\rd s  -\frac{1}{8\beta }\int^1_0 \int (f_s'(x)-f_s'(-x))^2\rd, \hat\nu_s(x)\rd s.
\end{align*}
If $\hat\nu_0\neq \hat\mu_0$,  $S^\al_{\hat\mu_{0}}(\{\hat\nu_t\}_{0\leq t\leq 1})=\infty$.
We then prove the following result 

\begin{theorem}\label{main2}
Fix a  symmetric probability measure  $\hat\mu_0$  and an initial condition with symmetrized empirical measure $\hat\nu^n_0$ with uniformly bounded second moment converging weakly to $\hat\mu_{0}$.  Then,
if $\alpha_n$ converges towards $\alpha\in [0,\infty)$ when $n$ goes to infinity {{so that either $\alpha_n\ge 1/\beta n$ or $\alpha_n\equiv 0$}},   the distribution of   the empirical particle density $\{\hat\nu_t^n\}_{0\leq t\leq 1}$ of the Dyson Bessel process \eqref{e:density} satisfies a   large deviations principle in the scale $n^2$ and with good rate function $S_{\hat\nu_0}^\al$. In particular, 
for any continuous symmetric measure-valued  process $\{\hat\nu_t\}_{0\leq t\leq 1}$, we have: 
\begin{align}\begin{split}\label{e:ulbb}
&\phantom{{}={}}\lim_{\delta\rightarrow 0}\liminf_{n\rightarrow\infty}\frac{1}{n^2}\log \bP(\{\hat\nu^n_t\}_{0\leq t\leq 1}\in \bB(\{\hat\nu_t\}_{0\leq t\leq 1}, \delta))\\
&=\lim_{\delta\rightarrow 0}\limsup_{n\rightarrow\infty}\frac{1}{n^2}\log \bP(\{\hat\nu^n_t\}_{0\leq t\leq 1}\in \bB(\{\hat\nu_t\}_{0\leq t\leq 1}, \delta))
= -{{S^\al_{\hat\mu_0}(\{\hat\nu_t\}_{0\leq t\leq 1}).}}
\end{split}\end{align}
\end{theorem}
\begin{remark}
In Theorem \ref{main2}, we assumed that either $\al_n\geq 1/\beta n$ or $\al_n\equiv 0$. This assumption is always true for $\beta=1,2$ and $m\geq n$. If this condition is violated, i.e. $0<\al_n<1/\beta n$, the particles $s_n(t)$ and $s_{-n}(t)$ in \eqref{e:DBPcopy} may collapse at $0$. In this case, to make sense of \eqref{e:DBPcopy}, we need to specify the boundary condition when they collapse at $0$. We will not discuss these conditions in this paper. 
\end{remark}

As a consequence, we deduce from the contraction principle \cite{De-Ze} that \eqref{conv1} holds and more precisely
\begin{corollary}
For any  symmetric probability measures  $\hat\nu^{n}_{A},\hat\nu^{n}_{B}$  with uniformly bounded  second moment converging weakly towards $\hat\nu_{A},\hat\nu_{B}$, under the measure \eqref{e:lawXX0} we have
\begin{align}\label{conv1}
\lim_{n\rightarrow \infty}\frac{1}{n^2}\log \bP(\hat\nu^n_{B} \in \bB(\hat\nu_B,\delta))
= -I_{\hat\nu_{A}}(\hat\nu_{B})+o_{\delta(1)},
\end{align}
where
$$I_{\hat\nu_{A}}(\hat\nu_{B})=\inf_{\hat\nu_{1}=\hat\nu_{B}}S_{\hat\nu_A}^\al(\{\hat\nu_{t}\}_{0\leq t\leq 1})\,.$$

\end{corollary}

Theorem \ref{main1} is deduced from Theorem \ref{main2} in section \ref{sph-sec}. The main difficulty to prove Theorem \ref{main2} lies in the singularity of the potential at the origin and the repulsion between the particles. To prove it, we revisit  in section \ref{sec-DBM} the large deviation principle for the empirical measure of the Dyson Brownian motion  of \cite{GZ3} and extend it to to all values of $\beta$ greater or equal to one. 

{
\noindent\textbf{Acknowledgements }
The research of J.H. is supported by the Simons Foundation as a Junior Fellow at
the Simons Society of Fellows, and NSF grant DMS-2054835. The work  of A. Guionnet is partly supported  by ERC Project LDRAM : ERC-2019-ADG Project 884584.  We thank O. Zeitouni for many inspiring discussions about spherical integrals, including preliminary ideas about the questions addressed in this article.
}

\noindent{\bf{Notations}}
$\cO(n)$ denotes  the orthogonal group in dimension $n$ and 
$\cU(n)$ the unitary group in dimension $n$. We denote by 
$d(\cdot, \cdot)$ the $2$-Wasserstein distance  defined on the space $\mathcal P_2(\mathbb R)$ of probability measures with finite second moment by 
$$d(\mu,\nu)=\inf\left\{ \int |x-y|^2 d\pi(x,y)\right\}^{1/2},$$
where the infimum is taken over distribution on $\mathbb R^2$ with marginal distribution $\mu$ and $\nu$. 
$\cC_b^{2,1}(\mathbb R\times [0,1])$ is the space of functions on $\mathbb R\times [0,1]$ with  bounded first two  derivatives in $x$ and bounded derivative in $t$.
$\cC([0,1], \bM_1(\bR))$ is the space of continuous (with respect to weak topology) measure valued process.

\section{Dyson Bessel Process}
In this section we introduce Dyson Bessel process, which is the singular value process of rectangular matrix brownian motions. Then in section \ref{s:changeM}, we will write Dyson Bessel process as a change of measure from Dyson Brownian motion using Girsanov's theorem.

\subsection{Decomposition}
The rectangular spherical integral \eqref{e:UAVB} is related to real ($\beta=1$) and complex ($\beta=2$) rectangular random matrices with nonzero mean. 
We consider an $n\times m$ rectangular random matrices $X_n$ with nonzero mean,
\begin{align}\label{e:defX}
X_n=A_n+\frac{1}{\sqrt n}G_n,
\end{align}
where $A_n=\bE[X_n]$ is deterministic, and $G_n$ is an $n\times m$ rectangular matrix with independent real ($\beta=1$) or complex ($\beta=2$) Gaussian entries.
We denote the singular value decomposition of $X_n=UB_nV^*$, with $B_n=\diag\{b_1,b_2,\cdots, b_n\}$. Then we can rewrite the law of $X_n$ as 
\begin{align}\begin{split}\label{e:lawU}
&\phantom{{}={}}\left(\sqrt{\frac{\beta n}{2\pi}}\right)^{\beta mn}e^{-\frac{\beta n}{2}\Tr((X_n-A_n)(X_n-A_n)^*)}\rd X_n\\
&\propto\prod_i b_i^{\beta(m-n+1)-1}
\prod_{i<j}|b^2_i-b^2_j|^\beta e^{-\frac{\beta N}{2}(\sum_i b^2_i+\sum a^2_i)+\beta n\Re[\Tr(A_n^*UB_nV^*)]}\rd U\rd V\rd B_n.
\end{split}\end{align}
Therefore, conditioning on the singular values of $X_n$, i.e. the matrix $B_n$, the joint law of singular vectors of $X_n$, i.e.  $U,V$ is given by the integrand of the rectangular spherical integral \eqref{e:UAVB}
\begin{align}\label{e:lawU2}
\frac{e^{\beta  N\Re[\Tr(A_n^*UB_nV^*)]}}{Z^\beta_{m,n}}\rd U\rd V.
\end{align}

We study the random matrices $X_n$ as in \eqref{e:defX} via a dynamical approach. By constructing a matrix valued real/complex Brownian motions starting from $A_n$, its value at time $t=1$ has the same law as $X_n$. 
\begin{theorem}[Dyson Bessel Process]  Take $\beta\ge 1$. Fix  $m\geq n$, and let $H(t)$ be a $n\times  m$ matrix with entries given by independent  real/complex Brownian motions starting from $A_n$: 
\begin{align}\label{e:defBB}
H(t)=A_n+\frac{1}{\sqrt{n}}G(t),
\end{align} 
The singular values $s_1(t)\geq s_2(t)\geq \cdots \geq s_{n-1}(t)\geq |s_n(t)|$ of $H(t)$ satisfies the following stochastic differential equations
\begin{align}\begin{split}\label{e:dsk}
\d s_i(t) 
&=\frac{\d W_{i}}{\sqrt{\beta n}}+\left(\frac{1}{2n}\sum_{j: j \neq i}\frac{1}{s_i(t)-s_j(t)}+\frac{1}{2n}\sum_{j: j\neq i}\frac{1}{s_i(t)+s_j(t)}+\frac{\al_n}{2 s_i(t)}\right)\d t,\quad 1\leq i\leq n,
\end{split}\end{align}
where 
\begin{align*}
 \alpha_n=\frac{m-n}{n}+\left(1-\frac{1}{\beta}\right)\frac{1}{n},
\end{align*}
and $W_1,W_2,\cdots, W_n$ are independent Brownian motions. We denote by $\mathbb P$ the law of ${\bf s}(t)=(s_1(t),\ldots, s_{n}(t)), 0\le t\le 1$. 
\end{theorem}

The eigenvalues process of $\la_1(t)\geq \la_2(t)\geq \cdots\geq \la_n(t)$ of $H(t)H^*(t)$ has been intensively studied in the literature \cite{bru1991wishart,bru1989diffusions, demni2007laguerre,demni2009radial,OCKo}, called the $\beta$-Laguerre process or $\beta$-Wishart process
\begin{align}\label{e:DLWM}
\d \lambda_i(t)=&2\sqrt{\lambda_i}\frac{\d W_{i}(t)}{\sqrt{\beta n}}+\left(\frac{1}{n}\sum_{j:j\neq i}\frac{\lambda_i(t)+\lambda_j(t)}{\lambda_i(t)-\lambda_j(t)}+\frac{m}{n}\right)\d t, \quad 1\leq i\leq n,
\end{align} 
where $W_1,W_2,\cdots, W_n$ are independent Brownian motions. 
In   \cite{OCKo},  the case $\beta=2$ Laguerre  process was shown to correspond to squared Bessel processes conditioned never to collide in the sense of Doob.
It is known that for $\beta\geq 1$ and $m\geq n$, \eqref{e:DLWM} has a unique strong solution satisfying $\la_1(t)> \la_2(t)> \cdots >\la_n(t)\geq 0$ for $t>0$. Then a formal calculation gives that $s_i(t)=\sqrt{\la_i(t)}$ satisfies \eqref{e:dsk}. When $n=1$, $s_1$ is a Bessel process. We call the process \eqref{e:dsk} Dyson Bessel process.
The same argument as in \cite[Lemma 4.3.3]{AGZ}, we can show that for $\beta\geq 1$, $\alpha_n\geq 1/\beta n$, and any initial condition $s_1(0)\geq s_2(0)\geq \cdots\geq s_n(0)\geq 0$, the unique strong solution of \eqref{e:dsk} satisfy $s_1(t)> s_2(t)> \cdots s_n(t)> 0$ for $t>0$. Therefore, $s_1(t)> s_2(t)> \cdots > s_{n-1}(t)>s_n(t)>0$ has the same law of singular values of $H(t)$. We notice that $\alpha_n\geq 1/\beta n$ is satisfied for any $m\geq n$ and $\beta=2$. In the special case that $\beta=1$, $m=n$ and $\al_n=0$, as discussed in \cite[Appendix 1]{che2019universality}, $s_n(t)$ can be negative, and $s_1(t)> s_2(t)> \cdots > s_{n-1}(t)>|s_n(t)|>0$ has the same law of singular values of $H(t)$. For our study of Dyson Bessel process, we restrict ourselves to these two choices of parameters

\subsection{Change of Measure}\label{s:changeM}
In this section, we relate the Dyson Bessel process \eqref{e:dsk} with the Dyson Brownian motion by a change of measure using Girsanov's theorem. We recall the Dyson Brownian motion (DBM) is given for $\beta\geq 1$ by
\begin{align}\begin{split}\label{e:DBM}
\d  x_i(t) 
&=\frac{\d W_{i}(t)}{\sqrt{\beta n}}+\frac{1}{2n}\sum_{j: j \neq i}\frac{\d t}{ x_i(t)- x_j(t)}.
\end{split}\end{align}
We denote the law of Dyson  Brownian motion \eqref{e:DBM} as $\bQ$. 

The Dyson Bessel process \eqref{e:dsk} can be obtained from the DBM \eqref{e:DBM} by a change of measure using an exponential martingale constructed from the following function 
\begin{align}\label{e:theta}
\theta( s_1,  s_2,\cdots,  s_n)=\frac{\beta}{2}\left(\sum_{i<j} \log( s_i+ s_j)+\alpha_n n\sum \log s_k\right),\quad s_1\geq s_2\geq \cdots \geq s_n.
\end{align} 
The above function $\theta$ has logarithmic singularity when $s_n$ is close to $0$. Fix a small parameter $\fa>0$,  we define the stopping time $\tau_\fa$, the first time that $s_n(t)$ gets too close to $0$,
 \begin{align}\label{e:taue}
 \tau_{\fa}=\inf\{t\geq 0: s_n\leq \fa\}.
 \end{align}
Then for $t\leq \tau_\fa$, we have $s_n(t)\geq \fa$, and 
$\theta(s_1(t), s_2(t), \cdots, s_n(t))$ is bounded below uniformly.

\begin{proposition}\label{p:changem}
Let $\cF_{t}$ be the $\sigma$ algebra generated by the Brownian motions $\{W_i(t)\}$. 
We take $\bQ$ the law of DBM 
\begin{align}\label{e:DBMa}
\rd  x_i(t)=\frac{\rd W_i(t)}{\sqrt{\beta n}}+\frac{1}{2n}\sum_{j:j\neq i}\frac{\rd t}{ x_i(t)- x_j(t)},\quad 1\leq i\leq n,
\end{align}
and $\bP^\fa$ the law of the following modified Dyson Bessel process
\begin{align*}
\d s_i(t) =\frac{\d W_{i}(t)}{\sqrt{\beta n}}+\frac{1}{2n}\sum_{j:j\neq i}\frac{1}{s_i(t)-s_j(t)}+\bm1(t\leq \tau_\fa)\left(\frac{1}{2n}\sum_{j:j\neq i}\frac{1}{s_i(t)+s_j(t)}+\frac{\alpha_n}{2 s_i(t)}\right)\d t,
\end{align*}
for $1\leq i\leq n$.
Then the two laws $\bP^\fa$ and $\bQ$ are related by a change of measure
\begin{align*}
\bP^\fa=e^{{L_{1\wedge \tau_\fa}-\frac{1}{2}\langle L,L\rangle_{1\wedge \tau_\fa}}} \bQ,
\end{align*}
where the exponent is given by
\begin{align*}\begin{split}
L_{t\wedge \tau_\fa}-\frac{1}{2}\langle L, L\rangle_{t\wedge \tau_\fa}
&=\left.\theta(x_1({u}), \cdots, x_n({u}))\right|_0^{t\wedge \tau_\fa}- \frac{\beta n}{2}
\int_0^{t\wedge \tau_\fa}\sum_i \frac{\alpha_n^2}{4x_i^2(u)}\rd u
\\
&-\left(\frac{\beta}{2}-1\right)\int_{0}^{t\wedge \tau_\fa}\frac{1}{4n}\sum_{k\neq \ell}\frac{\rd u}{(x_k(u)+x_\ell(u))^2}
+\int_{0}^{t\wedge \tau_\fa}\frac{\alpha_n}{4}\sum_k \frac{\rd u}{x_k^2(u)}.
\end{split}\end{align*}
\end{proposition}

\begin{remark}
We remark that $\bP^\fa$ depends on $\fa>0$. For any event  $\Omega$ of singular value DBM, we can can lower bound its probability in the following way
\begin{align*}
\bP(\Omega)\geq \bP(\Omega\cap \{s_n\geq \fa\})=\bP^\fa(\Omega\cap \{s_n\geq \fa\}).
\end{align*}
\end{remark}

\begin{proof}[Proof of Proposition \ref{p:changem}] 
The first and second derivatives of $\theta$ are given by
\begin{align}\begin{split}\label{e:dtheta}
\del_{ s_i}\theta( s_1,  s_2,\cdots,  s_n)&=\frac{\beta}{2}\left(\sum_{j:j\neq i}\frac{1}{s_i+s_j}+ \frac{\alpha_n n}{s_i}\right),\\
\del^2_{ s_i}\theta( s_1,  s_2,\cdots,  s_n)&=-\frac{\beta}{2}\left(\sum_{j:j\neq i}\frac{1}{(s_i+s_j)^2}+ \frac{\alpha_n n}{s_i^2}\right),
\end{split}\end{align}
for $1\leq i\leq n$.
Since $\theta$ is $\cC^\infty$ on sets where it is  bounded below, It{\^o}'s lemma gives that if ${\bf x}(t)=(x_{1}(t),\ldots,x_{n}(t))$, 
\begin{align}\label{e:dL}
\rd \theta({\bf x (t)})
&=\rd L_t+
\frac{1}{4n}\sum_{i\neq j}\frac{\del_{ x_i}\theta({\bf x}(t))-\del_{ x_j}\theta({\bf x}(t))}{ x_i(t)- x_j(t)}\d t
+\sum_i\frac{\del^2_{ x_i}\theta ({\bf x}(t))}{2\beta n} \d t,
\end{align}
where the martingale term $L_t$ is 
\begin{align*}
\rd L_t=\sum_{i}\del_{ x_i}\theta({\bf x}(t))\frac{\rd W_i(t)}{\sqrt{\beta n}}
=\sum_i \left(\frac{\sqrt \beta}{2\sqrt n}\sum_{j: j\neq i}\frac{1}{ x_i(t)+ x_j(t)}+\sqrt {\beta n}\frac{\alpha_n}{ 2x_i(t)}\right)\rd W_i(t),
\end{align*}
Its quadratic variance is given by
\begin{align*}
\langle L, L\rangle_t=\int^t_0 \sum_i \left(\frac{\sqrt \beta}{2\sqrt n}\sum_{j:j\neq i}\frac{1}{ x_i(u)+ x_j(u)}+\sqrt {\beta n}\frac{\alpha_n}{ 2x_i(u)}\right)^2\rd u.
\end{align*}
For the second term on the righthand side of \eqref{e:dL}, using \eqref{e:dtheta} we have
\begin{align}\begin{split}\label{e:term1}
&\phantom{{}={}}\frac{1}{4n}\sum_{i\neq j}\frac{\del_{ x_i}\theta({\bf x}(t)) -\del_{ x_j}\theta({\bf x}(t))}{ x_i- x_j}
=-\frac{\beta}{8n}\sum_{i\neq j\neq k}\frac{1}{(x_i+x_k)(x_j+x_k)}-\frac{\beta \alpha_n}{8}\sum_{i\neq j}\frac{1}{x_i x_j}\\
&=-\frac{\beta}{8n}\sum_i\left(\sum_{j:j\neq i}\frac{1}{(x_i+x_j)}\right)^2+\frac{\beta}{8n}\sum_{i\neq j}\frac{1}{(x_i+x_j)^2}-\frac{\beta \alpha_n}{8}\left(\sum_{i}\frac{1}{x_i}\right)^2+\frac{\beta \alpha_n}{8}\sum_i \frac{1}{x_i^2}.
\end{split}\end{align}
For the last term on the righthand side of \eqref{e:dL}, using \eqref{e:dtheta}we have
\begin{align}\label{e:term2}
\sum_i\frac{\del^2_{ x_i}\theta({\bf x}(t))}{2\beta n}=-\frac{1}{4n}\sum_{i\neq j}\frac{1}{(x_i+x_j)^2}-\frac{\alpha_n}{4}\sum_{i}\frac{1}{x_i^2}.
\end{align}
By plugging \eqref{e:term1} and \eqref{e:term2} back into \eqref{e:dL}, we get
\begin{align}\begin{split}\label{e:df}
\rd \theta({\bf x}(t))&=\rd L_t-\sum_i\frac{\beta}{8n}\left(\sum_{j:j\neq i}\frac{1}{(x_i(t)+x_j(t))}\right)^2-\frac{\beta \alpha_n}{8}\left(\sum_{i}\frac{1}{x_i(t)}\right)^2\\
&+\left(\frac{\beta}{2}-1\right)\left(\frac{1}{4n}\sum_{i\neq j}\frac{1}{(x_i(t)+x_j(t))^2}+\frac{\al}{4}\sum_i \frac{1}{x_i^2(t)}\right).
\end{split}\end{align}

We recall the stopping time $\tau_\fa$ from \eqref{e:taue}, then
\begin{align*}
\langle L, L\rangle_{t\wedge \tau_\fa}
&=\int^{t\wedge \tau_\fa}_0 \sum_i \left(\frac{\sqrt \beta}{2\sqrt n}\sum_{j:j\neq i}\frac{1}{ x_i(u)+ x_j(u)}+\sqrt {\beta n}\frac{\alpha_n}{2 x_i(u)}\right)^2\rd u\\
&\leq \int^{t\wedge \tau_\fa}_0 \sum_i \left(\frac{\sqrt \beta}{2\sqrt n}\sum_{j:j\neq i}\frac{1}{ 2\fa}+\sqrt {\beta n}\frac{\alpha_n}{ 2\fa}\right)^2\rd u\leq \left(\frac{\alpha_n}{2}+\frac{1}{4}\right)^2\frac{\beta n^2 {(t\wedge \tau_\fa)}}{\fa^2},
\end{align*}
which is uniformly bounded. Therefore, Novikov's theorem \cite[H.10]{AGZ} implies the following is an exponential martingale 
\begin{align*}
e^{{L_{t\wedge \tau_\fa}-\frac{1}{2}\langle L,L\rangle_{t\wedge \tau_\fa}}}.
\end{align*}
Using \eqref{e:df}, more explicitly,
we can rewrite 
\begin{align*}
L_{t\wedge \tau_\fa}-\frac{1}{2}\langle L, L\rangle_{t\wedge \tau_\fa}
&=\left.\theta(x_1({u}), \cdots, x_n({u})\right|_0^{t\wedge \tau_\fa}- \frac{\beta n}{2}
\int_0^{t\wedge \tau_\fa}\sum_i \frac{\alpha_n^2}{4x_i^2(u)}\rd u
\\
&-\left(\frac{\beta}{2}-1\right)\int^{t\wedge \tau_\fa}\frac{1}{4n}\sum_{k\neq \ell}\frac{\rd u}{(x_k(u)+x_\ell(u))^2}
+\int^{t\wedge \tau_\fa}\frac{\alpha_n}{4}\sum_k \frac{\rd u}{x_k^2(u)}.
\end{align*}
We recall that $\bQ$ is the law of DBM \eqref{e:DBMa}, and denote the rescaled Brownian motions $M$, 
\begin{align*}
M_i(t)= x_i(t)- x_i(0)-\int_0^t\frac{1}{2n}\sum_{j:j\neq i}\frac{\rd u}{ x_i(u)- x_j(u)}=\int_0^t \frac{\rd W_i(u)}{\sqrt{\beta n}}=\frac{W_i(t)}{\sqrt{\beta n}},
\end{align*}
then Girsanov's theorem \cite[Theorem H.11]{AGZ} implies that
\begin{align}\begin{split}\label{e:newM}
M_i(t)-\langle M_i, L\rangle_{t\wedge \tau_\fa}&= x_i(t)- x_i(0)-\int_0^t\frac{1}{2n}\sum_{j:j\neq i}\frac{\rd u}{ x_i(u)- x_j(u)}\\
&-\int_0^{t\wedge \tau_\fa}\left(\frac{1}{2 n}\sum_{j:j\neq i}\frac{1}{ x_i(u)+ x_j(u)}+\frac{\alpha_n}{ 2x_i(u)}\right)\rd u,
\end{split}\end{align}
are independent Brownian motions under the measure $\bP^{\fa}$:
\begin{align*}
\bP^\fa=e^{{L_{1\wedge \tau_\fa}-\frac{1}{2}\langle L,L\rangle_{1\wedge \tau_\fa}}} \bQ,
\end{align*}
Therefore,  $\bP^\fa$ is the unique solution of the stochastic differential system
\begin{align*}
\d s_i(t) =\frac{\d W_{i}(t)}{\sqrt{\beta n}}+\frac{1}{2n}\sum_{j:j\neq i}\frac{1}{s_i(t)-s_j(t)}+\bm1(t\leq \tau_\fa)\left(\frac{1}{2n}\sum_{j:j\neq i}\frac{1}{s_i(t)+s_j(t)}+\frac{\alpha_n}{2 s_i(t)}\right)\d t.
\end{align*}
where $W_1,W_2,\cdots, W_n$ are independent Brownian motions.
\end{proof}

\section{Large deviations for the Dyson Brownian motion}\label{sec-DBM}
Thanks to Proposition \ref{p:changem}, the law of singular value Dyson  Brownian motion can be rewritten as a change of measure from the Dyson  Brownian motion.{
The large deviations principle for Dyson  Brownian motion has been proven in \cite{GZ3, guionnet2004addendum} when $\beta=1$ or $2$ and the initial condition has finite $5+\epsilon$ moment for some $\epsilon>0$. 
In this section we give a shorter proof for the large deviations principle valid for any $\beta\geq 1$ and under the assumption that the initial condition has finite second  moment  only. The main technical improvement comes from Propositions \ref{p:approximation} and  \ref{p:lowerb} which allow to prove the lower bound in  greater generality, thanks to better approximation of  our processes  by processes with smooth drifts}

We denote the empirical particle density of the Dyson  Brownian motion \eqref{e:DBMa} as
\begin{align}\label{e:empd}
\nu^n_t=\frac{1}{n}\sum_{i=1}^n \delta_{x_i(t)}.
\end{align}

\begin{assumption}\label{a:mu0}
We assume the probability density $\mu_0$  has bounded  second moment. 
Moreover, as $n$ goes to infinite, $\nu_0^n$ converges to $\mu_0$ in  $2$-Wasserstein distance, i.e. $d(\mu_0, \nu_0^n)=\oo_n(1)$.
\end{assumption}

Given a continuous measure process $\{\nu_t\}_{0\leq t\leq 1}$ with $\nu_0$ satisfying Assumption \ref{a:mu0}, we define the following dynamical entropy: 
\begin{align}\begin{split}\label{e:rateD}
S(\{\nu_t,f_t\}_{0\leq t\leq 1})&=\left\{\nu_1(f_1)-\nu_0(f_0)-\int_0^1 \int\del_t f_t(x) \rd\nu_t(x)\rd t\right.\\
&\left.-\frac{1}{2}\int_0^1\int \frac{f_t'(x)-f_t'(y)}{x-y} \rd\nu_t(x)\rd \nu_t(y)\rd t
-\frac{1}{2\beta }\int^1_0 \int (f_t'(x))^2\rd \nu_t\rd t\right\},
\end{split}\end{align}
where $f_t(x)\in \cC_b^{2,1}$ has bounded twice derivative in $x$ and bounded derivative in $t$.
For any measure $\mu_0$, if $\nu_0=\mu_0$, we set
\begin{align}\label{e:rateSmu}
S_{\mu_0}(\{\nu_t\}_{0\leq t\leq 1})=\sup_{f\in \cC^{2,1}}S(\{\nu_t,f_t\}_{0\leq t\leq 1}).
\end{align}
If $\nu_0\neq \mu_0$, we set $S_{\mu_0}(\{\nu_t\}_{0\leq t\leq 1})=\infty$.
In this section we give a new proof of the following large deviations principle for the empirical particle density of the Dyson  Brownian motion \eqref{e:empd}

\begin{theorem}\label{t:DBMLDP}
Fix a probability density $\mu_0$ and an initial condition with empirical distribution $\nu_0^n$ satisfying Assumption \ref{a:mu0}.
Then, the empirical particle density $\{\nu^n_t\}_{0\leq t\leq 1}$ of the Dyson  Brownian motion \eqref{e:empd} satisfies a large deviations principle in the scale $n^2$ and with good rate function $S_{\mu_0}(\{\nu_t\}_{0\leq t\leq 1})$. In particular
for any continuous measure process $\{\nu_t\}_{0\leq t\leq 1}$, it holds
\begin{align}
&\lim_{\delta\rightarrow 0}\limsup_{n\rightarrow\infty}\frac{1}{n^2}\log \bP(\{\nu^n_t\}_{0\leq t\leq 1}\in \bB(\{\nu_t\}_{0\leq t\leq 1}, \delta))\nonumber\\
&=\lim_{\delta\rightarrow 0}\liminf_{n\rightarrow\infty}\frac{1}{n^2}\log \bP(\{\nu^n_t\}_{0\leq t\leq 1}\in \bB(\{\nu_t\}_{0\leq t\leq 1}, \delta))
= -S_{\mu_0}(\{\nu_t\}_{0\leq t\leq 1})\,.\nonumber
\end{align}
\end{theorem}

For any measure valued proces $\{\nu_{t}\}_{0\leq t\leq 1}$ such that  $S_{\mu_0}(\{\nu_t\}_{0\leq t\leq 1})<\infty$, by Riesz representation theorem, there exists a measurable function $\del_x k_t\in L^{2}(\d\nu_{t}(x)\d t)$, such that for any $f\in \cC_b^{2,1}$
\begin{align}\label{e:f1f0}
\nu_1(f_1)-\nu_0(f_0)-\int \del_t f_t(x) \rd\nu_t(x)\rd t-\frac{1}{2}\int_0^1\int f_t'(x) H(\nu_t)\rd \nu_t(x)\rd t=\int_0^1\int f'_t(x) \del_x k_t(x)\rd \nu_t \rd t.
\end{align}
Here $H(\nu)$ denotes the Hilbert transform of $\nu$. 
Then we can rewrite the rate function $S_{\mu_0}(\{\nu_t\}_{0\leq t\leq 1})$ in \eqref{e:rateD} as 
\begin{align}\label{e:minimizereq}
S_{\mu_0}(\{\nu_t\}_{0\leq t\leq 1})=\sup_{f\in \cC_b^{2,1}} \int_0^1 f'_t(x) \del_x k_t(x)\rd \nu_t \rd t-\frac{1}{2\beta}\int^1_0 \int (f_t'(x))^2\rd \nu_t\rd t=\frac{\beta}{2}\int_0^1 \int (\del_x k_t(x))^2\rd \nu_t\rd t,
\end{align}
where the equality is achieved when $f'_t(x)=\beta \del_x k_t(x)$.

We collect some properties of the rate function \eqref{e:rateD}, which were essentially proven in \cite{GZ3, MR2034487,guionnet2004addendum}.
\begin{proposition}\label{p:rate}
Fix a probability measure $\mu_0$ with finite second moment and bounded free entropy, i.e. $\Sigma(\mu_0)>-\infty$. Then,  $S_{\mu_0}$ is a good rate function on $\cC([0,1], \bM_1(\bR))$. If $S_{\mu_0}(\{\nu_t\}_{0\leq t\leq 1})<0$, then we have
\begin{enumerate}
\item There exists a constant $\fC$ depending only on $\mu_0$ and $S_{\mu_0}(\{\nu_t\}_{0\leq t\leq 1})$, such that 
the $L_2$ norms of $\nu_t$ are uniformly bounded,
\begin{align}\label{e:L2norm}
\int x^2 \rd \nu_t(x)\leq \fC.
\end{align}
\item $\nu_t$ has a density for almost surely all $0\leq t\leq 1$, i.e. 
\begin{align*}
\frac{\rd \nu_t(x)}{\rd x}=\rho_t(x).
\end{align*}
\item We denote the velocity field $u_t(x)=H(\nu_t)(x)/2+\del_x k_t(x)$,
then it satisfies the conservation of mass equation 
\begin{align}\label{e:masseq}
\del_t\rho_t+\del_x(\rho_t u_t)=0,\quad 0\leq t\leq 1,
\end{align}
in the sense of distribution.
We can rewrite the dynamical entropy \eqref{e:rateD} as
\begin{align}\begin{split}\label{Sent}
S_{\mu_0}(\{\nu_t\}_{0\leq t\leq 1})&=\frac{\beta}{2}\left(\int_0^1 \int (u_t^2+H(\nu_t)^2/4)\rho_t(x)\rd x\rd t-\frac{1}{2}(\Sigma(\nu_1)-\Sigma(\nu_0))\right)\\
&=\frac{\beta}{2}\left(\int_0^1 \int (u_t^2+\frac{\pi^2}{12}\rho_t(x)^2)\rho_t(x)\rd x\rd t-\frac{1}{2}(\Sigma(\nu_1)-\Sigma(\nu_0))\right).
\end{split}\end{align}
\end{enumerate} 
\end{proposition}
\begin{proof}

It is proven in \cite[Theorem 1.4]{GZ3} that $S_{\mu_0}(\{\nu_t\}_{0\leq t\leq 1})$ is a good rate function. 
 If  $S_{\mu_0}(\{\nu_t\}_{0\leq t\leq 1})<\infty$, by definition we have $\mu_0=\nu_0$.
 For Item (i), we take a test function 
$f_\varepsilon(x)=x^2/(1+\varepsilon x^2)$ with small $\varepsilon>0$. Then it is easy to see that $f_{\varepsilon}'(x)=2x/(1+\varepsilon x^2)^2$ and $|f_{\varepsilon}''(x)|\leq 10$. By the definition of the dynamical free entropy \eqref{e:rateD}, for any $0<t\leq 1$, we have
\begin{align}\begin{split}\label{e:bbd}
&\phantom{{}={}}\nu_t(f_{\varepsilon})-\nu_0(f_{\varepsilon})-\frac{1}{4}\int_0^t \int\frac{f_{\varepsilon}'(x)-f_{\varepsilon}'(y)}{y-x}  \rd \nu_s(y)\rd \nu_s(x)\rd s-\frac{1}{2\beta}\int_0^t\int (f_{\varepsilon}'(x))^2\rd \nu_s \rd s\\
&\leq S_{\mu_0}(\{\nu_s\}_{0\leq s\leq 1})<0.
\end{split}\end{align}
By our assumption that $\nu_0=\mu_0$ has finite second moment, it holds that $\sup_\varepsilon\nu_0(f_{\varepsilon})<\infty$. Using $|f_{\varepsilon}''(x)|\leq 10$, we find for $t\le 1$, 
\begin{align*}
\left|\frac{1}{4}\int_0^t \int\frac{f_{\varepsilon}'(x)-f_{\varepsilon}'(y)}{y-x}  \rd \nu_s(y)\rd \nu_s(x)\rd s\right|\leq 5/2.
\end{align*}
Therefore, there exists a constant $\fC$ depending only on $\mu_0$ and $S_{\mu_0}(\{\nu_t\}_{0\leq t\leq 1})$,  such that
\begin{align*}
\nu_t(f_{\varepsilon})&=\int \frac{x^2}{1+\varepsilon x^2}\rd \nu_t\leq \fC+\frac{1}{2\beta}\int_0^t\int (f_{\varepsilon}'(x))^2\rd \nu_s \rd s\\
&\leq \fC +\frac{2}{\beta} \int_0^t \int \frac{x^2}{1+\varepsilon x^2}\rd \nu_s\rd s
=\fC+\frac{2}{\beta}\int_0^t \nu_s(f_{\varepsilon})\rd s.
\end{align*}
Gr{\"o}nwall's inequality then  implies that for all $t\le 1$
\begin{align*}
\nu_t(f_{\varepsilon})\le e^{2/\beta} \fC.
\end{align*}
The claim \ref{e:L2norm} follows by sending $\varepsilon$ to $0$ { and monotone convergence theorem}.

It was proven in \cite[Theorem 2.1]{MR2034487} and \cite[Theorem 3.3]{guionnet2004addendum} that if $\mu_0=\nu_0$ has bounded $5+\varepsilon$ moments, i.e.
\begin{align}\label{e:five}
\int |x|^{5+\varepsilon}\rd \nu_0<\infty, 
\end{align}
and $\Sigma(\mu_{0}),\Sigma(\mu_{1})$ are  finite,
then Item (ii) and (iii) hold. This can be extended to the case where $\mu_{0}$ has only a finite second moment following the  arguments of the proof of \cite[Lemma 5.9]{cabanal2003discussions}. We briefly recall the main steps of  the proof. First recall that free convolution reduces the dynamical entropy
(see  \cite{cabanal2003discussions}) so  that if $\sigma_{\epsilon}$ denotes the semi-circle law with covariance $\epsilon$
\begin{align*}
S_{\mu_{0}\boxplus\sigma_\epsilon}(\{\nu_{t}\boxplus\sigma_{\epsilon}\}_{0\leq t\leq 1})\leq S_{\mu_{0}}(\{\nu_t\}_{0\leq t\leq 1}).
\end{align*}
But on the other hand, $H(\nu_{t}\boxplus \sigma_{\epsilon})$ is uniformly bounded by $1/\sqrt{\epsilon}$. Therefore if we denote by $u^{\epsilon}$ the velocity field of $\nu_t^{\epsilon}=\nu_{t}\boxplus\sigma_{\epsilon}$,

$$\int_{0}^{1}\int (u^{\epsilon}_{t})^{2}\d\nu_t^{\epsilon} \d t \le 2\int_{0}^{1}\int (u^{\epsilon}_{t}-H\nu_{t}^{\epsilon})^{2}\d\nu_{t}^{\epsilon} \d t +2\int_{0}^{1}\int (H\nu_{t}^{\epsilon})^{2}\d\nu_{t}^{\epsilon} \d t 
\le \frac{4}{\beta}S_{\mu_{0}\boxplus\sigma_\epsilon}(\{\nu_{t}^{\epsilon}\}_{0\leq t\leq 1})+\frac{2}{\epsilon}<\infty\,.$$
Hence, we can write
\begin{equation}\label{lk}S_{\mu_{0}\boxplus\sigma_\epsilon}(\{\nu_{t}^{\epsilon}\}_{0\leq t\leq 1})=\frac{\beta}{2} \int_{0}^{1}\int (u^{\epsilon}_{t})^{2}\d\nu_t^{\epsilon} \d t+\frac{\beta}{2}\int_{0}^{1}\int (H\nu^{\epsilon}_{t})^{2}\d\nu_t^{\epsilon} \d t -\beta \int_{0}^{1}\int H\nu^{\epsilon}_{t }u^{\epsilon}_{t}\d\nu_t^{\epsilon} \d t\,.\end{equation}
For the  second term we used the  well known formula (recall that $d\nu^{\epsilon}_{t}\ll \d x$)
$$\int_{0}^{1}\int (H\nu^{\epsilon}_{t})^{2}\d\nu_t^{\epsilon} =\frac{\pi^{2}}{3}\int_{0}^{1}\int (\frac{d\nu^{\epsilon}_{t}}{\d x})^{3}\d x\,.$$
Finally for the last term of \eqref{lk}, we observe  following \cite[Lemma 5.9]{cabanal2003discussions} that the continuity of $t\mapsto \nu_{t}$ implies that $t\mapsto H\nu^{\epsilon}_{t}$ is continuous (thanks to the explicit formulas for the Hilbert transform of measures freely convoluted with the semi-circle laws given by Biane \cite{Bi97}). Since it is bounded  and $u^{\epsilon}$ is in $L^{2}$, we see that we can approximate the last term by Riemann sum. Then, 
recall that by definition we have
$$\int^{t }u^{\epsilon}_{s}\frac{\d \nu_{s}^{\epsilon}}{\d x} \d s=-\int^{x} \d \nu^{\epsilon}_{t},$$ to conclude that
$$ \int_{0}^{1}\int H\nu^{\epsilon}_{t }u^{\epsilon}_{t}\d\nu_t^{\epsilon} \d t=\frac{1}{2}\int_{0}^{1}\partial_{t}\Sigma(\nu^{\epsilon}_{t}) \d t=\frac{1}{2}\left(\Sigma(\nu_{1}\boxplus\sigma_{\epsilon})-\Sigma(\mu_{0}\boxplus \sigma_{\epsilon})\right)\,.$$
Hence, \eqref{Sent} holds for $\{\nu^{\epsilon}_t\}_{0\leq t\leq 1}$. This implies that $\Sigma(\nu_{1}\boxplus\sigma_{\epsilon})$  is bounded since it  is bounded from above  as 
$\nu_{1}\boxplus\sigma_{\epsilon}$  has bounded second moment and also from below since
$$\frac{\beta}{2} \int_{0}^{1}\int (u^{\epsilon}_{t})^{2}\d\nu_t^{\epsilon} \d t+\frac{\beta}{2}\int_{0}^{1}\int (H\nu^{\epsilon}_{t})^{2}\d\nu_t^{\epsilon} \d t -
\frac{\beta }{2}(\Sigma(\nu_1\boxplus\sigma_{\epsilon})-\Sigma(\nu_0))\le
S_{\mu_0\boxplus\sigma_{\epsilon}}(\{\nu_t^{\epsilon}\}_{0\leq t\leq 1})\le S_{\mu_0}(\{\nu_t\}_{0\leq t\leq 1})\,.$$
In fact, because we could have done the same reasoning on the time interval $[0,t]$, we also see that
 for all $s\le 1$
$$S_{\mu_0}(\{\nu_t\}_{0\leq t\leq 1})\ge \frac{\beta}{2} \int_{0}^{s}\int (u^{\epsilon}_{t})^{2}\d\nu_t^{\epsilon} \d t+\frac{\beta}{2}\int_{0}^{s}\int (H\nu^{\epsilon}_{t})^{2}\d\nu_t^{\epsilon} \d t -
\frac{\beta }{2}(\Sigma(\nu_s\boxplus\sigma_{\epsilon})-\Sigma(\nu_0)),$$
which implies that  $\Sigma(\nu_s\boxplus\sigma_{\epsilon})$ is  uniformly bounded.  We can finally let $\epsilon$ going to zero to conclude.  As a consequence of \eqref{e:L2norm}, we deduce that $\nu_t$ has finite free entropy, i.e. $\Sigma(\nu_t)<+\infty$. We refer the reader to  \cite{cabanal2003discussions} for details.

\end{proof}

\subsection{Large deviations upper bound}
In this section, we prove the large deviations upper bound. We recall that the exponential tightness was already proven in this setting in the proof of \cite[Theorem 2.4]{GZ3}: for the sake of completeness we will recall this proof but in the new setting of the Bessel Dyson processes, see section \ref{secbessel}. We next  prove the large deviations upper bound of Theorem \ref{t:DBMLDP}
\begin{align}\label{e:DBMupbb}
&\limsup_{\delta\rightarrow 0}\limsup_{n\rightarrow\infty}\frac{1}{n^2}\log \bP(\{\nu^n_t\}_{0\leq t\leq 1}\in \bB(\{\nu_t\}_{0\leq t\leq 1}, \delta))
\leq -S_{\mu_0}(\{\nu_t\}_{0\leq t\leq 1}).
\end{align}

Take any test function $f_t(x)\in \cC_b^{2,1}([0,1]\times \bR)$, and use It{\^o}'s lemma to find that
\begin{align}\begin{split}\label{e:lin0}
&\rd\sum_{i}f_t(x_i(t))
=\sum_{i}f'_t(x_i(t))\rd x_i(t)+\sum_{i}\left( \frac{f''_t(x_i(t))}{2\beta n}+\del_t f_t(x_i(t))\right)\rd t\\&=
\sum_{i}\left( \frac{f''_t(x_i(t))}{2\beta n}+\del_t f_t(x_i(t))
+
\frac{f'_t(x_i(t))}{2n}\sum_{j: j\neq i}\frac{1}{x_i(t)-x_j(t)}\right)\d t+\sum_{i}f'_t(x_i(t))\frac{\rd W_i(t)}{\sqrt {\beta n}}
\\
&=\rd L^f_t+
\frac{1}{4n}\sum_{i \neq j}\frac{f'_t(x_i(t))-f'_t(x_j(t))}{x_i(t)-x_j(t)}\d t
+\sum_i\frac{f''_t(x_i(t))}{2\beta n}\rd t+\sum_i\del_t f_t(x_i(t))\rd t,
\end{split}\end{align}
where the martingale term is given by
\begin{align}\label{e:MLt0}
\rd L^f_t=\sum_{i}f'_t(x_i(t))\frac{\rd W_i(t)}{\sqrt {\beta n}},\quad
\langle L^f, L^f\rangle_t=\frac{1}{\beta n}\int^t_0 \sum_i(f_t'(x_i(t)))^2\rd t.
\end{align}
We recall the empirical particle density $\{\nu^n_t\}_{0\leq t\leq 1}$ from \eqref{e:empd}. With it, we can rewrite \eqref{e:lin0} as
\begin{align}\begin{split}\label{e:df220}
&\phantom{{}={}}\int f_t(x)\rd \nu^n_t-\int f_t(x)\rd \nu^n_0\\
&=\frac{ L^f_t}{n}
+\frac{1}{4}\int_0^t \int_{x\neq y} \frac{f_s'(x)-f_s'(y)}{x-y}\rd \nu^n_s(x)\rd\nu^n_s(y)\rd s
+ \frac{1}{2\beta n}\int_0^t \int f''_s(x)\rd \nu^n_s(x)\rd s+\int_0^t \int \del_s f_s(x) \rd\nu^n_s(x)\rd s\\
&=\frac{ L^f_t}{n}
+\int_0^t\left( \frac{1}{4}\int \frac{f_s'(x)-f_s'(y)}{x-y}\rd \nu^n_s(x)\rd\nu^n_s(y)
+\int \del_s f_s(x) \rd\nu^n_s(x)+\frac{1}{n}\left(\frac{1}{2\beta }-\frac{1}{4}\right)\int f''_t(x)\rd \nu^n_s(x)\right)\rd s.
\end{split}\end{align}
As $L^{f}$ is bounded uniformly for $f\in C^{2,1}_{b}$, we can construct an exponential martingale using the martingale $\rd L^f_t$ from \eqref{e:MLt0}
\begin{align}\label{e:Dt0}
D_t=e^{n L^f_t-\frac{n^2}{2}\langle L^f,L^f\rangle_t}, \quad \bE[D_t]=\bE[D_0]=1.
\end{align}
Using \eqref{e:df220}
we can rewrite 
$$nL^f_t-\frac{n^2}{2}\langle L^f, L^f\rangle_t
= n^2 S^n_t(\{\nu_t^n, f_t\}_{0\leq t\leq 1}), $$
where 
\begin{align*}
S^n_t(\{\nu^n_s, f_s\}_{0\leq s\leq t})
&=\int f_t\rd \nu^n_t-\int f_0\rd \nu^n_0-\frac{1}{4}\int_0^t\int \frac{f_s'(x)-f_s'(y)}{x-y}\rd \nu^n_s(x)\rd\nu^n_s(y)\rd s
-\int_0^t \int \del_s f_s(x) \rd\nu^n_s(x)\rd s\\
&-\int_0^t\int 
\frac{1}{n}\left(\frac{1}{2\beta }+\frac{1}{4}\right) f''_t(x)\rd \nu^n_s\rd s
-\frac{1}{2\beta }\int^t_0 \int (f_t'(x))^2\rd \nu^n_s\rd s.
\end{align*}
We also define
\begin{align}\label{e:defSn}
S^n(\{\nu_t^n, f_t\}_{0\leq t\leq 1})
&=\frac{1}{n^2}\left(\frac{n L^f_1}{2}-\frac{n^2}{2}\langle L^f, L^f\rangle_1\right)=S^n_1(\{\nu_t^n, f_t\}_{0\leq t\leq 1}) .
\end{align}
Then for $\{\nu^n_t\}_{0\leq t\leq 1}\in \bB(\{\nu_t\}_{0\leq t\leq 1}, \delta)$, we have by uniform (in $n\ge 1$) continuity of $\nu\mapsto S^n(\{\nu_t, f_t\}_{0\leq t\leq 1})$ for any $f\in \cC_b^{2,1}([0,1]\times \bR)$, 

\begin{align*}
S^n(\{\nu_t^n, f_t\}_{0\leq t\leq 1})
=S^n(\{\nu_t, f_t\}_{0\leq t\leq 1})+\oo_n(1). 
\end{align*}
We can use the exponential martingale \eqref{e:Dt0} to obtain the large deviations upper bound as follows.
\begin{align}\begin{split}\label{e:upp}
&\phantom{{}={}}\bP(\{\nu^n_t\}_{0\leq t\leq 1}\in \bB(\{\nu_t\}_{0\leq t\leq 1}, \delta))
=\bE\left[\bm1(\{\nu^n_t\}_{0\leq t\leq 1}\in \bB(\{\nu_t\}_{0\leq t\leq 1}, \delta))\frac{e^{n^2S^n(\{\nu_t^n, f_t\}_{0\leq t\leq 1})}}{e^{n^2S^n(\{\nu_t^n, f_t\}_{0\leq t\leq 1})}}\right]\\
&=\bE\left[\bm1(\{\nu^n_s\}_{0\leq t\leq 1}\in \bB(\{\nu_s\}_{0\leq t\leq 1}, \delta))e^{n^2S^n(\{\nu_t^n, f_t\}_{0\leq t\leq 1})}\right]\frac{e^{\oo(n^2)}}{e^{n^2S^n(\{\nu_t, f_t\}_{0\leq t\leq 1})}}\\
&\leq \bE\left[e^{n^2S^n(\{\nu_t^n, f_t\}_{0\leq t\leq 1})}\right]\frac{e^{\oo(n^2)}}{e^{n^2S^n(\{\nu_t, f_t\}_{0\leq t\leq 1})}}=e^{-n^2S^n(\{\nu_t, f_t\}_{0\leq t\leq 1})+\oo_n(1))}.
\end{split}\end{align}
The large deviations upper bound \eqref{e:DBMupbb} follows from rearranging \eqref{e:upp}, and taking the  infimum over $f\in \cC^{2,1}_b$.

\subsection{Large deviations Lower Bound}
In the rest of this section, we prove the large deviations lower bound of Theorem \ref{t:DBMLDP},
namely we show that for any continuous measure-valued process $\{\nu_t\}_{0\leq t\leq 1}$, we have

\begin{align}\label{e:DBMlow}
&\liminf_{\delta\rightarrow 0}\liminf_{n\rightarrow\infty}\frac{1}{n^2}\log \bP(\{\nu^n_t\}_{0\leq t\leq 1}\in \bB(\{\nu_t\}_{0\leq t\leq 1}, \delta))
\geq -S_{\mu_0}(\{\nu_t\}_{0\leq t\leq 1}).
\end{align}
The proof itself will be used to derive the large deviations for Dyson  Bessel processes as it allows to control the positions of the extreme particules, see Proposition \ref{p:lowerbound2}, key to control the singularity at the origin of the Dyson Bessel process. 

The proof consists of two steps, in the first step we approximate $\{\nu_t\}_{0\leq t\leq 1}$ by a sequence of measure-valued process with benign properties.
\begin{proposition}\label{p:approximation}
Fix a probability measure $\mu_0$ with finite second moment.
Then, any  measure-valued process $\{\nu_t\}_{0\leq t\leq 1}$ with $S_{\mu_0}(\{\nu_t\}_{0\leq t\leq 1})<\infty$
can be approximated by a sequence of measure-valued processes $\{\nu^\varepsilon_t\}_{0\leq t\leq 1}$ satisfying
\begin{itemize}
\item $\nu_t^\varepsilon$ has uniformly bounded density $\rho_t^\varepsilon$, $\supp(\nu_t^\varepsilon)$ is a single interval for all times $t\in [0,1]$, and
\begin{align*}
\lim_{\varepsilon\rightarrow 0} \sup_{0\leq t\leq 1}d(\nu_t, \nu^\varepsilon_t )=0.
\end{align*}
\item The dynamical entropy satisfies
\begin{align}\label{convS}
\lim_{\varepsilon\rightarrow 0}S_{\nu_0^\varepsilon}(\{\nu_t^\varepsilon\}_{0\leq t\leq 1})= S_{\mu_0}(\{\nu_t\}_{0\leq t\leq 1}).
\end{align}
\item The density  $\{\rho^\varepsilon_t(x)\}_{0\leq t\leq 1}$ of the measure-valued process $\{\nu^\varepsilon_t\}_{0\leq t\leq 1}$ is smooth in both $x,t$, and the corresponding drift $\del_x k^{\varepsilon}_t(x)$ as defined by
\begin{align*}
\del_t\rho^{\varepsilon}_t+\del_x(\rho^{\varepsilon}_tu^{\varepsilon}_t)=0,\quad 
u_t^\varepsilon(x) = \frac{1}{2} H(\nu_t^\varepsilon)(x) +\del_xk_t^\varepsilon (x),
\end{align*}
is also smooth in both $x,t$.

\end{itemize}
\end{proposition}
Above, smooth means differentiable and with continuous derivative (we shall not need more the proof yields eventually the existence of more derivatives).

\begin{proposition}\label{p:lowerb}
Fix a probability measure $\mu_0$ satisfying Assumption \ref{a:mu0} and $\delta>0$. Let $\tilde \mu_0$ be a compactly supported probability measure such that  $d(\mu_0, \tilde \mu_0)\leq \delta/3$.
Let $\{\tilde\nu_t(x)\}_{0\leq t\leq 1}$
be a  compactly supported measure-valued process  with a smooth density  $\tilde\rho$ in both $x,t$ such that $\tilde\nu_0=\tilde\mu_0$. Assume that the corresponding drift $\del_x \tilde k_t$  defined by
\begin{align}\label{driftk}
\del_t\tilde \rho_t(x)+\del_x(\tilde \rho_t\tilde u_t)=0,\quad 
\tilde u_t (x)= \frac{1}{2} H(\tilde \nu_t)(x)+\del_x\tilde k_t(x),
\end{align}
is also smooth in both $x,t$.  Then, the following large deviations lower bound holds
\begin{align*}\begin{split}
& \liminf_{n\rightarrow\infty}\frac{1}{n^2}\log\bP(\{\nu^{n}_t\}_{0\leq t\leq 1}\in \bB(\{\tilde \nu_t\}_{0\leq t\leq 1}, \delta))
\geq  -S_{\tilde \mu_0}(\{\tilde \nu_t\}_{0\leq t\leq 1})+\oo_\delta(1).
\end{split}\end{align*}
\end{proposition}

\begin{proof}[Proof of Large deviations lower bound \eqref{e:DBMlow}]
We  approximate $\{\nu_t\}_{0\leq t\leq 1}$ by $\{\nu^\varepsilon_t\}_{0\leq t\leq 1}$ as in Proposition \ref{p:approximation} and  take $\{\tilde \nu_t(x)\}_{0\leq t\leq 1}$ equal  $\{\nu^\varepsilon_t\}_{0\leq t\leq 1}$   in Proposition \ref{p:lowerb}  with sufficiently small $\varepsilon$.
 Then it follows
\begin{align*}\begin{split}
&\phantom{{}={}}\lim_{n\rightarrow\infty}\frac{1}{n^2}\log\bP(\{\nu^n_t\}_{0\leq t\leq 1}\in \bB(\{\nu_t\}_{0\leq t\leq 1}, \delta))
\geq \lim_{\varepsilon\rightarrow 0}\lim_{n\rightarrow\infty}\frac{1}{n^2}\log\bP(\{\nu^n_t\}\in \bB(\{\nu^{\varepsilon}_t\}_{0\leq t\leq 1}, \delta/2))\\
&\geq \lim_{\varepsilon\rightarrow 0} -S_{\nu_0^\varepsilon}(\{\nu_t^\varepsilon\}_{0\leq t\leq 1}))+\oo_\delta(1)\geq -S_{\mu_0}(\{\nu_t\}_{0\leq t\leq 1}))+\oo_\delta(1),
\end{split}\end{align*} where in the last inequality we used \eqref{convS}.
The large deviations lower bound follows by taking $\delta\rightarrow 0$.
\end{proof}


\begin{proof}[Proof of Proposition \ref{p:approximation}]
We fix three parameters $\varepsilon_3\ll\varepsilon_2\ll\varepsilon_1\ll 1$.  The construction of $\nu_t^\varepsilon$ consists of the following three steps. Note that $S$ is lower semi-continuous hence we only need to show that
$$\limsup_{\varepsilon\rightarrow 0} S_{\nu^\varepsilon_0}(\{\nu_t^\varepsilon\}_{0\leq t\leq 1})\le S_{\mu_0}(\{\nu_t\}_{0\leq t\leq 1}).$$ 
\emph{Step 1 (Free Convolution).} We replace $\nu_t$ by $\nu_t^{(1)}=\nu_t\boxplus\sigma_{\varepsilon_1}$, its free convolution with a small semi-circle distribution of size $\varepsilon_1$. Then we have $d(\nu_t, \nu_t^{(1)})=\oo_{\varepsilon_1}(1)$. More importantly, $\nu_t^{(1)}(dx)=\rho_t^{(1)}(x)\d x $ has density bounded by $\OO(1/\sqrt{{\varepsilon_1}})$, and  it is proven in \cite{cabanal2003discussions} that
\begin{align}\label{bor}
S_{\nu_0^{(1)}}(\{\nu_t^{(1)}\}_{0\leq t\leq 1})\leq S_{\mu_0}(\{\nu_t\}_{0\leq t\leq 1}).
\end{align}
By our assumption $S_{\mu_0}(\{\nu_t\}_{0\leq t\leq 1})<\infty$, Proposition \ref{p:rate} implies that $\nu_1$ has bounded second moment and finite free entropy $-\infty<\Sigma(\nu_1)<\infty$. The same bound holds for its free convolution with semi-circle distribution, i.e. $-\infty<\Sigma(\nu_1^{(1)})<\infty$. Moreover, the second moments of $\rho^{(1)}$ and $u^{(1)}$ under $\rho^{(1)}_t(x)\d x\d t$ are bounded independently of $\varepsilon_1$ by Proposition \ref{p:rate} and \eqref{bor}.

\emph{Step 2 (Truncation).} 
If $\{\nu^{(1)}_{t}\}_{0\leq t\leq 1}$ is not compactly supported, in this step we
 truncate it to have compact support.
Let $\fa(t), \fb(t)$ be such that
\begin{align*}
\int_\infty^{\fa(t)} \rho_t^{(1)}(x)\rd x={\varepsilon_1}/2, \quad \int^\infty_{\fb(t)}\rho_t^{(1)}(x)\rd x={\varepsilon_1}/2\,.
\end{align*}
Observe that because $t\rightarrow \nu^{(1)}_t$ is weakly continuous and with bounded density, $t\rightarrow \fa(t)$ and $t\rightarrow \fb(t)$ are continuous. Moreover, because the second moments of $\nu^{(1)}_t$ are uniformly bounded we see that $\fa(t)$ and $\fb(t)$ are at most of order $1/\sqrt{\varepsilon_1}$. 
Then we restrict $\rho^{(1)}_t$ on $[\fa(t),\fb(t)]$ by setting
\begin{align*}
\rho^{(2)}_t= \frac{\bm1([\fa(t),\fb(t)])}{1-{\varepsilon_1}}\rho^{(1)}_t(x).
\end{align*}
and let $\nu^{(2)}_{t}(\d x)=\rho^{(2)}_{t}(x)\d x$.
From the construction,  we have $d(\rho^{(1)}_t, \rho_t^{(2)})=\oo_{\varepsilon_1}(1)$, and $\rho_t^{(2)}$ has bounded $L_2$ norm. The corresponding $u^{(2)}_t(x)$ is the restriction of $u^{(1)}_t(x)$ to $[\fa(t),\fb(t)]$. Hence, we get
\begin{align}\label{e:largeb}
&S_{\nu_0^{(2)}}(\{\nu_t^{(2)}\}_{0\leq t\leq 1})
=\frac{\beta}{2}\left(\int_0^1 \int ((u_t^{(2)})^2+\frac{\pi^2}{12}(\rho^{(2)}_t(x))^2)\rho^{(2)}_t(x)\rd x\rd t-\frac{1}{2}(\Sigma(\nu^{(2)}_1)-\Sigma(\nu^{(2)}_0))\right)\nonumber\\
&=\frac{\beta}{2}\left(\int_0^1 \frac{1}{1-\varepsilon_1}\int_{\fa(t)}^{\fb(t)} ((u_t^{(1)})^2+\frac{\pi^2}{12}(\frac{\rho^{(1)}_t(x)}{1-\varepsilon_1})^2)\rho^{(1)}_t(x)\rd x\rd t-\frac{1}{2}(\Sigma(\nu^{(2)}_1)-\Sigma(\nu^{(2)}_0))\right)\\
&\rightarrow\frac{\beta}{2}\left(\int_0^1 \int ((u_t^{(1)})^2+\frac{\pi^2}{12}(\rho^{(1)}_t(x))^2)\rho^{(1)}_t(x)\rd x\rd t-\frac{1}{2}(\Sigma(\nu^{(1)}_1)-\Sigma(\nu^{(1)}_0))\right)=S_{\nu_0^{(1)}}(\{\nu_t^{(1)}\}_{0\leq t\leq 1}),\nonumber
\end{align}
as ${\varepsilon_1}\rightarrow 0$ {by monotone convergence theorem.}

\emph{Step 3. (Smoothing)} 
We first extend $\rho_t^{(2)}$  by $\rho_0^{(2)}$ for $t\leq 0$ and $\rho^{(2)}_1$ for $t\geq 1$. Then we replace it by its convolution with a bump function $\varphi_{\varepsilon_3}$ on the scale $\varepsilon_3$, i.e a smooth function  with $L^{1}$ norm equal to one, supported on $[-\varepsilon_3, \varepsilon_3]^2$, with $\varepsilon_3\ll \varepsilon_1$:
\begin{align*}
\rho_t^{(3)}(x)=\int\rho^{(2)}_{t-s}(x-y) \varphi_{\varepsilon_3}(y,s)\rd y\rd s.
\end{align*}
Then $\nu^{(3)}(\d x)=\rho_t^{(3)}(x)\d x$ is supported on $[\fa(t)-\varepsilon_3,\fb(t)+\varepsilon_3]$.
Next we replace this smoothed density by its average with a smooth characteristic function $\chi_t(x)$ constructed in the following way. Take a smoothed step function $\phi(x)$ such that $\phi(x)=0$ for $x\leq 0$, $\phi(x)=1$ for $x\geq {\varepsilon_1}$. Let $\tilde \fa(t),\tilde\fb(t)$ be two smooth functions such that for all times $[\fa(t)-\varepsilon_3,\fb(t)+\varepsilon_3]\subset [\tilde \fa(t),\tilde\fb(t)]$. Then we let 
\begin{align*}
\chi_t(x)=
\left\{\begin{array}{cc}
\phi(x-\tilde\fa(t)-2{\varepsilon_1}), & x\in [\tilde\fa(t)-2{\varepsilon_1}, \tilde\fa(t)-{\varepsilon_1}],\\
\Omega(1), & x\in[\tilde\fa(t)-{\varepsilon_1}, \tilde\fb(t)+{\varepsilon_1}],\\
\phi(\tilde\fb(t)+2{\varepsilon_1}-x), & x\in [\tilde\fb(t)+{\varepsilon_1}, \tilde\fb(t)+2{\varepsilon_1}].
\end{array}\right.
\end{align*}
Moreover, we construct $\chi_t(x)$ such that $\int \chi_t(x)\rd x=1$ for all $0\leq t\leq 1$.  Because $\fa(t),\fb(t)$ are at most of order $1/\sqrt{\varepsilon_1}$,  we can choose $\tilde \fa(t)$ and $\tilde\fb(t)$ such that $\chi_t$ is lower bounded by $\OO(\sqrt{\varepsilon_1})$ on $[\tilde\fa(t)-{\varepsilon_1}, \tilde\fb(t)+{\varepsilon_1}]$. We can also make sure that it is upper bounded by one, and such that $\int^x\partial_t \chi_t$ is uniformly bounded. 
We replace $\rho_t^{(2)}$ by
\begin{align*}
\rho_t^\varepsilon(x)={\varepsilon_2}\chi_t(x)+(1-{\varepsilon_2})\rho_t^{(3)}(x).
\end{align*}
Then $\rho_t^\varepsilon(x)$ is smooth, and for $x\in[\tilde \fa(t)-{\varepsilon_1}, \tilde \fb(t)+{\varepsilon_1}]$, $\rho_t^\varepsilon$ is uniformly lower bounded by $\Omega({\varepsilon_2}\sqrt{\varepsilon_1})$. For $x\in [\tilde \fa(t)-2{\varepsilon_1}, \tilde \fa(t)-{\varepsilon_1}]$, $\rho_t^\varepsilon(x)={\varepsilon_2}\phi(x-\tilde \fa(t)-2{\varepsilon_1})$ and $x\in [\tilde \fb(t)+{\varepsilon_1}, \tilde \fb(t)+2{\varepsilon_1}]$, $\rho_t^\varepsilon(x)={\varepsilon_2}\phi(\tilde \fb(t)+2{\varepsilon_1}-x)$.
From the construction, we have that $d(\rho^{\varepsilon}_t, \rho_t^{(2)})=\oo_{\varepsilon_1}(1)$.

For the second term in the dynamical entropy \eqref{Sent}, by Young's convolution inequality,
\begin{align}\begin{split}\label{e:Young}
\int (\rho_t^{ \varepsilon}(x))^3 \rd x\rd t
&=\int \left({\varepsilon_2}\chi_t(x)+(1-{\varepsilon_2})\rho_t^{(3)}(x)\right)^3\rd x\rd t \\
&=(1+\OO({\varepsilon_2}))\int \left(\int\rho^{(2)}_{t-s}(x-y) \varphi_{\varepsilon_3}(y,s)\rd y\rd s\right)^3\rd x\rd t +\OO({\varepsilon_2})\\
&\leq (1+\OO({\varepsilon_2}))\int (\rho_t^{(2)}(x))^3\rd x\rd t +\OO({\varepsilon_2}),
\end{split}\end{align}
where we used that $\chi_t$ was uniformly bounded above and with bounded expectation. 
In the following we study the first term in \eqref{Sent},
\begin{align*}
\int_0^1\int \frac{\left(\del_t\int^x \rho^{\varepsilon}_t(y)\rd y\right)^2}{\rho^{\varepsilon}_t(x)}\rd x\rd t
=\int_0^1\int \frac{\left(\del_t\int^x ({\varepsilon_2}\chi_t(y)+(1-{\varepsilon_2})\rho_t^{(3)}(y))\rd y\right)^2}{{\varepsilon_2}\chi_t(x)+(1-{\varepsilon_2})\rho_t^{(3)}(x)}\rd x\rd t.
\end{align*}

For $x\in [\tilde \fa(t)-2{\varepsilon_1}, \tilde\fa(t)-{\varepsilon_1}]$, $\rho_t^\varepsilon(x)={\varepsilon_2}\phi(x-\tilde\fa(t)-2{\varepsilon_1})$, the integrand simplifies and is of order  ${\varepsilon_2} (\tilde\fa'(t))^2\phi(x-\tilde\fa(t)-2{\varepsilon_1})$, the total contribution is $\OO(\varepsilon_1\varepsilon_2)$. Similarly for $x\in [\tilde\fb(t)+{\varepsilon_1}, \tilde\fb(t)+2{\varepsilon_1}]$, $\rho_t^\varepsilon(x)={\varepsilon_2}\phi(\tilde \fb(t)+2{\varepsilon_1}-x)$, the total contribution is $\OO(\varepsilon_1\varepsilon_2)$. We get
\begin{align}\label{e:trucc}
\int_0^1\int \frac{\left(\del_t\int^x \rho^{\varepsilon}_t(y)\rd y\right)^2}{\rho^{\varepsilon}_t(x)}\rd x\rd t
=\int_0^1\int_{\fa(t)-{\varepsilon_1}}^{\fb(t)+{\varepsilon_1}} \frac{\left(\del_t\int^x \rho^{\varepsilon}_t(y)\rd y\right)^2}{\rho^{\varepsilon}_t(x)}\rd x\rd t+\OO({\varepsilon_1}\varepsilon_2).
\end{align}
For $x\in [\tilde\fa(t)-{\varepsilon_1}, \tilde \fb(t)+{\varepsilon_1}]$, we have $\rho_t^\varepsilon(x)$ is lower bounded by $\Omega(\varepsilon_2({\varepsilon_1})^{1/2})$. Moreover, from Step 1, we know that $\rho^{(1)}_t(x)$ is at most of order
$ \OO( 1/{\sqrt{{\varepsilon_1}}})$, and so  are the densities $\rho_t^{(2)}(x)$ and $ \rho^\varepsilon_t(x)$. 
The boundedness of $S(\{\nu_t^{(2)}\}_{0\leq t\leq 1})$ implies that $\del_t\int^x \rho_t^{(2)}(y)\rd y$ is in $L^2$:
$$\int_0^1 \int \left(\del_t\int^x \rho_t^{(2)}(y)\rd y\right)^2 \rd x \rd t \lesssim \frac{1}{\sqrt{\varepsilon_1}} \int_0^1 \int  (u_t^{(2)}(y))^2 \rho_t^{(2)}(y)\rd y\rd t\le \frac{2}{\beta\sqrt{\varepsilon_1} }S_{\nu_0^{(2)}}(\{\nu_t^{(2)}\}_{0\leq t\leq 1})
\,.$$
Therefore the convolution density $\del_t\int^x \rho_t^{(3)}(y)\rd y=\int^x \del_t(\rho_\cdot^{(2)}* \varphi_{\varepsilon_3})(y,t)\rd y$ converges to $\int^x \del_t\rho_t^{(2)}(y)\rd y$ in $L^2$ norm as $\varepsilon_3\rightarrow 0$ faster than $\varepsilon_1$ . 
\begin{align}\begin{split}\label{e:r3r2}
&\phantom{{}={}}\left|\int_0^1\int_{\tilde\fa(t)-{\varepsilon_1}}^{\tilde\fb(t)+{\varepsilon_1}} \frac{\left(\del_t\int^x \rho^{(3)}_t(y)\rd y\right)^2}{\rho^{\varepsilon}_t(x)}\rd x\rd t-\int_0^1 \int\frac{\left(\del_t\int^x \rho^{(2)}_t(y)\rd y\right)^2}{\rho^\varepsilon_t(x)}\rd x\rd t\right|\\
&\lesssim \int_0^1 \int_{\tilde\fa(t)-{\varepsilon_1}}^{\tilde\fb(t)+{ \varepsilon_1}} \frac{1}{\varepsilon_2(\varepsilon_1)^{\frac{1}{2}}}\left|\left(\int^x \del_t\rho^{(3)}(y)\rd y \right)^2-\left(\int^x \del_t\rho^{(2)}_t(x)\rd y\right)^2\right|
\rd x\rd t\rightarrow 0,
\end{split}\end{align}
provided we choose $\varepsilon_3$ going to zero fast enough with respect to $\varepsilon_1$ and $\varepsilon_2$. 
Now we can estimate the right hand side of  \eqref{e:trucc} as
\begin{align}\begin{split}\label{e:r3}
&\int_0^1\int_{\tilde\fa(t)-{\varepsilon_1}}^{\tilde\fb(t)+{\varepsilon_1}} \frac{\left(\del_t\int^x \rho^{\varepsilon}_t(y)\rd y\right)^2}{\rho^{\varepsilon}_t(x)}\rd x\rd t
=\int_0^1\int_{\tilde\fa(t)-{\varepsilon_1}}^{\tilde\fb(t)+{\varepsilon_1}} \frac{\left({\varepsilon_2}\int^x \del_t \chi_t(y)+ (1-{\varepsilon_2})\int^x \del_t\rho^{(3)}(y)\rd y \right)^2}{\rho^{\varepsilon}_t(x)}\rd x\rd t\\
&=
\OO(\frac{{\varepsilon_2}}{\varepsilon_1^{1/2}})+\OO(\varepsilon_2) \int_0^1\int_{\tilde\fa(t)-{\varepsilon_1}}^{\tilde\fb(t)+{\varepsilon_1}} \frac{| \int^x \del_t\rho^{(3)}(y)\rd y |}{\rho^{\varepsilon}_t(x)}\rd x\rd t
+(1-{\varepsilon_2})^2\int_0^1\int_{\tilde\fa(t)-{\varepsilon_1}}^{\tilde\fb(t)+{\varepsilon_1}} \frac{\left(\int^x \del_t\rho^{(3)}(y)\rd y \right)^2}{\rho^{\varepsilon}_t(x)}\rd x\rd t\\
&\leq \OO(\frac{{\varepsilon_2}}{\varepsilon_1^{1/2}})+\OO(\varepsilon_2)\int_0^1\int_{\tilde\fa(t)-{\varepsilon_1}}^{\tilde\fb(t)+{\varepsilon_1}} \frac{ ({\varepsilon_2})^{1/2}+(1/{\varepsilon_2})^{1/2}\left( \int^x \del_t\rho^{(3)}(y)\rd y\right)^2 }{\rho^{\varepsilon}_t(x)}\rd x\rd t\\
&+(1-{\varepsilon_2})^2\int_0^1\int_{\tilde\fa(t)-{\varepsilon_1}}^{\tilde\fb(t)+{\varepsilon_1}} \frac{\left(\int^x \del_t\rho^{(3)}(y)\rd y \right)^2}{\rho^{\varepsilon}_t(x)}\rd x\rd t\\
&= \OO(\frac{{\varepsilon_2}}{\varepsilon_1^{1/2}})+ \OO(\frac{{\varepsilon_2^{1/2}}}{\varepsilon_1^{1/2}}) +
(1+\OO(({\varepsilon_2})^{1/2}))\int_0^1\int_{\tilde\fa(t)-{\varepsilon_1}}^{\tilde\fb(t)+{\varepsilon_1}} \frac{\left(\int^x \del_t\rho^{(3)}(y)\rd y \right)^2}{\rho^{\varepsilon}_t(x)}\rd x\rd t,
\end{split}\end{align}
where we used that $\rho_t^\varepsilon(x)\gtrsim {\varepsilon_2}\sqrt{\varepsilon_1}$ on $x\in[\tilde\fa(t)-{\varepsilon_1}, \tilde\fb(t)+{\varepsilon_1}]$. 
Moreover,  for the second term in \eqref{e:r3} we have 
\begin{align}\begin{split}\label{e:r2}
&\int_0^1 \int\frac{\left(\del_t\int^x \rho^{(3)}_t(y)\rd y\right)^2}{\rho^\varepsilon_t(x)}\rd x\rd t
=
\int_0^1 \int\frac{\left(\del_t\int^x \rho^{(3)}_t(y)\rd y\right)^2}{{\varepsilon_2}\chi_t(x)+(1-{\varepsilon_2})\rho_t^{(3)}(x)}\rd x\rd t\\
&\rightarrow 
\int_0^1 \int\frac{\left(\del_t\int^x \rho^{(2)}_t(y)\rd y\right)^2}{{\varepsilon_2}\chi_t(x)+(1-{\varepsilon_2})\rho_t^{(2)}(x)}\rd x\rd t\leq \frac{1}{1-{\varepsilon_2}}\int_0^1 \int\frac{\left(\del_t\int^x \rho^{(2)}_t(y)\rd y\right)^2}{\rho_t^{(2)}(x)}\rd x\rd t.
\end{split}\end{align}
when $\varepsilon_{3}$ goes to zero and we used \eqref{e:r3r2}.
Combining the estimates \eqref{e:trucc}, \eqref{e:r3r2}, \eqref{e:r3} and \eqref{e:r2} all together, we get
\begin{align}\label{e:fft}
\int_0^1\int \frac{\left(\del_t\int^x \rho^{\varepsilon}_t(y)\rd y\right)^2}{\rho^{\varepsilon}_t(x)}\rd x\rd t
=(1+\oo_{\varepsilon_3,\varepsilon_2,{\varepsilon_1}}(1))\int_0^1 \int\frac{\left(\del_t\int^x \rho^{(2)}_t(y)\rd y\right)^2}{\rho_t^{(2)}(x)}\rd x\rd t+\oo_{\varepsilon_3,\varepsilon_2,{\varepsilon_1}}(1).
\end{align}
where $\oo_{\varepsilon_3,\varepsilon_2,{\varepsilon_1}}$ is small when $\varepsilon_3\ll\varepsilon_2\ll\varepsilon_1$. 
Moreover, since the densities $\rho_0^{(2)}(x), \rho_1^{(2)}(x)$ are bounded by $\OO(1/\sqrt{{\varepsilon_1}})$, and are compactly supported, it is easy to see that 
\begin{align*}
\frac{1}{2}\left(\Sigma(\nu_1^{(3)})-\Sigma(\nu_0^{(3)})\right)
\rightarrow \frac{1}{2}\left(\Sigma(\nu_1^{(2)})-\Sigma(\nu_0^{(2)})\right),
\end{align*}
as $\varepsilon_2,\varepsilon_3$ go to zero.
We conclude from combining \eqref{e:Young} and \eqref{e:fft} that
\begin{align*}
\limsup_{\varepsilon\downarrow 0} S(\{\nu_t^{\varepsilon}\}_{0\leq t\leq 1})\leq S(\{\nu_t\}_{0\leq t\leq 1}).
\end{align*}

From the construction, $\rho_t^\varepsilon$ has uniformly bounded density, i.e. $\rho_t^\varepsilon=\OO(1/\sqrt{\varepsilon_1})$, $\supp(\nu_t^\varepsilon)$ is a single interval, and
\begin{align*}
\lim_{\varepsilon\rightarrow 0} \sup_{0\leq t\leq 1}d(\nu_t, \nu^\varepsilon_t )=0.
\end{align*}

For the density $\rho^\varepsilon_s(x)$, it satisfies
\begin{align*}
\del_t\rho^{\varepsilon}_t+\del_x(\rho^{\varepsilon}_tu^{\varepsilon}_t)=0.
\end{align*}
The drift is given by
\begin{align*}
\del_x k^{\varepsilon}_t(x)=-\frac{\int^x \del_t\rho^{\varepsilon}_t(y)\d y}{\rho_t^{\varepsilon}(x)}-\frac{1}{2}H(\nu^{\varepsilon}_t)(x).
\end{align*}
Since $\rho^{\varepsilon}_t(x)$ is smooth, i.e. in $\cC^\infty$, then $H(\nu^{\varepsilon}_t)$ is also smooth (see Remark \ref{hilberts}). For the regularity of the drift term $\del_x k^{\varepsilon}_t(x)$, we need to understand the regularity of $(\int^x \del_t\rho^{\varepsilon}_t)/\rho^{\varepsilon}_t$. 
By our construction, $\rho_t^{\varepsilon}$ is supported on $[\tilde\fa(t)-2{\varepsilon_1}, \tilde\fb(t)+{\varepsilon_1}]$, and has positive smooth density. Thus $(\int^x \del_t\rho^{\varepsilon}_t)/{\rho^{\varepsilon}_t}$ is smooth inside the support of $\rho_t^{\varepsilon}$. 
Close to the boundary of the support, on $[\tilde \fa(t)-2{\varepsilon_1}, \tilde \fa(t)-{\varepsilon_1}]$, $\rho_t^{\varepsilon}=\phi(x-\tilde \fa(t)-2{\varepsilon_1})$, where $\phi(x)$ is smooth, and $\phi(x)=0$ for $x\leq 0$. In this way
\begin{align}
\del_x k^{\varepsilon}_t(x)=-\frac{\int^x \del_t\rho^{\varepsilon}_t(y)\rd y}{\rho^{\varepsilon}_t(x)}-\frac{1}{2}H(\nu^{\varepsilon}_t)(x)=\frac{\del_t\int^x \phi(y-\tilde \fa(t)-2{\varepsilon_1})\rd y }{\phi(x-\tilde \fa(t))}-\frac{1}{2}H(\nu^{\varepsilon}_t)(x)=\tilde \fa'(t)-\frac{1}{2}H(\nu^{\varepsilon}_t)(x),
\end{align}
which is smooth in a neighborhood of $[\tilde \fa(t)-2{\varepsilon_1}, \tilde \fa(t)-{\varepsilon_1}]$. The same argument holds in a neighborhood of the right edge $[\tilde \fb(t)+{\varepsilon_1}, \tilde \fb(t)+2{\varepsilon_1}]$. 
So  $\del_x k^{\varepsilon}_t(x)$ is a smooth drift.
    
\end{proof}

%

In the following proposition, we show that for the  Dyson  Brownian motion with smooth drift, the locations of its particles are close to the quantiles of the limiting profile. Proposition \ref{p:lowerb} will then be an easy consequence.
\begin{proposition}\label{p:lowerbound2}
Let  $\{\nu_t\}_{0\leq t\leq 1}$ be a  measure valued process  with bounded support and  with smooth  density $\rho_t(x)$ in both $x,t$, and such that  the drift $\del_x  k_t(x)$  such that
\begin{align}\label{e:VK}
\del_t\rho_t(x)+\del_x(\rho_t(\frac{1}{2} H( \nu_t)(x)+\del_x k_t(x)))=0,
\end{align}
is  uniformly Lipschitz : $|\del_x  k_t(x)-\del_x  k_t(y)|\leq K|x-y|$  for all real numbers $x,y$.
We denote by $\gamma_i(t)$ the $ith$ $(1/n)$-quantiles of $ \rho_t$ given by 
\begin{align}\label{e:quant}
\frac{i-1/2}{n}=\int_\infty^{\gamma_i(t)} \rho_t(x)\rd x,\quad 1\leq i\leq n.
\end{align}
For $\beta\ge 1$, we consider the  Dyson  Brownian motion with drift $\del_x  k_t$ which is the unique strong solution of
\begin{align}\label{e:DBMk}
\rd  x_i(t)=\frac{\rd W_i (t)}{\sqrt {\beta n}}+\frac{1}{2n}\left(\sum_{j: j \neq i}\frac{1}{ x_i(t)- x_j(t)}\right)\d t+\del_x  k_t(x_i (t)) \d t.
\end{align}
starting from $(x_j(0))_{1\le i\le n}$.
Let  $M$ be a positive real number.  Then it holds 
\begin{align*}
\sup_{t\in [0,1]}\max_{1\le i\le n} | x_i(t)-\gamma_i(t)|\leq  e^{Kt}\left(\max_{1\le j\le n} (| x_j(0)-\gamma_{j}(0)|)+\frac{M}{\sqrt n}\right),
\end{align*}
with probability going to one as $M$ goes to infinity.
\end{proposition}
\begin{remark}\label{hilberts} Note that if $\nu$ is a probability measure with density $\rho$ which is $\cC^k_b$ and with compact support, then $H\nu$ is $\cC^{k-1}_b$. Indeed, for $\rho$ with support in $[-A,A]$ and $|x|\le M$
$$H\nu(x)=P.V. \int\frac{\rho(y)}{x-y} \d y=P.V. \int_{|x-y|\le A+M} \frac{\rho(y)}{x-y} \d y= \int_{|x-y|\le A+M} \frac{\rho(y)-\rho(x)}{x-y}\d y,$$
where we noticed that $ P.V. \int_{|x-y|\le A+M} \frac{1}{x-y} \d y=0$. 
\end{remark}
\begin{proof}[Proof of Proposition \ref{p:lowerbound2}]
We first show that the $(1/n)$-quantiles of $ \rho_t$ approximately satisfy the equations of Dyson  Brownian motion:
\begin{align}\label{e:apDBM}
\del_t \gamma_i(t)=\frac{1}{2n}\sum_{j:j\neq i}\frac{1}{\gamma_i(t)-\gamma_j(t)}+\del_x  k_t(\gamma_i(t))+\OO\left(\frac{1}{\sqrt{n}}\right), \quad 1\leq i\leq n.
\end{align}
In the following, we denote in short $H(\rho)$ for $H(\rho(x)\d x)$.
By taking derivative with respect to $t$ on both sides of \eqref{e:quant}, we have
\begin{align*}\begin{split}
0&=\del_t \gamma_i(t) \rho_t(\gamma_i(t))+\int^{\gamma_i(t)}\del_t \rho_t(x)\rd x\\
&=\del_t \gamma_i(t) \rho_t(\gamma_i(t))-\int^{\gamma_i(t)}\del_x(\rho_t(H( \rho_t)/2+\del_x k_t(x))) \rd x\\
&=\del_t \gamma_i(t) \rho_t(\gamma_i(t))-\rho_t(\gamma_i(t))(H( \rho_t(\gamma_i(t)))/2+\del_x k_t(\gamma_i(t)) ).
\end{split}\end{align*}
By rearranging the above equality, we obtain the following differential equation for the quantiles of $ \rho_t(x)$
\begin{align}\label{e:gammaeq}
\del_t \gamma_i(t)=\frac{1}{2}H( \rho_t)(\gamma_i(t))+\del_x k_t(\gamma_i(t)). 
\end{align} 
In the following, we prove that we can approximate the Hilbert transform by a discrete sum
\begin{align}\label{e:apHib}
H( \rho_t)(\gamma_i(t))=\frac{1}{n}\sum_{j:j\neq i}\frac{1}{\gamma_i(t)-\gamma_j(t)}+\OO\left(\frac{1}{\sqrt{n}}\right).
\end{align}
Then the proof of  \eqref{e:apDBM} follows from combining \eqref{e:gammaeq} and \eqref{e:apHib}. We fix a large constant $\fC>0$ (which will be chosen later), and divide \eqref{e:apHib} into two cases: either $ \rho_t(\gamma_i(t))\leq \fC/\sqrt n$ or $ \rho_t(\gamma_i(t))\geq \fC/\sqrt n$.

In the following we first discuss the case that $ \rho_t(\gamma_i(t))\leq \fC/\sqrt n$. Since our density $ \rho_t(x)$ is smooth, in particular its first derivative is uniformly bounded $\| \rho_t'\|_\infty\lesssim 1$.   For   $\fc\le \| \rho_t'\|_\infty^{-1} \fC$,
 $\rho_t(x)\leq 2\fC/\sqrt n$ on  $[\gamma_i(t)-\fc/\sqrt n, \gamma_i(t)+\fc/\sqrt n]$. If we take also $\fc\leq 1/4\fC$, we deduce that $\gamma_{i-1}(t)\leq \gamma_i(t)-\fc/\sqrt n$ and $\gamma_{i+1}(t)\geq \gamma_i(t)+\fc/\sqrt n$ (we made the convention $\gamma_0=-\infty$ and $\gamma_{n+1}=+\infty$). Then 
\begin{align}\begin{split}\label{e:diff1}
P.V. \int_{\gamma_{i-1}(t)}^{{\gamma_{i+1}}(t)}\frac{ \rho_t(x)}{\gamma_i(t)-x}\rd x
&=P.V. \int_{\gamma_{i}(t)-\fc/\sqrt n}^{{\gamma_i(t)}+\fc/\sqrt n}\frac{ \rho_t(x)}{\gamma_i(t)-x}\rd x+\OO\left(\frac{1}{\fc\sqrt n}\right)\\
&=P.V. \int_{\gamma_{i}(t)-\fc/\sqrt n}^{{\gamma_i(t)}+\fc/\sqrt n}\frac{ \rho_t(x)- \rho_t(\gamma_i(t))}{\gamma_i(t)-x}\rd x+\OO\left(\frac{1}{\fc\sqrt n}\right)\\
&=\OO\left(\frac{\fc\|\rho_t'\|_\infty}{\sqrt n}+\frac{1}{\fc\sqrt n}\right)=\OO\left(\frac{1}{\sqrt n}\right),
\end{split}\end{align}
where we used  in the  second  line that $P.V.\int_{\gamma_{i}(t)-\fc/\sqrt n}^{{\gamma_i(t)}+\fc/\sqrt n} (\gamma_{i}(t)-x)^{-1}\rd x =0$.
For the integral outside the interval $[\gamma_{i-1}(t), \gamma_{i+1}(t)]$, we have the trivial bounds
\begin{align*}
\int_{-\infty}^{\gamma_{i-1}(t)}\frac{\rho_t(x)\rd x}{\gamma_i(t)-x}
\leq \sum_{j=1}^{i-1}\int_{\gamma_{j-1}(t)}^{\gamma_{j}(t)}\frac{\rho_t(x)\rd x}{\gamma_i(t)-\gamma_j(t)}\leq \sum_{j=1}^{i-1}\frac{1}{n(\gamma_i(t)-\gamma_j(t))},
\end{align*}
and the lower bound
\begin{align*}
\int_{-\infty}^{\gamma_{i-1}(t)}\frac{\rho_t(x)\rd x}{\gamma_i(t)-x}
\geq \sum_{j=1}^{i-1}\int_{\gamma_{j-1}(t)}^{\gamma_{j}(t)}\frac{\rho_t(x)\rd x}{\gamma_i(t)-\gamma_{j-1}(t)}\geq \sum_{j=1}^{i-2}\frac{1}{n(\gamma_i(t)-\gamma_j(t))}.
\end{align*}
Thus we conclude that
\begin{align}\label{e:diff2}
\left|\int_{-\infty}^{\gamma_{i-1}(t)}\frac{\rho_t(x)\rd x}{\gamma_i(t)-x}-\sum_{j=1}^{i-1}\frac{1}{n(\gamma_i(t)-\gamma_j(t))}\right|
\leq \frac{1}{n(\gamma_i(t)-\gamma_{i-1}(t))}\leq \frac{1}{\fc \sqrt n}.
\end{align}
We have the same estimate for the integral from $\gamma_{i+1}(t)$ to $\infty$. The claim \eqref{e:apHib} follows from combining \eqref{e:diff1} and \eqref{e:diff2}.

For the case that $ \rho_t(\gamma_i(t))\geq \fC/\sqrt n$, we have for any integer $k$,
\begin{align*}\begin{split}
\frac{k}{n}=\int_{\gamma_i(t)}^{\gamma_{i+k}(t)} \rho_t(x)\rd x
&=\int_{\gamma_i(t)}^{\gamma_{i+k}(t)}( \rho_t(\gamma_i(t))+\OO(\| \rho'_t\|_\infty |x-\gamma_i(t)|))\rd x\\
&=(\gamma_{i+k}(t)-\gamma_i(t)) \rho_t(\gamma_i(t))+\OO(\| \rho'_t\|_\infty |(\gamma_{i+k}(t)-\gamma_i(t))|^2).\nonumber
\end{split}\end{align*}
By rearranging, we get 
\begin{align}\begin{split}\label{e:loc}
\gamma_{i+k}(t)-\gamma_i(t)&=\frac{k}{n \rho_t(\gamma_i(t))}\frac{1}{1+\OO(\| \rho'_t\|_\infty |(\gamma_{i+k}(t)-\gamma_i(t))| / \rho_t(\gamma_i(t))}\\
&=\frac{k}{n \rho_t(\gamma_i(t))}\left(1+\OO\left(\frac{ k\|  \rho'_t\|_\infty}{ n  \rho_t(\gamma_i(t))^2}\right)\right),
\end{split}\end{align}
provided that 
$k\| \rho'_t\|_\infty \leq  n \rho_t(\gamma_i(t))^2/2$. We have exactly the same estimates for $\gamma_{i}(t)-\gamma_{i-k}(t)$, and 
\begin{align}\label{e:gigk}
\frac{1}{n}\left|\frac{1}{\gamma_{i}(t)-\gamma_{i-k}(t)}+\frac{1}{\gamma_{i}(t)-\gamma_{i+k}(t)}\right|
\lesssim \frac{1}{n}\frac{\OO\left(\frac{ k\|  \rho'_t\|_\infty}{ n  \rho_t(\gamma_i(t))^2}\right)}{\frac{k}{n \rho_t(\gamma_i(t))}}
\lesssim \frac{\|  \rho'_t\|_\infty}{n \rho_t(\gamma_i(t))}.
\end{align}

We take 
\begin{align}\label{e:defd}
d=\left\lfloor \frac{\sqrt n  \rho_t(\gamma_i(t))}{\fC}\right\rfloor +1,
\end{align}
then $d\geq 1$ and 
\begin{align*}
d\| \rho'_t\|_\infty \leq \| \rho'_t\|_\infty+\frac{\| \rho'_t\|_\infty\sqrt n  \rho_t(\gamma_i(t))}{\fC}\leq  n \rho_t(\gamma_i(t))^2/2,
\end{align*}
provided we take $\fC$ large enough. 
By summing over \eqref{e:gigk} from $k=1$ to $k=d-1$, we have
\begin{align}\label{e:diff3}
\sum_{ k=1}^{d-1}\frac{1}{n}\left(\frac{1}{\gamma_{i}(t)-\gamma_{i-k}(t)}+\frac{1}{\gamma_{i}(t)-\gamma_{i+k}(t)}\right)
\lesssim \frac{d\|  \rho'_t\|_\infty}{n \rho_t(\gamma_i(t))}\lesssim \frac{1}{\sqrt n}.
\end{align}
Moreover, by taking $k=d$ in \eqref{e:loc}, we have 
$
|\gamma_{i+d}(t)-\gamma_i(t)|, |\gamma_{i}(t)-\gamma_{i-d}(t)|\asymp 1/\sqrt n.
$
The same argument as for \eqref{e:diff1} gives
\begin{align}\begin{split}\label{e:diff4}
P.V. \int_{\gamma_{i-d}(t)}^{{\gamma_{i+d}}(t)}\frac{ \rho_t(x)}{\gamma_i(t)-x}\rd x
=\OO\left(\frac{1}{\sqrt n}\right).
\end{split}\end{align}
Similarly, following the proof of  
 \eqref{e:diff2}, we get
\begin{align}\label{e:diff5}
\left|\int_{-\infty}^{\gamma_{i-d}(t)}\frac{ \rho_t(x)\rd x}{\gamma_i(t)-x}-\sum_{j=1}^{i-d}\frac{1}{n(\gamma_i(t)-\gamma_{j}(t)}\right|
\leq \frac{1}{n(\gamma_i(t)-\gamma_{i-d}(t))}\lesssim \frac{1}{ \sqrt n}.
\end{align}
The claim \eqref{e:apHib} follows from combining \eqref{e:diff3}, \eqref{e:diff4} and \eqref{e:diff5}.

By taking  the difference of \eqref{e:DBMk} and \eqref{e:apDBM}, we get
\begin{align}\label{e:dTt}\begin{split}
\rd ( x_i(t)-\gamma_i(t))
&=\frac{\rd W_i(t)}{\sqrt {\beta n}}-\frac{1}{2n}\sum_{j:j\neq i} \frac{( x_i(t)-\gamma_i(t))-( x_j(t)-\gamma_j(t))}{( x_i(t)- x_j(t))(\gamma_i(t)- \gamma_j(t))}\d t\\
&+(\del_x  k_t( x_i(t))-\del_x  k_t(\gamma_i(t))) \d t,\quad 1\leq i\leq n.
\end{split}\end{align}
We take $i_*(t)=\arg\max_{i\in\qq{1,n}} ( x_i(t)-\gamma_i(t))$. Since $i_*(t)$ is piecewise constant, $\rd i_*(t)$ is almost surely zero. By plugging $i=i_*(t)$ in \eqref{e:dTt}, and noticing that  the second term on the righthand side is nonpositive, we get
\begin{align*}
\rd ( x_{i_*(t)}(t)-\gamma_{i_*(t)}(t))\leq \frac{\rd W_{i_*(t)}(t)}{\sqrt {\beta n}}+K( x_{i_*(t)}(t)-\gamma_{i_*(t)}(t)) \d t.
\end{align*}
The martingale term $\rd W_{i_*(t)}(t)$ has the same law as a standard Brownian motion.
By Gronwall's inequality, we deduce that
\begin{align*}
 x_{i_*(t)}(t)-\gamma_{i_*(t)}(t)
&\leq e^{Kt}\left( x_{i_*(0)}(0)-\gamma_{i_*(0)}(0)+\int_0^t{{e^{-Ks}}}\frac{\rd W_{i_*(s)}(s)}{\sqrt {\beta n}}\right)\\
&\leq e^{Kt}\left(\max_i| x_{i}(0)-\gamma_{i}(0)|+\frac{M}{\sqrt n}\right),
\end{align*}
where $M$ is a stochastically bounded random variable, uniformly in time (by Doob's martingale inequality).  
It follows that uniformly for any $i\in \qq{1,n}$, 
\begin{align*}
 x_{i}(t)-\gamma_{i}(t)\lesssim e^{Kt}\left(\max_i| x_i(0)-\gamma_i(0)|+\frac{M}{\sqrt n}\right).
\end{align*}
By the same argument, we have a similar lower bound by considering  $i_*(t)=\arg\min_{i\in\qq{1,n}} ( x_i(t)-\gamma_i(t))$. The following holds
\begin{align*}
\gamma_{i}(t)- x_{i}(t)\lesssim e^{Kt}\left(\max_i| x_i(0)-\gamma_i(0)|+\frac{M}{\sqrt n}\right),
\end{align*}
where $M$ is stochastically bounded. This finishes the proof of Proposition \ref{p:lowerbound2}.
\end{proof}

\begin{proof}[Proof of Proposition \ref{p:lowerb}] We first show that changing slightly the initial condition of the Dyson Brownian motion will not change much the large deviations lower bound. This will enable us to consider an initial measure $\tilde\mu_0$ with compact support and finite free entropy.
Let $\mu_0$ be a probability measure and $\tilde\mu_0$ a compactly supported approximation so that $d(\mu_0,\tilde\mu_0)\le\delta/3$. Denote $\tilde F_0(x)=\tilde\mu_0((-\infty,x])$. We construct a new family of initial data 
\begin{align*}
\tilde\nu_0^n=\frac{1}{n}\sum_{i=1}^n \delta_{\tilde x_i(0)},\quad \tilde x_i(0)=\tilde F_0^{-1}((i-1/2)/n).
\end{align*}
In this way $\tilde x_i(0)$ is the $(i-1/2)/n$ quantile of $\tilde \rho_0$. Thanks to Assumption \ref{a:mu0}, we have that 
\begin{align*}
d(\tilde \nu_0^n, \nu_0^n)=\sqrt{\frac{1}{n}\sum_{1\leq i\leq n} |\tilde x_i(0)-x_i(0)|^2}\leq \frac{\delta}{3}+\oo_n(1).
\end{align*}
We consider the Dyson  Brownian motion starting from $\tilde \nu_0^n$,
\begin{align}\label{e:tDBM}
\rd \tilde x_i(t)=\frac{\rd W_i(t)}{\sqrt {\beta n}}+\frac{1}{2n}\left(\sum_{j:j \neq i}\frac{1}{\tilde x_i(t)-\tilde x_j(t)}\right)\d t,
\end{align}
which shares the same Brownian motions $W_i$  as \eqref{e:DBM}. By taking the difference between \eqref{e:DBMa} and \eqref{e:tDBM}, we get
\begin{align}\label{e:difeq}
\del_t (\tilde x_i(t)-x_i(t))^2=\frac{(\tilde x_i(t)-x_i(t))}{n}\sum_{j:j \neq i}\frac{-(\tilde x_i(t)-x_i(t))+(\tilde x_j(t)-x_j(t))}{(\tilde x_i(t)-\tilde x_j(t))( x_i(t)- x_j(t))}.
\end{align}
Averaging over all the indices $i\in \qq{1, n}$, we find
\begin{align*}
\frac{1}{n}\del_t\sum_i (\tilde x_i(t)-x_i(t))^2=-\frac{1}{n^2}\sum_{i<j}\frac{((\tilde x_i(t)-x_i(t))-(\tilde x_j(t)-x_j(t)))^2}{(\tilde x_i(t)-\tilde x_j(t))( x_i(t)- x_j(t))}\leq 0.
\end{align*}
It follows that if we let $\tilde\nu_t^n=\frac{1}{n}\sum_{i=1}^n \delta_{\tilde x_i(t)}$, we have 
\begin{align*}{
d(\tilde \nu_t^n, \nu_t^n)}=\sqrt{\frac{1}{n}\sum_i (\tilde x_i(t)-x_i(t))^2}\leq \sqrt{\frac{1}{n}\sum_i (\tilde x_i(0)-x_i(0))^2}\leq \delta/2,
\end{align*}
provided $n$ is large enough.
As a consequence, we deduce that for any  compactly supported measure-valued process $\{\tilde\nu_t(x)\}_{0\leq t\leq 1}$
  with a smooth density  such that $\tilde\nu_0=\tilde\mu_0$ as in Proposition  \ref{p:lowerb}
\begin{align}\begin{split}\label{e:replace1}
&\phantom{{}={}}\lim_{n\rightarrow\infty}\frac{1}{n^2}\log\bP(\{\nu^n_t\}_{0\leq t\leq 1}\in \bB(\{\tilde \nu_t\}_{0\leq t\leq 1}, \delta))\geq \lim_{n\rightarrow\infty}\frac{1}{n^2}\log\bP(\{\tilde \nu^n_t\}_{0\leq t\leq 1}\in \bB(\{\tilde \nu_t\}_{0\leq t\leq 1}, \delta/2)).
\end{split}\end{align}
Let $\del_x\tilde k_t$ denote the drift  \eqref{driftk}
and  $\bQ^{\beta \tilde k}$ be the distribution
\begin{align*}
\bQ^{\beta \tilde k}=e^{n^2S^n(\{\tilde \nu^n_t, \beta \tilde k_t)\}_{0\leq t\leq 1}} \bQ,
\end{align*}
where $S^n(\{\tilde \nu^n_t, \beta \tilde k_t)\}_{0\leq t\leq 1})$
\begin{align}
S^n(\{\tilde \nu^n_t, \beta \tilde k_t)\}_{0\leq t\leq 1})
&=\frac{1}{n^2}\left(\frac{n L_1^{\beta \tilde k}}{2}-\frac{n^2}{2}\langle L^{\beta \tilde k}, L^{\beta \tilde k}\rangle_1\right),
\quad 
\rd L^{\beta \tilde k}_t=\sum_{i}\beta \del_x\tilde k_t(x_i(t))\frac{\rd W_i(t)}{\sqrt {\beta n}},
\end{align}
is defined in \eqref{e:MLt0} and \eqref{e:defSn}. Because $\del_x \tilde  k_t(x)$  is continuously differentiable and bounded, this is a well defined change of measure.
By Girsanov's formula, 
 under $\bQ^{\tilde k}$, the measure valued process $\{\tilde \nu^{n}_t\}_{0\leq t\leq 1}$ has the same law as the empirical distribution of 
\begin{align}\begin{split}\label{e:DBM3}
\rd \tilde x_i(t)&=\frac{\rd W_i(t)}{\sqrt {\beta n}}+\frac{1}{2n}\left(\sum_{j: j\neq i}\frac{1}{\tilde x_i(t)-\tilde x_j(t)}\right)\d t+\langle \rd L^{\beta \tilde k}, \rd \tilde x_i\rangle_t \\
&=\frac{\rd W_i(t)}{\sqrt {\beta n}}+\frac{1}{2n}\left(\sum_{j: j\neq i}\frac{1}{\tilde x_i(t)-\tilde x_j(t)}\right)\d t+\del_x \tilde k_t(\tilde x_i) \d t.
\end{split}\end{align}

Thanks to Proposition \ref{p:lowerbound2},  it holds that  
\begin{align}\label{e:xkbb}
\sup_{t\in [0,1]}\max_{1\le k\le n}|\tilde x_k(t)-\tilde \gamma_k(t)|\lesssim e^{Kt}\left(\max_{1\le k\le n}( |\tilde x_k(0)-\tilde \gamma_k(0)|)+\frac{M}{\sqrt n}\right),
\end{align}
where the constant $K$ depends on the Lipschitz constant of $\tilde \del_x k_t$, and $M$ is stochastically bounded.  Especially, \eqref{e:xkbb} implies that
$\bQ^{\tilde k}(\{\tilde \nu^{n}_t\}_{0\leq t\leq 1}\in \bB(\{\tilde \nu_t\}_{0\leq t\leq 1}, \delta/2))$ with probability $1-\oo(1)$.
We  conclude that
\begin{align}\begin{split}\label{e:replace2}
\bQ(\{\tilde \nu^{n}_t\}_{0\leq t\leq 1}\in \bB(\{\tilde \nu_t\}_{0\leq t\leq 1}, \delta/2))
&=\bQ^{\beta \tilde k}( e^{-n^2 S^n(\{\tilde \nu^n_t, \beta \tilde k_t\}_{0\leq t\leq 1})}\bm1(\{\tilde \nu^{n}_t\}_{0\leq t\leq 1}\in \bB(\{\tilde \nu_t\}_{0\leq t\leq 1}, \delta/2)))\\
&=\exp\{-n^2(S_{\tilde \mu_0}(\tilde \nu_t)+\oo_\delta(1))\}\bQ^{\beta\tilde k}( \{\tilde \nu^{n}_t\}_{0\leq t\leq 1}\in \bB(\{\tilde \nu_t\}_{0\leq t\leq 1}, \delta/2)\\
&=\exp\{-n^2(S_{\tilde \mu_0}(\tilde \nu_t)+\oo_\delta(1))\}(1-\oo(1)),
\end{split}\end{align}
where in the second line, we used that $S^n(\{\tilde \nu^n_t, \beta \tilde k_t\}_{0\leq t\leq 1})=S_{\tilde \mu_0}(\tilde \nu_t)+\oo_\delta(1)$ for $\{\tilde \nu^{n}_t\}_{0\leq t\leq 1}$ in $ \bB(\{\tilde \nu_t\}_{0\leq t\leq 1}, \delta/2)$ by continuity of
$\nu\rightarrow S^n(\{ \nu_t, \beta \tilde k_t\}_{0\leq t\leq 1})$ and \eqref{e:minimizereq}.
This and \eqref{e:replace1} together finish the proof of \eqref{e:DBMlow}.
\end{proof}

\section{Large deviations for the  Dyson Bessel process}\label{secbessel}

In this section, we prove the  large deviation principle, Theorem \ref{main2},  for the symmetrized  empirical particle density of the Dyson Bessel process
\begin{align}\begin{split}\label{e:dskccopy}
\d s_i(t) 
&=\frac{\d W_{i}}{\sqrt{\beta n}}+\left(\frac{1}{2n}\sum_{j: j \neq i}\frac{1}{s_i(t)-s_j(t)}+\frac{1}{2n}\sum_{j: j\neq i}\frac{1}{s_i(t)+s_j(t)}+\frac{\al_n}{2 s_i(t)}\right)\d t,\quad 1\leq i\leq n,
\end{split}\end{align}
We can symmetrize the Dyson Bessel process \eqref{e:dskccopy}, by setting $s_{-i}(t)=-s_i(t), W_{-i}(t)=-W_i(t)$ for $1\leq i\leq n$, then for $ i\in \qq{-n,n}\setminus \{0\}$, we have 
\begin{align}\begin{split}\label{e:dsk2}
\d s_i(t) 
&=\frac{\d W_{i}(t)}{\sqrt{\beta n}}+\left(\frac{1}{2n}\sum_{j: j \neq \pm i}\frac{1}{s_i(t)-s_j(t)}+\frac{\al_n}{2 s_i(t)}\right)\d t,\\
&=\frac{\d W_{i}(t)}{\sqrt{\beta n}}+\frac{1}{2n}\sum_{j:j\neq i}\frac{\rd t}{s_i(t)-s_j(t)}+\frac{\alpha_n-1/(2n)}{2 s_i(t)}\d t 
\end{split}\end{align}
where in the last line we added a term $(1/2n)\rd t/(s_i(t)-s_j(t))$ with $j=i$, and replaced $\al_n$ by $\al_n-(1/2n)$.

We denote the law of the Dyson Bessel process by $\bP$, and the  empirical particle density and its symmetrized version as 
\begin{align}\label{e:density}
\nu^n_t=\frac{1}{n}\sum_{i=1}^n \delta_{ s_i(t)},\quad \hat\nu^n_t=\frac{1}{2n}\sum_{i=1}^n\left( \delta_{ s_i(t)}+\delta_{ -s_i(t)}\right),\quad 0\leq t\leq 1.
\end{align}
More generally, for a probability measure $\nu$ on the real line we set $\hat \nu$ to be its symmetrized version $\hat\nu(f)=\int (f(x)+f(-x))/2 \d\nu$. Reciprocally, if $\nu$ is a probability on $(0,\infty)$, we can retrieve $\nu$ from $\hat \nu$ by setting $\nu=2\hat\nu|_{(0,+\infty)}$.
If $\alpha_n\geq 1/n\beta$, the solution of \eqref{e:dskccopy} for $t>0$ is non negative  almost surelyand thus $\nu_t^n=2\hat\nu^n_t|_{(0,+\infty)}$. We denote  by $\mathbb M_1^s(\mathbb R)$ the set of symmetric probability measures on the real line and observe that it is a closed subset of $\mathbb M_1(\mathbb R)$.

We recall from \eqref{e:ratt} that given a symmetric measure $\hat \mu_0$ and a continuous symmetric measure-valued  process $\{\hat\nu_t\}_{0\leq t\leq 1}$ with $\hat\nu_0=\hat \mu_0$, we define the following dynamical free entropy: 
\begin{equation}\label{e:ratesingular}
S_{\hat \mu_0}^\al(\{\hat\nu_t\}_{0\leq t\leq 1})=\sup_{f\in \cC^{2,1}_b} S^\al(\{\hat\nu_t,f_t\}_{0\leq t\leq 1}),
\end{equation}
where
\begin{align*}
S^\al(\{\hat\nu_t,f_t\}_{0\leq t\leq 1})
&= L^\alpha_1(\{\hat\nu_t,f_t\}_{0\leq t\leq 1})-\frac{1}{8\beta }\int^1_0 \int (f_s'(x)-f_s'(-x))^2\rd \hat\nu_s(x)\rd s\\
L^\alpha_u(\{\hat\nu_t,f_t\}_{0\leq t\leq 1})&=
\hat\nu_u(f_u)-\hat\nu_0(f_0)-\int_0^u\int \del_s f_s(x) \rd\hat\nu_s(x)\rd s-\frac{1}{2}\int_0^u\int \frac{f_s'(x)-f_s'(y)}{x-y}\rd \hat\nu_s(x)\rd \hat\nu_s(y)\rd s \\
& -\frac{\alpha}{2}\int_0^u\int \frac{f_s'(x)}{x} \rd \hat\nu_s(x)\rd s  .\end{align*}
The supremum is taken over $f_t(x)\in \cC_b^{2,1}(\mathbb R\times [0,1])$ which has bounded twice derivative in $x$ and bounded derivative in $t$. 
We notice that since the measure process $\hat\nu_t$ is symmetric, for any $f_t(x)\in \cC_b^{2,1}(\mathbb R\times [0,1])$, it holds
\begin{align*}
S^\al(\{\hat\nu_t,(f_t(x)+f_t(-x))/2\}_{0\leq t\leq 1})=S^\al(\{\hat\nu_t,f_t(x)\}_{0\leq t\leq 1}).
\end{align*}
Therefore,  the optimization problem \eqref{e:ratesingular} can be restricted to the set of even functions, i.e. $f_t(x)=f_t(-x)$.  If $\hat\nu_t$ is not symmetric, or $\hat\nu_0\neq \hat\mu_0$, we simply set $S_{\hat\mu_0}^\al(\{\hat\nu_t\}_{0\leq t\leq 1})=\infty$.

In this section we prove  Theorem \ref{main2}. We study the properties of the rate function \eqref{e:ratesingular} and its relation to rectangular free convolution in Sections \ref{s:rrt} and \ref{s:frconv}. In Section \ref{s:ito}, we derive dynamical equations of linear statistics of Dyson Bessel process for general test function using It\^o's formula. We prove the large deviations upper bound in Section \ref{s:ldup}, and the large deviations lower bound in Section \ref{s:lddown}. Finally we give the proof of Theorem \ref{main2} in Section \ref{s:pmain2}.

%

\subsection{Study of the Rate Function}\label{s:rrt}
In this section, we study the rate function $S_{\hat\mu_0}^\al(\{\hat\nu_t\}_{0\leq t\leq 1})$ as defined in \eqref{e:ratesingular}.  
If $S^\al_{\hat\mu_0}(\{\hat\nu_t\}_{0\leq t\leq 1})<\infty$, then $\hat\nu_0=\hat\mu_0$, and
by Riesz representation theorem, for any measure process $\{\hat\nu_t\}_{0\leq t\leq 1}$, there exists a measurable function $\del_x k_t(x)$, such that for any $f\in \cC_b^{2,1}([0,1]\times \bR)$,
\begin{align}\label{e:testfb}\begin{split}
\int_0^1 \int f'_s(x) \del_x k_s(x)\rd \hat\nu_s \rd s = L_1^\alpha(\{\hat\nu_t,f_t\}_{0\le t\le 1})&=
\hat\nu_1(f_1)-\hat\nu_0(f_0)-\int \del_s f_s(x) \rd\hat\nu_s(x)\rd s\\
&-\int_0^1\int f_s'(x)(H(\hat\nu_s)+(\alpha/2) H(\delta_0))\rd \hat\nu_s(x)\rd s.
\end{split}\end{align}
Since $\hat\nu_s$ is symmetric, it is necessary that $\del_x k_s(x)$ is an odd function.
With this notation, we can rewrite the rate function as
\begin{align}\begin{split}\label{e:testfb2}
S_{\hat\mu_0}^\al(\{\hat\nu_t\}_{0\leq t\leq 1})
&=\sup_{f\in \cC^{2,1}} \left\{\int_0^1 f'_s(x) \del_x k_s(x)\rd \hat\nu_s(x) \rd s-\frac{1}{8\beta}\int^1_0 \int (f_s'(x)-f_s'(-x))^2\rd \hat\nu_s(x)\rd s\right\}\\
&=\frac{\beta}{2}\int_0^1 \int \del_x k_s(x)^2\rd \hat\nu_s(x)\rd s,
\end{split}\end{align}
where the equality is achieved when $f'_t(x)-f'_t(-x)=2\beta \del_x k_t(x)$.

We next see as in Proposition \ref{p:rate} that  when the rate function $S_{\hat\mu_0}^\al(\{\hat\nu_t\}_{0\leq t\leq 1})$ is finite, the measure valued  process $\{\hat\nu_t\}_{0\leq t\leq 1}$ satisfies nice properties. 
\begin{proposition}\label{p:rate2}
Fix a symmetric probability density $\hat\mu_0$ with bounded second moment. $S^\al_{\hat\mu_0}$ is a good rate function on $\cC([0,1], \bM^s_1(\bR))$.  If $S^\al_{\hat\mu_0}(\{\hat\nu_t\}_{0\leq t\leq 1})$ is finite, and
 $\int \log|x| \rd\hat\nu_0$, and the free entropy $\Sigma(\hat\nu_0)$ are finite, 
then we have
\begin{enumerate}
\item There exists universal constant $\fC>0$ depends only on $\hat\mu_0$ and $S_{\hat\mu_0}(\{\hat\nu_t\}_{0\leq t\leq 1})$ such that the $L_2$ norms of $\hat\nu_t$ are uniformly bounded,
\begin{align}\label{e:L2norm2}
\int x^2 \rd \hat\nu_t(x)\leq \fC,
\end{align}
and if $\al>0$
\begin{align}\label{e:1x}
\int_0^1 \int\frac{\rd \hat\nu_t(x)}{x^2}\rd t\leq \fC,\quad \quad \int -\log |x|\rd \hat\nu_t \leq \fC. 
\end{align}

\item $\hat\nu_t$ has a density for almost  all $0\leq t\leq 1$, i.e. 
\begin{align*}
\frac{\rd \hat\nu_t(x)}{\rd x}=\hat\rho_t(x).
\end{align*}
\item We denote the velocity field $u_t=(H(\hat\nu_t)+(\al/2) H(\delta_0))+\del_x k_t$,
then it satisfies the conservation of mass equation 
\begin{align}\label{e:masseq2}
\del_t\hat\rho_t+\del_x(\hat\rho_t u_t)=0,\quad 0\leq t\leq 1,
\end{align}
in the sense of distribution. We can rewrite the dynamical entropy \eqref{e:ratesingular} as
\begin{align}\begin{split}\label{Sent2}
S_{\hat\mu_0}^\al(\{\hat\nu_t\}_{0\leq t\leq 1})&=\frac{\beta}{2}\left(
\int_0^1 \int u_s^2  \hat\rho_s(x)\rd x \rd s+\frac{\pi^2}{3}\int_0^1\int \hat \rho^3_s \rd s
+\frac{\alpha^2}{4}\int \frac{\hat\rho_s(x)}{x^2}\rd x\rd s\right.\\
&-\left.\left.\left(\Sigma(\hat\nu_t)+ \al \int \log |x|\rd\hat\nu_t(x)\right)\right|_{t=0}^1
\right).
\end{split}\end{align}
\end{enumerate} 
\end{proposition}

\begin{remark}\label{r:unsym}
Under assumptions of Proposition \ref{p:rate2}, if $\hat\mu_0$ and $\hat\nu_t$ are symmetrization of measures $\mu_0$ and $\nu_t$ supported on $[0,\infty)$ respectively,
\begin{align}
\hat\mu_0(x)=(\mu_0(x)+\mu_0(-x))/2,\quad \hat \nu_t(x)=(\nu_t(x)+\nu_t(-x))/2,\quad 0\leq t\leq 1,
\end{align}
then for almost all $0\leq t\leq 1$, we have $\nu_t(x)=\rho_t(x)\rd x$, and it satisfies $\hat\rho_t(x)=(\rho_t(x)+\rho_t(-x))/2$ and
\begin{align*}
\del_t\rho_t(x)+\del_x(u_t(x)\rho_t(x))=0.
\end{align*}
In particular we have for almost all $0\leq t\leq 1$, $\hat \rho_t(x)=(\rho_t(x)+\rho_t(-x))/2$.
With $\rho_t$, we can rewrite \eqref{Sent} as
\begin{align*}
S_{\hat\mu_0}^\al(\{\hat\nu_t\}_{0\leq t\leq 1})&=\frac{\beta}{2}\left(
\int_0^1 \int u_s^2 \rd \rho_s \rd s+\frac{\pi^2}{12}\int_0^1\int  \rho^3_s \rd s
+\frac{\alpha^2}{4}\int \frac{\rho_s(x)}{x^2}\rd x\rd s\right.\\
&-\left.\left.\left(\Sigma((\nu_t(x)+\nu_t(-x))/2)+ \al \int \log |x|\rd\nu_t(x)\right)\right|_{t=0}^1
\right)=:S_{\mu_0}^\al(\{\nu_t\}_{0\leq t\leq 1}).
\end{align*}

\end{remark}

\begin{proof} 
The claim that $S^\al_{\hat\mu_0}$ is a good rate function follows from essentially the same arguments as in \cite[Theorem 1.4]{GZ3}. The fact that $S^\al_{\hat\mu_0}$ is lower semi-continuous comes from the continuity of 
$\{\hat\nu_t\}_{0\leq t\leq 1}\rightarrow  S^\al(\{\hat\nu_t,f_t\}_{0\leq t\leq 1})$ for any $\cC^{2,1}_b(\mathbb R\times [0,1])$ function. In fact, the only difference lies in the new term
$$-\frac{\alpha}{2}\int_0^1\int \frac{f_s'(x)}{x} \rd \hat\nu_s(x)\rd s,$$
which are continuous since $\hat\nu_s$ is even we can rewrite it as
\begin{align*}
-\frac{\alpha}{4}\int_0^1\int \frac{f_s'(x)-f_s'(-x)}{x} \rd \hat\nu_s(x)\rd s,
\end{align*}
and it no long has a singularity at $x=0$.  If $S^\al_{\hat\mu_0}(\{\hat\nu_t\}_{0\leq t\leq 1})$ is finite, by definition we have $\hat\nu_0=\hat\mu_0$.  To check
that the level set $\{S^\al_{\hat\mu_0}\le M\}$ is included in a compact set, 
one first checks that $\int x^2 \rd\hat\nu_t(x) $ is uniformly bounded as in Proposition \ref{p:rate}. Moreover if $f$ is in $\cC_b^{2,1}([0,1], \bR)$, \eqref{e:testfb} implies that
\begin{eqnarray*}
|\hat\nu_t(f)-\hat\nu_s(f)|&\le& C \|f''\|_\infty |t-s|+\int_s^t \int |\del_x k_s||f'(x)|\rd\hat\nu_s(x) \d s\\
&\le & C \|f''\|_\infty |t-s|+\|f'\|_\infty \left(\int_0^1 \int \rd\hat\nu_s(x)\int |\del_x k_s|^2 \rd s\right)^{1/2} \sqrt{t-s}\\\
&\le&  C \|f''\|_\infty |t-s|+\|f'\|_\infty \left(2M\right)^{1/2} \sqrt{t-s}\,.\end{eqnarray*}
This implies that $t\rightarrow \hat\nu_t(f)$ is tight by Arzela-Ascoli theorem. The conclusion follows.

The estimate \eqref{e:L2norm2} can be proven in the same way as \eqref{e:L2norm}. In the following we prove \eqref{e:1x}. We take a test function $f(x)=-(\al\beta/4)\ln (\varepsilon +x^2)$ in \eqref{e:ratesingular}, 
\begin{align}\begin{split}\label{e:tets}
&+\infty>S_{\hat\mu_0}^\al(\{\hat\nu_t\}_{0\leq t\leq 1}) = \hat\nu_1(f)-\hat\nu_0(f)+\frac{\al\beta}{4}\int_0^1\int \frac{x}{\varepsilon+x^2}H(\hat\nu_s)\rd \hat\nu_s(x)\rd s\\
&+\frac{\alpha^2\beta}{4}\int_0^1\int \frac{1}{\varepsilon+x^2}\rd \hat\nu_s(x)\rd s -\frac{\alpha^2\beta}{8}\int^1_0 \int \left(\frac{x}{\varepsilon+x^2}\right)^2\rd \hat\nu_s\rd s\\
& \geq \hat\nu_1(f)-\hat\nu_0(f)+\frac{\al\beta}{4}\int_0^1\int \frac{x}{\varepsilon+x^2}H(\hat\nu_s)\rd \hat\nu_s(x)\rd s
+\frac{\alpha^2\beta}{8}\int_0^1\int \frac{1}{\varepsilon+x^2}\rd \hat\nu_s(x)\rd s.
\end{split}\end{align}

By our assumption on $\hat\mu_0$, it holds that $\hat\mu_0(\log |x|)>-\infty$. Using the fact that $\hat\nu_s$ is symmetric, i.e. $\hat\nu_s(y)=\hat\nu_s(-y)$, we can rewrite the first integral on the righthand side of \eqref{e:tets} as
\begin{align}\begin{split}\label{e:2t}
&\phantom{{}={}}\int_0^1\int \frac{x}{\varepsilon+x^2}H(\hat\nu_s)\rd \hat\nu_s(x)\rd s
=\int_0^1\int \frac{x}{\varepsilon+x^2}\int \frac{\rd\hat\nu_s(y)}{x-y}\rd \hat\nu_s(x)\rd s\\
&
=\frac{1}{2}\int_0^1\int\left(\frac{x}{\varepsilon+x^2} -\frac{y}{\varepsilon+y^2}\right)\frac{\rd\hat\nu_s(y)}{x-y}\rd \hat\nu_s(x)\rd s
=\frac{1}{2}\int_0^1\int\frac{\varepsilon -xy}{(\varepsilon+x^2)(\varepsilon+y^2)} \rd\hat\nu_s(y)\rd \hat\nu_s(x)\rd s
\\
&=\frac{\varepsilon}{2}\int_0^1\left(\int \frac{1}{\varepsilon+x^2} \rd \hat\nu_s(x)\right)^2\rd s\geq 0.
\end{split}\end{align}
By plugging \eqref{e:2t} into \eqref{e:tets}, and rearranging, we conclude that
there exists a constant $\fC$ depending only on $\rho_0$ and $S_{\hat\mu_0}(\{\hat\nu_t\}_{0\leq t\leq 1})$,  such that
\begin{align*}
\frac{\alpha^2\beta}{8}\int_0^1\int \frac{1}{\varepsilon+x^2}\rd \hat\nu_s(x)\rd s\leq \fC+\int \ln(\varepsilon +x^2)\rd \hat\nu_1(x)\lesssim \fC,
\end{align*}
where we used \eqref{e:L2norm2} for the last inequality. 
Moreover, we also have that $\int \ln(\varepsilon+x^2)\rd \hat\nu_t(x)\gtrsim -\fC$.
The claim \eqref{e:1x} follows by sending $\varepsilon$ to $0$.

For Item (ii), let $u_t=(H(\hat\nu_t)+(\alpha/2) H(\delta_0))+\del_x k_t(x)$, then \eqref{e:testfb2} implies that 
\begin{align}\label{e:tgo}
S^\al(\{\hat\nu_t\}_{0\leq t\leq 1})=\frac{\beta}{2}\int_0^1 \del_x k_s(x)^2\rd \hat\nu_s \rd s
=\frac{\beta}{2}\int_0^1 (u_s- H(\hat\nu_s)+(\alpha/2) H(\delta_0))^2\rd \hat\nu_s \rd s<+\infty.
\end{align}
In the following we show that 
\begin{align}\label{e:tgo2}
\int_0^1 \int(u_s- H(\hat\nu_s))^2\rd \hat\nu_s \rd s<+\infty.
\end{align}
If $\al=0$, then \eqref{e:tgo2} is the same as \eqref{e:tgo}. We assume $\al>0$, then we can write \eqref{e:tgo} as
\begin{align}\begin{split}\label{e:tgo3}
+\infty&>\int_0^1 \int\left((u_s- H(\hat\nu_s))^2
-(u_s- H(\hat\nu_s))\frac{\al}{x}
+\frac{\alpha^2}{4x^2}  
  \right)\rd \hat\nu_s \rd s\\
  &\geq \int_0^1 \left(\frac{1}{2}(u_s- H(\hat\nu_s))^2
-\frac{\alpha^2}{4x^2}  
  \right)\rd \hat\nu_s \rd s\\
  &\geq 
  \int_0^1 \frac{1}{2}(u_s- H(\hat\nu_s))^2\rd \hat\nu_s \rd s-\fC,
\end{split}\end{align}
where we used \eqref{e:1x} for the last inequality. The claim \eqref{e:tgo2} follows by rearranging \eqref{e:tgo3}.

The measure process $\hat\nu_t$ satisfies \eqref{e:tgo2}, which verifies the assumption in Proposition \ref{p:rate}. Item (ii) in Proposition \ref{p:rate} implies that $\hat\nu_t$ has a density for almost surely all $0\leq t\leq 1$, $\hat\nu_t=\hat\rho_t(x)\rd x$, and 
\eqref{e:masseq2} holds. Moreover, we have
\begin{align}\begin{split}\label{e:nee}
&\phantom{{}={}}\int_0^1 (u_s- H(\hat\nu_s))^2\rd \hat\nu_s \rd s\\
&=\int_0^1 u_s^2\hat\rho_s\rd x\rd s+\int_0^1\int H(\hat\rho_s)^2 \hat\rho_s \rd x \rd s-\left(\Sigma(\hat\nu_1)-\Sigma(\hat\nu_0)\right)\\
&=\int_0^1 u_s^2\hat\rho_s\rd x\rd s+\frac{\pi^2}{3}\int_0^1\int \hat\rho_s^3  \rd x \rd s-\left(\Sigma(\hat\nu_1)-\Sigma(\hat\nu_0)\right).
\end{split}\end{align}

In the following we prove \eqref{Sent2}. The same as in \eqref{e:tgo3}, we have
\begin{align}\begin{split}\label{e:newexp}
\int_0^1 \int \del_x k_s(x)^2 \hat\rho_s \rd x\rd s
=\int_0^1 \left((u_s- H(\hat\rho_s))^2
-(u_s- H(\hat\rho_s))\frac{\al}{x}
+\frac{\alpha^2}{4x^2}  
  \right) \hat\rho_s(x) \rd x \rd s.
\end{split}\end{align}
For the second term on the righthand side of \eqref{e:newexp}, we have
\begin{align}\label{e:newexp3}
\int_0^1(u_s- H(\hat\rho_s))\frac{\al}{x}\hat\rho_s\rd x\rd s
=\int_0^1\frac{\al}{x}u_s\hat\rho_s\rd x\rd s-\int_0^1\frac{\al}{x}H(\hat\rho_s)\hat\rho_s\rd x\rd s.
\end{align}
For the first term on the righthand side of \eqref{e:newexp3}, we have
\begin{align}\begin{split}\label{e:newexp4}
\del_s\int \log x\hat\rho_s \rd x 
&=\int \log |x|\del_s \hat\rho_s\rd x
=-\int \log |x|\del_x(u_s \hat\rho_s)\rd x
=\int \frac{u_s}{x} \hat\rho_s\rd x.
\end{split}\end{align}
For the second term on the righthand side of \eqref{e:newexp3}, we notice the following equality
\begin{align*}\begin{split}
\int H(\hat\rho_s)\frac{1}{x}\hat\rho_s(x)\rd x
&=\iint \frac{\hat\rho_s(y)}{x-y}\frac{\hat\rho_s(x)}{x}\rd x\rd y
=\frac{1}{2}\iint \left(\frac{1}{x(x-y)}-\frac{1}{y(x-y)}\right)\hat\rho_s(x)\hat\rho_s(y)\rd x\rd y\\
&=-\frac{1}{2}\iint \frac{1}{xy}\hat\rho_s(x)\hat\rho_s(y)\rd x\rd y=0,
\end{split}\end{align*}
where we used that $\hat\rho_s$ is symmetric, i.e. $\hat\rho_s(x)=\hat\rho_s(-x)$.

The estimates \eqref{e:nee} and \eqref{e:newexp} together give us the new formula of the dynamical entropy
\begin{align}\begin{split}\label{Sent20}
S_{\hat\mu_0}^\al(\{\hat\nu_t\}_{0\leq t\leq 1})&=\frac{\beta}{2}\left(
\int_0^1 \int u_s^2 \rd \hat\rho_s \rd s+\frac{\pi^2}{3}\int_0^1\int  \hat\rho^3_s \rd s
+\frac{\alpha^2}{4}\int \frac{\hat\rho_s(x)}{x^2}\rd x\rd s\right.\\
&-\left.\left.\left(\Sigma(\hat\nu_t)+ \al \int \log |x|\hat\nu_t(x)\rd x\right)\right|_{t=0}^1
\right).
\end{split}\end{align}
This finishes the proof of Item (iii). 
\end{proof}

\subsection{Free Rectangular Convolution}\label{s:frconv}
The minimizer of the dynamical entropy \eqref{Sent20} is characterized by the free rectangular convolution as introduced in \cite{benaych2009rectangular,belinschi2009regularization}. For any $\la=1/(1+\alpha)\in [0,1]$, the rectangular free convolution denoted by $\boxplus_\la$ can be defined in terms of the rectangular $R$-transform. 
For any symmetric measure $\hat\nu$ on $\bR$, its Stieltjes transform is given by
\begin{align*}
G_{\hat\nu}(z)=\int\frac{\rd \hat\nu(x)}{z-x}.
\end{align*}
The rectangular $R$-transform $C_{\hat\nu}(w)$  with ratio $\la$ of $\hat\nu$ is defined on a neighborhood of zero by
\begin{align}\label{e:Rtransform}
zG_{\hat\nu}(z)-1=C_{\hat\nu}(w),\quad w=G_{\hat\nu}(z)(\la G_{\hat\nu}(z)+(1-\la)/z).
\end{align}
When the symmetric measure $\hat\nu$ is the delta mass $\delta_0$ at zero, we have $G_{\delta_0}(z)=1/z$, and the $R$-transform $C_{\delta_0}(w)=0$. Let $W_n$ be a sequence of $n\times  m$ matrices with entries given by independent real/complex Gaussian random variables with mean zero and variance one, where $m\geq n$ and $n/m\rightarrow \la\in[0,1]$. Then the empirical eigenvalues of $(\sigma W_n/\sqrt n)(\sigma W_n/\sqrt n)^*$ converges to the rescaled Marchenko-Pastur law 
\begin{align*}
\mu_{\sigma W}=\frac{\sqrt{(\sigma^2(1/\sqrt\la+1)^2)-x)(x-\sigma^2(1/\sqrt\la-1)^2})}{\sqrt{2\pi}\sigma x},
\end{align*}
with the Stieltjes transform given by
\begin{align}\label{e:zm}
zm_{\mu_{\sigma W}}(z)-1=\frac{\sigma^2}{\la}m_{\mu_{\sigma W}}(z)(\la zm_{\mu_{\sigma W}}(z)+1-\la),\quad m_{\mu_{\sigma W}}=\int\frac{\rd\mu_{\sigma W}(x)}{z-x}.
\end{align}
We denote the limiting symmetrized singular value distribution of $\sigma W_n$ by $\hat\nu_{\sigma W}$, we call it the square root Marchenko-Pastur law. 
Then
\begin{align}\label{e:sMP}
G_{\hat\nu_{\sigma W}}(z)=\int\frac{\rd \hat\nu_{\sigma W}}{z-x}
=\frac{1}{2}\left(\int\frac{\rd \hat\nu_{\sigma W}}{z-x}+\int\frac{\rd \hat\nu_{\sigma W}}{z-x}\right)
=\int\frac{z\rd \hat\nu_{\sigma W}}{z^2-x^2}
=\int\frac{z\rd \mu_{\sigma W}}{z^2-x}=zm_{\mu_{\sigma W}}(z^2),
\end{align}
and we can rewrite the relation \eqref{e:zm} as
\begin{align*}
zG_{\hat\nu_{\sigma W}}(z)-1=\frac{\sigma^2}{\la}G_{\hat\nu_{\sigma W}}(z)(G_{\hat\nu_{\sigma W}}+(1-\la)/z).
\end{align*}
Comparing with the defining relation of rectangular $R$-transform \eqref{e:Rtransform}, we conclude that the rectangular $R$-transform of the square root Marchenko-Pastur law is given by 
\begin{align}\label{e:sqrtMP}
C_{\hat\nu_{\sigma W}}(w)=\frac{\sigma^2}{\la}w.
\end{align}

\begin{theorem}\label{t:ctransform}
Let $A_n, B_n\in \bR^{n\times m}$ and  $U\in \cO(n),V\in \cO(m)$ following Haar distribution over orthogonal group for $\beta=1$;  $A_n, B_n\in \bC^{n\times m}$ and $U\in \cU(n),V\in \cU(m)$ following Haar distribution over unitary group, for $\beta=2$,  where $m\geq n$ and $n/m\rightarrow \la\in[0,1]$. We assume that  the symmetrized empirical singular values  $\hat\nu_A^n$ and $\hat \nu_B^n$ of $A_n$ and $B_n$ converge to $\hat\nu_A$ and $\hat\nu_B$ respectively. Then the symmetrized empirical singular values of $A_n+U_n B_n V^*$ converges weakly in probability to $\hat\nu_{A\boxplus_\la B}$,  the free rectangular convolution
\begin{align*}
\hat\nu_{A\boxplus_\la B}=\hat\nu_A\boxplus_\la \hat\nu_B,
\end{align*}
which is characterized by
\begin{align}\label{e:freecon}
C_{\hat\nu_{A\boxplus_\la  B}}(w)=C_{\hat\nu_{A}}(w)+C_{\hat\nu_{B}}(w).
\end{align}
\end{theorem}
To use \eqref{e:freecon} to solve for the measure $\hat\nu_{A\boxplus_\la B}$, we need to solve  
\begin{align}\label{e:wz}
w=G_{\hat\nu_{A\boxplus_\la  B}}(z)(\la G_{\hat\nu_{A\boxplus_\la  B}}(z)+(1-\la)/z)=G_{\hat\nu_{A}}(z')(\la G_{\hat\nu_{A}}(z')+(1-\la)/z').
\end{align}
where $z,z'$ belongs to some neighborhood of infinity. We formally show below how to deduce a closed equation for $G_{\hat\nu_{A\boxplus_\la  B}}(z)$ from \eqref{e:wz} and \eqref{e:freecon} and leave the reader check that we can take  $z,z'$ in such a neighborhood.
We notice that the defining relation \eqref{e:Rtransform} gives that
\begin{align}\label{e:Gz'}
\frac{G_{\hat\nu_{A}}(z')}{z'}=\frac{1+C_{\hat\nu_{A}}(w)}{(z')^2}=\frac{w}{\la C_{\hat\nu_{A}}(w)+1},
\end{align}
and from \eqref{e:freecon}  we get
\begin{align}\label{e:Gz}
\frac{G_{\hat\nu_{A\boxplus_\la B}}(z)}{z}&=\frac{1+C_{\hat\nu_{A\boxplus_\la  B}}(w)}{z^2}=\frac{w}{\la C_{\hat\nu_{A\boxplus_\la  B}}(w)+1}\\
&=\frac{w}{\la C_{\hat\nu_{A}}(w)+1+\la C_{\hat\nu_{B}}(w)}
=\frac{1}{\frac{z'}{G_{\hat\nu_{A}}(z')}+\la \frac{C_{\hat\nu_{B}}(w)}{w}}.\nonumber
\end{align}
In particular we can rearrange \eqref{e:Gz'} and \eqref{e:Gz} as
\begin{align}\label{e:zz}
\frac{z'}{G_{\hat\nu_{A}}(z')}=\frac{z}{G_{\hat\nu_{A\boxplus_\la B}}(z)}-\frac{\la C_{\hat\nu_{B}}(w)}{w}.
\end{align}
Thanks to \eqref{e:wz}, we have the expression of $w$ in terms of $G_{\hat\nu_{A\boxplus_\la B}}(z)$ and $z$. We can then solve $z'$ using \eqref{e:wz} and \eqref{e:zz} in terms of $G_{\hat\nu_{A\boxplus_\la B}}(z)$ and $z$. Plugging them into \eqref{e:zz}, we  finally get a self-consistent equation for  $G_{\hat\nu_{A\boxplus_\la B}}(z)$ and $z$ which has a unique solution in a neighborhood of infinity, which determines $\hat\nu_{A\boxplus_\la B}$.

The empirical distribution of eigenvalues of large dimensional information-plus-noise type matrices \cite{dozier2007analysis,dozier2007empirical,bai2012no} can also be characterized by rectangular free convolution. This model is of particular interest because of its applications in statistics. 
The following Theorem is a special case of \cite{dozier2007analysis,dozier2007empirical}, which deals with more general noise.
\begin{theorem}\label{t:addfree}
Let $A_n$ be an sequence of $n\times m$ matrices  and $W_n$ be a sequence of $n\times  m$ matrices with entries given by independent real or complex Gaussian random variables with mean zero and variance $1/n$, where $m\geq n$ and $n/m\rightarrow \la\in[0,1]$.  If the eigenvalue distributions of $A_nA_n^*$ converge to $\mu_A$.  Then the empirical eigenvalue distributions of $(A_n+\sigma W_n)(A_n+\sigma W_n)^*$ converge to a deterministic measure $\mu_{A\boxplus_\la \sigma W}$ with Stieltjes transform $m(z)$ given by
\begin{align}\label{e:recst}
\int \frac{\rd \mu_A(x)}{(1-\sigma^2 m(z))\left((1-\lambda \sigma^2 m(z))z-(1-\la)\sigma^2)\right)-x}=\frac{m(z)}{1-\sigma^2 m(z)},\quad m(z)=\int \frac{\rd \mu_{A\boxplus_\la \sigma W}(x)}{z-x}.
\end{align}
The limit $\lim_{\eta\rightarrow 0+}\Im[m(x+\eta\rm i)]$ exists, it is analytic when it is positive and away from $0$,  and
\begin{align}\label{e:mbb}
|m(z)|\leq \left(\frac{1}{\sigma^2|z|}\right)^{1/2}.
\end{align}
\end{theorem}

We can reformulate Theorem \ref{t:addfree} in terms of rectangular $R$-transform.
We recall the rectangular $R$-transform of the square root Marchenko-Pastur law from \eqref{e:sqrtMP}
$C_{\hat\nu_{\sigma W}}(w)=\sigma^2w/\la$.
We denote the limiting symmetrized empirical singular value distribution of $A_n$ and $A_n+\sigma W_n$ as $\hat\nu_A$ and $\hat\nu_{A\boxplus_\la\sigma W}$ respectively, then \eqref{e:recst} is equivalent to
\begin{align*}
C_{\hat\nu_{A\boxplus_\la\sigma W}}(w)=C_{\hat\nu_A}(w)+\sigma^2w/\la.
\end{align*}
The Stieltjes transform $m(z)$ can be expressed in terms of the Stieltjes transform of $\hat\nu_{A\boxplus_\la\sigma W}$, and the bound \eqref{e:mbb} becomes
\begin{align}\label{e:mbb2}
\frac{1}{z}G_{\hat\nu_{A\boxplus_\la\sigma W}}(z)=m(z^2),\quad |G_{\hat\nu_{A\boxplus_\la\sigma W}}(z)|\leq \frac{1}{\sigma}.
\end{align}
By letting $z$ approach the support of $\hat\nu_{A\boxplus_\la\sigma W}$ in \eqref{e:mbb2}, we conclude that $\hat\nu_{A\boxplus_\la\sigma W}$ has a density bounded by $\OO(1/\sigma)$, and it is analytic on its support. 

For later purpose, we show that free convolution reduces the dynamical entropy.
\begin{lemma} \label{l:decrease}Let $\hat\nu\in \mathcal C([0,1], \mathbb M_1^s(\mathbb R))$ and $p\in \mathbb M_1^s(\mathbb R))$. Then for all $\lambda\ge 0$, 
$$S^\alpha_{\hat\nu_0\boxplus_\la p}(\{ \hat\nu_{t}\boxplus_\la p\}_{t\in [0,1]})\le S^\alpha(\{ \hat\nu_{t}\}_{t\in [0,1]})\,.$$
\end{lemma}

\begin{remark}\label{r:rectconv}
If we take $p$ to be the square root Marchenko-Pastur law on scale $\varepsilon$, i.e. $\sigma_\varepsilon=\hat\nu_{\varepsilon W}$ as in \eqref{e:sMP}, then Theorem \ref{t:addfree} implies
$\hat\nu_t^\varepsilon=\hat\nu_t\boxplus_\la \sigma_\varepsilon$ has an analytic density, which is bounded by $\OO(1/\varepsilon)$, and Lemma \ref{l:decrease} implies 
\begin{align*}
S^\alpha_{\hat\nu^\varepsilon_0}(\{ \hat\nu^\varepsilon_{t}\}_{t\in [0,1]})\le S_{\hat\mu_0}^\alpha(\{ \hat\nu_{t}\}_{t\in [0,1]})\,.
\end{align*}
\end{remark}

\begin{proof}
{
With the notation of rectangular free convolution, we can construct the limiting object of the matrix Brownian motions $A_n$ and $G(t)/\sqrt n$ from \eqref{e:defBB}, 
\begin{align*}
\sfa =\left(
\begin{array}{cc}
0 &a\\
a^* & 0
\end{array}
\right),
\quad
\sfg(t)=\left(
\begin{array}{cc}
0 &g(t)\\
g(t)^* & 0
\end{array}
\right),
\end{align*}
where $g(t)$ and $a$ are the limit in $*$-moments of $G(t)/\sqrt n$ and $A_n$ respectively. Then $\sfa$ and $\sfg(t)$ are free and $\sfg(t)$ has independent increment in terms of rectangular free convolution, with $\la=1/(1+\al)$.

For any measurable odd function $\{\partial_x k_t(x)\}_{0\leq t\leq 1}$, we consider the following noncommutative stochastic process,
\begin{align*}
\rd\sfh(t)=\rd \sfg(t)+\partial_x k_t(\sfh(t))\rd t,\quad \sfh(0)=\sfa.
\end{align*}
Then the nonzero part of the spectral measure $\hat \rho_t$ of $\sfh(t)$, satisfies the equation \eqref{e:masseq2}
\begin{align}\label{e:masseq3copy}
\del\hat \rho_t+\del_x((H(\hat\rho_t)+\al H(\delta_0)/2+\del_x k_t)\hat\rho_t)=0.
\end{align}
Then for any symmetric probability measure $\hat p$ on $\bR$, the process $\hat\rho^p_t=\hat\rho_t\boxplus_\la \hat p$ satisfies the same differential equation but with $\del_x k_t$ replaced by $\del_x k^p_t$ defined by, if $\sfh(t)$ and $\sfp$ are two free random variables with respect to the rectangular free convolution in a non-commutative probability space $\tau$ with distribution $\hat\rho_t$ and $\hat p_t$ respectively, then
\begin{align*}
\del_x k^p_t(x)=\tau(\del_x k_t(\sfh(t))|\sfh(t)+\sfp).
\end{align*}
As a consequence, we find 
\begin{align*}
S^\al(\{\hat\rho_t\boxplus_\la \hat p\}_{0\leq t\leq 1})
&=\frac{\beta}{2}\int_0^1 \hat\rho_s\boxplus_\la \hat p(\del_x( k^p_s)^2)\rd s
=\frac{\beta}{2}\int_0^1 \tau(\tau(\del_x k_s(\sfh(t))|\sfh(t)+\sfp))^2) \rd s\\
&\leq \frac{\beta}{2}\int_0^1 \tau(\del_x k_s(\sfh(t))^2) \rd s
=\frac{\beta}{2}\int_0^1 \int \del_x k_s(x)^2\hat\rho_s\rd s
=S^\al(\{\hat\rho_t\}_{0\leq t\leq 1}).
\end{align*}

}

\end{proof}
\subsection{It\^o's calculus for the Dyson Bessel process}\label{s:ito}
In this section, we derive dynamical equations of linear statistics of Dyson Bessel process for general test functions using It\^o's formula. Take any test function $f_t(x)\in \cC_b^{2,1}([0,1]\times \bR)$.  Using to It{\^o}'s lemma, \eqref{e:dsk2} gives
\begin{align}\begin{split}\label{e:lin}
&\phantom{{}={}}\rd\sum_{i\in\qq{-n,n}\setminus \{0\}}f_t(s_i(t))
=\sum_{i\in\qq{-n,n}\setminus \{0\}}f'_t(s_i(t))\rd s_i(t)+\del_t f_t(s_i(t))\rd t+\frac{f''_t(s_i(t))}{2\beta n}\rd t\\
&=\rd L^f_t
+
\frac{1}{4n}\sum_{i\neq j\in\qq{-n,n}\setminus \{0\}}\frac{f'_t(s_i(t))-f'_t(s_j(t))}{s_i(t)-s_j(t)}\d t+\left(\alpha_n-\frac{1}{2n}\right) \sum_{i\in\qq{-n,n}\setminus \{0\}} \frac{f_t'(s_i(t))}{2s_i(t)}\rd t\\
&+\sum_{i\in\qq{-n,n}\setminus \{0\}}\del_t f_t(s_i(t))\rd t+\sum_{i\in\qq{-n,n}\setminus \{0\}}\frac{f''_t(s_i(t))}{2\beta n}\rd t,
\end{split}\end{align}
where the martingale term is given by
\begin{align}\label{e:MLt}
\rd L^f_t=\sum_{i\in\qq{-n,n}\setminus \{0\}}f'_t(s_i(t))\frac{\rd W_i(t)}{\sqrt {\beta n}},
\end{align}
for ${i\in\qq{-n,n}\setminus \{0\}}$.

We recall the empirical particle density $\{\hat\nu^n_t\}_{0\leq t\leq 1}$ from \eqref{e:density}. With it, we can rewrite \eqref{e:lin} as
\begin{align}\begin{split}&\phantom{{}={}}\int f_t(x)\rd \hat\nu^n_t-\int f_t(x)\rd \hat\nu^n_0\\
&=\frac{L^f_t}{2n}
+\frac{1}{2}\int_0^t\int_{x\neq y} \frac{f_s'(x)-f_s'(y)}{x-y}\rd \hat\nu^n_s(x)\rd\hat\nu^n_s(y)\rd s+\left(\al_n-\frac{1}{2n}\right) \int_0^t\int \frac{f_s'(x)}{2x}\rd \hat\nu^n_s(x)\rd s\\
&+\frac{1}{2\beta n}\int_0^t \int (f''_s(x))\rd \hat\nu^n_s(x)\rd s+\int_0^t\int \del_s f_s(x) \rd\hat\nu^n_s(x)\rd s\\
&=\frac{ L^f_t}{2n}
+\frac{1}{2}\int_0^t\int  \frac{f_s'(x)-f_s'(y)}{x-y}\rd \hat\nu^n_s(x)\rd\hat\nu^n_s(y)\rd s+\left(\al_n-\frac{1}{2n}\right) \int_0^t\int \frac{f_s'(x)}{2x}\rd \hat\nu^n_s(x)\rd s\\
&+\left(\frac{1}{2\beta n}-\frac{1}{4n}\right)\int_0^t\int f''_s(x)\rd \hat\nu^n_s(x)\rd s+\int_0^t\int \del_s f_s(x) \rd\hat\nu^n_s(x)\rd s\\
&=\frac{ L^f_t}{2n}+\int_0^t\int \del_s f_s(x) \rd\hat\nu^n_s(x)\rd s+\int_0^t\int f_s'(x)H(\hat\nu^n_s)\rd \hat\nu^n_s(x)\rd s\\
&+\left(\frac{\alpha_n}{2}-\frac{1}{4n}\right) \int_0^t\int f_s'(x)H(\delta_0)\rd \hat\nu^n_s(x)\rd s
+\frac{1}{n}\left(\frac{1}{2\beta }-\frac{1}{4}\right)\int_0^t\int f''_t(x)\rd \hat\nu^n_s(x)\rd s.\\
&=\frac{ L^f_t}{2n}+\int_0^t \int \del_s f_s(x) \rd\hat\nu^n_s(x)\rd s+\int_0^t\int f_s'(x)H(\hat\nu^n_s)\rd \hat\nu^n_s(x)\rd s\\
&+\frac{\al}{2} \int_0^t\int f_s'(x)H(\delta_0)\rd \hat\nu^n_s(x)\rd s +\varepsilon_t^n(\{f_s\}_{0\leq s\leq 1}),\label{e:df22}
\end{split}\end{align}
where $$\varepsilon^n_t(\{f_s\}_{0\leq s\leq t})=\frac{1}{n}\left(\frac{1}{2\beta }-\frac{1}{4}\right)\int_0^t\int f''_t(x)\rd \hat\nu^n_s(x)\rd s +\left( \frac{\alpha_n-\alpha}{2}-\frac{1}{4n}\right) \int_0^t\int f_s'(x)H(\delta_0)\rd \hat\nu^n_s(x)\rd s,$$ goes to zero uniformly when $\sup_t\|f''_t\|_\infty$ is finite. 
Since $s_{-i}(t)=-s_{i}(t)$, we can rewrite the martingale $\{L_t^f\}_{0\leq t\leq 1}$ from \eqref{e:df22} as
\begin{equation}\label{martdef}L^f_t=\sum_{i=1}^n \frac{1}{\sqrt{\beta n}} \int_0^t (f'(s_i(u)) - f'(-s_i(u)))\d W_i(u),\end{equation}
and its quadratic variation is given by, 
\begin{align}\label{brLf}
\langle L^f, L^f\rangle_t=\frac{1}{\beta n}\sum_{i=1}^n  \int_0^t (f'(s_i(u)) - f'(-s_i(u)))^2 \rd
u=
\frac{1}{\beta} \int_0^t  \int (f'(x)-f'(-x))^2 \d\hat\nu^n_u(x)\rd u,
\end{align}
where we used that the measure $\hat\nu_u^n$ is symmetric.
We can construct an exponential martingale using the martingale $L^f_t$ from \eqref{martdef}
\begin{align}\label{e:Dt}
D_t=e^{\frac{n}{2} L^f_t-\frac{n^2}{8}\langle L^f,L^f\rangle_t},\quad \bE[D_t]=\bE[D_0]=1.
\end{align}
We set
\begin{align}\label{e:defSn2}
S^{\al_n,n}(\{\hat\nu_t^n, f_t\}_{0\leq t\leq 1})
&=\frac{1}{n^2}\left(\frac{nL^f_1}{2}-\frac{n^2}{8}\langle L^f, L^f\rangle_1\right).
\end{align}
Then for $\{\hat\nu^n_t\}_{0\leq t\leq 1}\in \bB(\{\hat\nu_t\}_{0\leq t\leq 1}, \delta)$, we have by uniform (in $n\ge 1$) continuity of $\{\hat\nu_t\}_{0\leq t\leq 1}\mapsto S^n(\{\hat\nu_t, f_t\}_{0\leq t\leq 1})$ and the convergence of $\alpha_n$ to $\alpha$, that for any $f\in \cC_b^{2,1}([0,1]\times \bR)$, 
\begin{align}\begin{split}\label{e:sanb}
S^{\al_n,n}(\{\hat\nu_t^n, f_t\}_{0\leq t\leq 1})
&= \left(L^\alpha_1(\{\hat\nu^n_t,f_t\}_{0\leq t\leq 1})-\frac{1}{8\beta} \int_0^1  \int (f'(x)-f'(-x))^2 d\hat\nu^n_u(x)\rd u\right)+\varepsilon^n_1(\{f_t\}_{0\leq t\leq 1})\\
&=S^\al (\{\hat\nu_t, f_t\}_{0\leq t\leq 1})+\oo_{n,\delta}(1).
\end{split}\end{align}

%
%
%
%

\subsection{Large deviations upper bound}\label{s:ldup}
In this section we prove the large deviations upper bound of Theorem \ref{main2}
\begin{align}\label{e:DBMupbb2}
&\limsup_{\delta\rightarrow 0}\limsup_{n\rightarrow\infty}\frac{1}{n^2}\log \bP(\{\hat\nu^n_t\}_{0\leq t\leq 1}\in \bB(\{\hat\nu_t\}_{0\leq t\leq 1}, \delta))
\leq -S^\al_{\hat\mu_0}(\{\hat\nu_t\}_{0\leq t\leq 1}),
\end{align}
by tilting the measure using the  exponential Martingale \eqref{e:Dt}. Indeed, by using \eqref{e:sanb} uniformly on 
 $\{\hat\nu^n_t\}_{0\leq t\leq 1}\in \bB(\{\hat\nu_t\}_{0\leq t\leq 1}, \delta)$, the large deviations upper bound follows  from
\begin{align}\begin{split}\label{e:uppha}
&\phantom{{}={}}\bP(\{\hat\nu^n_s\}_{0\leq t\leq 1}\in \bB(\{\hat\nu_s\}_{0\leq t\leq 1}, \delta))
=\bE\left[\bm1(\{\hat\nu^n_s\}_{0\leq t\leq 1}\in \bB(\{\hat\nu_s\}_{0\leq t\leq 1}, \delta))\frac{e^{n^2S^{\al_n,n}(\{\hat\nu_t^n, f_t\}_{0\leq t\leq 1})}}{e^{n^2S^{\al_n,n}(\{\hat\nu_t^n, f_t\}_{0\leq t\leq 1})}}\right]\\
&=\bE\left[\bm1(\{\hat\nu^n_s\}_{0\leq t\leq 1}\in \bB(\{\hat\nu_s\}_{0\leq t\leq 1}, \delta))e^{n^2S^{\al_n,n}(\{\hat\nu_t^n, f_t\}_{0\leq t\leq 1})}\right]\frac{e^{\oo(n^2)}}{e^{n^2S^{\alpha}(\{\hat\nu_t, f_t\}_{0\leq t\leq 1})}}\\
&\leq \bE\left[e^{n^2S^{\al_n,n}(\{\hat\nu_t^n, f_t\}_{0\leq t\leq 1})}\right]\frac{e^{\oo(n^2)}}{e^{n^2S^{\alpha}(\{\hat\nu_t, f_t\}_{0\leq t\leq 1})}}
=e^{-n^2(S^{\alpha}(\{\hat\nu_t, f_t\}_{0\leq t\leq 1})+\oo_{n,\delta}(1))}.
\end{split}\end{align}
The large deviations upper bound \eqref{e:DBMupbb2} follows from rearranging \eqref{e:uppha}, and taking infimum over functions $f\in \cC^{2,1}_b$.

\subsection{large deviations lower bound}\label{s:lddown}
%
%
%

In this section we prove the large deviations lower bound of Theorem \ref{main2}
\begin{align}\label{e:DBMlower2}
&\limsup_{\delta\rightarrow 0}\limsup_{n\rightarrow\infty}\frac{1}{n^2}\log \bP(\{\hat\nu^n_t\}_{0\leq t\leq 1}\in \bB(\{\hat \nu_t\}_{0\leq t\leq 1}, \delta))
\geq -S^\al_{\hat\mu_0}(\{\hat \nu_t\}_{0\leq t\leq 1}),
\end{align}
using the large deviations lower bound of Dyson  Brownian motion Theorem \ref{t:DBMLDP}, and the change of measure Proposition \ref{p:changem}.
If $S^\al_{\hat\mu_0}(\{\hat \nu_t\}_{0\leq t\leq 1})=+\infty$, there is nothing to prove. Otherwise, we prove that there exists a symmetric probability measure in the form $(\tilde \nu_t(x)+\tilde \nu_t(-x))/2$, such that 
$\tilde \nu_t(x)$ is supported on $[\fa,\infty)$, with $\fa>0$, has a smooth density $\rd\tilde \nu_t(x)=\tilde \rho_t(x)\rd x$, and 
\begin{align}\label{e:bor}
S^\al_{(\tilde \nu_0(x)+\tilde \nu_0(-x))/2}(\{(\tilde \nu_t(x)+\tilde \nu_t(-x))/2\}_{0\leq t\leq 1})\leq S^\al_{\hat\mu_0}(\{\hat\nu_t\}_{0\leq t\leq 1})+\oo(1).
\end{align}

Next we construct $\tilde \nu_t(x)$ in \eqref{e:bor}. Take small $\varepsilon>0$, thanks to Remark \ref{r:rectconv}, let $\sigma_\varepsilon=\hat\nu_{\varepsilon W}$ as in \eqref{e:sMP}, then Theorem \ref{t:addfree} implies
$\hat\nu_t^{(1)}=\hat\nu_t\boxplus_\la \sigma_\varepsilon$ has an analytic density on its support, which is bounded by $\OO(1/\varepsilon)$ by  \eqref{e:mbb2}. In particular $\hat\nu_t^{(1)}$ has no atom at the origin. 
Moreover,  Lemma \ref{l:decrease} implies 
\begin{align*}
S^\alpha_{\hat\nu^{(1)}_0}(\{ \hat\nu^{(1)}_{t}\}_{t\in [0,1]})\le S_{\hat\mu_0}^\alpha(\{ \hat\nu_{t}\}_{t\in [0,1]})\,.
\end{align*}
Let $\hat \nu^{(1)}_t(x)=\hat\rho^{(1)}_t(x)\rd x$ for $0\leq t\leq 1$. Then $\hat\rho^{(1)}_t$ is a symmetric measure, with analytic density on its support. We denote $\rho^{(1)}_t=2\hat\rho^{(1)}_t|_{[0,\infty)}$, then $\hat\rho^{(1)}_t(x)=(\rho^{(1)}_t(x)+\rho^{(1)}_t(-x))/2$. Let $u_t^{(1)}$ be the weak solution of $\del_t\rho_t^{(1)}+\del_x(\rho_t^{(1)}u_t^{(1)})=0$. Thanks to Remark \ref{r:unsym}, we can rewrite the dynamical rate function as
\begin{align}\label{e:largeupb2}\begin{split}
S_{\hat\mu^{(1)}_0}^\alpha(\{\hat \nu^{(1)}_t\}_{0\leq t\leq 1})&=\frac{\beta}{2}\left(
\int_0^1 \int (u^{(1)}_s)^2 \rd \nu^{(1)}_s \rd s+\frac{\pi^2}{12}\int_0^1\int  (\rho_s^{(1)})^3 \rd s
+\frac{\alpha^2}{4}\int \frac{ \rho^{(1)}_s(x)}{x^2}\rd x\rd s\right.\\
&-\left.\left.\left(\Sigma(\hat \nu^{(1)}_t)+ \al \int \log |x|\rd\hat\nu^{(1)}_t(x)\right)\right|_{t=0}^1
\right)=S_{\nu^{(1)}_0}^\alpha(\{\nu^{(1)}_t\}_{0\leq t\leq 1}).
\end{split}\end{align}

For any small number $\fa>0$, we denote 
$\nu^{(2)}_t$ the probability obtained from shifting $\nu^{(1)}_t$ to the right by $2\fa$, and corresponding $\rho^{(2)}_t, u^{(2)}_t$:
\begin{align*}
\rho^{(2)}_t(x)=\rho^{(1)}_t(x-2\fa), \quad \del_t \rho^{(2)}_t+\del_x(\rho^{(2)}_t  u^{(2)}_t)=0.
\end{align*}
Then $u^{(2)}_t(x)=u_t(x-2\fa)$, and it is easy to see from \eqref{e:largeupb2} that 
\begin{align*}
S_{\nu^{(2)}_0}^\alpha(\{ \nu^{(2)}_t\}_{0\leq t\leq 1})\leq  S_{\nu^{(1)}_0}^\alpha(\{\nu^{(1)}_t\}_{0\leq t\leq 1})+\oo_\fa(1).
\end{align*}
Since the support of $\nu^{(2)}_t$ is on $[2\fa, +\infty)$, $S_{\nu^{(2)}_0}^\alpha(\{ \nu^{(2)}_t\}_{0\leq t\leq 1})$ no longer have a singularity at $0$. 

Thanks to Proposition \ref{p:approximation}, we can further approximate $\{\nu^{(2)}_t\}_{0\leq t\leq 1}$ by a sequence of measure-valued processes $\{\nu^\varepsilon_t(x)\}_{0\leq t\leq 1}$, such that 
\begin{align*}
\lim_{\varepsilon\rightarrow 0} \sup_{0\leq t\leq 1}d(\nu^{(2)}_t, \nu^\varepsilon_t )=0.
\end{align*}
which satisfies the properties of the Proposition. 
For $\varepsilon>0$ small enough, we can construct $\nu^\varepsilon_t$ such that it is supported on $[\fa, +\infty)$, away from $0$. Moreover, 
\begin{align*}
S^\al_{\nu_0^\varepsilon}(\{\nu_t^\varepsilon\}_{0\leq t\leq 1})&\leq S_{\nu^{(2)}_0}^\alpha(\{\nu^{(2)}_t\}_{0\leq t\leq 1})+\oo_\varepsilon(1)\\
&\leq  S_{\nu^{(1)}_0}^\alpha(\{\nu^{(1)}_t\}_{0\leq t\leq 1})+\oo_{\fa,\varepsilon}(1)\leq S_{\hat\mu_0}^\alpha(\{ \hat\nu_{t}\}_{t\in [0,1]})+\oo_{\fa,\varepsilon}(1).
\end{align*}
Moreover, for $\varepsilon,\fa$ sufficiently small, from the construction, we have $d((\nu^{\varepsilon}_t(x)+\nu^{\varepsilon}_t(-x))\rd x/2, \hat\nu_t)\leq \delta/3$.

We take $\{\tilde \nu_t\}_{0\leq t\leq 1}$ as  $\{\nu^\varepsilon_t\}_{0\leq t\leq 1}$  with sufficiently small $\varepsilon,\fa$, with $d((\tilde \rho_t(x)+\tilde \rho_t(-x))\rd x/2, \hat\nu_t)\leq \delta/3$. 
Next, we construct a new family of initial data, using the $1/n$ quantiles of $\tilde \nu_0$, 
\begin{align*}
\tilde\nu_0^n=\frac{1}{n}\sum_{i=1}^n \delta_{\tilde s_i(0)},\quad \tilde s_i(0)=\tilde F_0^{-1}((i-1/2)/n), \quad F_0(x) =\int^x \tilde \rho_0(y)\rd y.
\end{align*}
In this way $\tilde s_i(0)$ are the $1/n$ quantiles of $\tilde \rho_0$. From our construction and Assumption \ref{a:mu0}, we have that 
\begin{align*}
d(\tilde \nu^n_0, \nu^n_0)&=\sqrt{\frac{1}{n}\sum_{1\leq i\leq n} |\tilde s_i(0)-s_i(0)|^2}\\
&\leq d((\tilde \rho_0(x)+\tilde \rho_0(-x))\rd x/2, \hat\mu_0)+d(\hat\nu_0^n, \hat\mu_0)+d(\tilde \nu_0^n, \tilde \rho_0)\leq \frac{\delta}{3}+\oo_n(1).
\end{align*}
We consider the Dyson Bessel process starting from $\tilde \nu_0^n$,
\begin{align}\begin{split}\label{e:modifydsk}
\d \tilde s_i(t) =\frac{\d W_{k}(t)}{\sqrt{\beta n}}+\frac{1}{2n}\sum_{j:j\neq i}\frac{1}{\tilde s_i(t)-\tilde s_j(t)}+\left(\frac{1}{2n}\sum_{j:j\neq i}\frac{1}{\tilde s_i(t)+\tilde s_j(t)}+\frac{\alpha_n}{2 \tilde s_i(t)}\right)\d t,\quad 1\leq i\leq n.
\end{split}\end{align}
which shares the same Brownian motions $W_i$  as \eqref{e:dsk}. We denote its particle density as
\begin{align*}
\tilde \nu_t^n=\frac{1}{n}\sum_{i=1}^n \delta_{\tilde s_i(t)}.
\end{align*}

By taking the difference between \eqref{e:dsk} and \eqref{e:modifydsk}, we get
\begin{eqnarray}
\del_t (\tilde s_i(t)-s_i(t))^2&=&\frac{(\tilde s_i(t)-s_i(t))}{n}\sum_{j:j \neq i}\frac{-(\tilde s_i(t)-s_i(t))+(\tilde s_j(t)-s_j(t))}{(\tilde s_i(t)-\tilde s_j(t))( s_i(t)- s_j(t))}\label{e:difeqs}\\
&+&\frac{(\tilde s_i(t)-s_i(t))}{n}\sum_{j:j \neq i}\frac{-(\tilde s_i(t)-s_i(t))-(\tilde s_j(t)-s_j(t))}{(\tilde s_i(t)+\tilde s_j(t))( s_i(t)+s_j(t))}-\frac{\al_n(\tilde s_i(t)-s_i(t))^2}{\tilde s_i(t)s_i(t)}.\nonumber \end{eqnarray}
Averaging over all the indices $i\in \qq{n}$, we get 
\begin{align}\begin{split}\label{e:difeqs2}
&\frac{1}{n}\del_t\sum_i (\tilde s_i(t)-s_i(t))^2=-\frac{1}{n^2}\sum_{i<j}\frac{((\tilde s_i(t)-s_i(t))-(\tilde s_j(t)-s_j(t)))^2}{(\tilde s_i(t)-\tilde s_j(t))(s_i(t)- s_j(t))}\\
&-\frac{1}{n^2}\sum_{i<j}\frac{\left((\tilde s_i(t)-s_i(t))+(\tilde s_j(t)-s_j(t))\right)^2}{(\tilde s_i(t)+\tilde s_j(t))( s_i(t)+s_j(t))}-\frac{\al_n}{n}\sum_i \frac{(\tilde s_i(t)-s_i(t))^2}{\tilde s_i(t)s_i(t)}\leq 0,
\end{split}\end{align}
where the first two terms on the righthand side of \eqref{e:difeqs2} are negative; For the last term, we used our assumption, either $\al_n=0$ the last term in \eqref{e:difeqs2} vanishes; or $\al_n\geq 1/n\beta$ and $\tilde s_i(t), s_i(t)>0$, the last term in \eqref{e:difeqs2} is nonpositive.
It follows that
\begin{align*}
d(\tilde \nu^n_t, \nu^n_t)=\sqrt{\frac{1}{n}\sum_i (\tilde s_i(t)-s_i(t))^2}\leq \sqrt{\frac{1}{n}\sum_i (\tilde s_i(0)-s_i(0))^2}\leq \delta/2,
\end{align*}
provided $n$ is large enough.
From our construction, we have for $0\leq t\leq 1$ $\supp \tilde \nu_t\in [\fa, \infty)$, then
\begin{align*}\begin{split}
&\phantom{{}={}}\bP(\{\hat\nu^{n}_t\}_{0\leq t\leq 1}\in \bB(\{\hat \nu_t\}_{0\leq t\leq 1}, \delta))
\geq \bP(\{\tilde \nu^{n}_t\}_{0\leq t\leq 1}\in \bB(\{\tilde \nu_t\}_{0\leq t\leq 1}, \delta/2))\\
&\geq \bP(\{\tilde \nu^{n}_t\}_{0\leq t\leq 1}\in \bB(\{\tilde \nu_t\}_{0\leq t\leq 1}, \delta/2), \tilde s_n(t)\geq \fa, 0\leq t\leq 1)\\
&=\bP^\fa(\{\tilde \nu^{n}_t\}_{0\leq t\leq 1}\in \bB(\{\tilde \nu_t\}_{0\leq t\leq 1}, \delta/2), \tilde s_n(t)\geq \fa, 0\leq t\leq 1),
\end{split}\end{align*}
where $\bP^\fa$ is defined in Proposition \ref{p:changem}. 
We can use Proposition \ref{p:changem} to rewrite the law of Dyson Bessel process in term of the law of Dyson Brownian motion. Moreover, on the event $\{\tilde s_n(t)\geq \fa, 0\leq t\leq 1\}$, the stopping time $\tau_\fa=1$, and we can ignore the stopping time. Thus we have
\begin{align}\begin{split}\label{e:abt}
&\phantom{{}={}}\bP(\{\nu^{n}_t\}_{0\leq t\leq 1}\in \bB(\{\nu_t\}_{0\leq t\leq 1}, \delta))
\geq \bP^\fa(\{\tilde \nu^{n}_t\}_{0\leq t\leq 1}\in \bB(\{\tilde \rho_t\}_{0\leq t\leq 1}, \delta/2), \tilde s_n(t)\geq \fa, 0\leq t\leq 1)\\
&=\bE_\bQ\left[ e^{L_{1}-\frac{1}{2}\langle L, L\rangle_{1}}\bm1(\{\tilde \nu^{n}_t\}_{0\leq t\leq 1}\in \bB(\{\tilde \rho_t\}_{0\leq t\leq 1}, \delta/2), \tilde s_n(t)\geq \fa, 0\leq t\leq 1)\right],
\end{split}\end{align}
where the exponential martingale is from Proposition \ref{p:changem}
\begin{align}\begin{split}\label{e:haha}
L_{1}-\frac{1}{2}\langle L, L\rangle_{1}
&=\left.\theta(s_1({t}), \cdots, s_n({t}))\right|_0^{1}- \frac{\beta n}{2}
\int_0^{1}\sum_i \frac{\alpha_n^2}{4s_i^2(t)}\rd t
\\
&-\left(\frac{\beta}{2}-1\right)\int_0^{1}\frac{1}{4n}\sum_{i\neq j}\frac{\rd t}{(s_i(t)+s_j(t))^2}
+\int_0^{1}\frac{\alpha_n}{4}\sum_i \frac{\rd t}{s_i^2(t)}.
\end{split}\end{align}
On the event that $\{\tilde \nu^{n}_t\}_{0\leq t\leq 1}\in \bB(\{\tilde \rho_t\}_{0\leq t\leq 1}, \delta/2), \tilde s_n(t)\geq \fa, 0\leq t\leq 1$, we can rewrite \eqref{e:haha} as
\begin{align}\begin{split}\label{e:abb}
\frac{1}{n^2}\left(L_{1}-\frac{1}{2}\langle L, L\rangle_{1}\right)
&=\left. \frac{\beta}{2}\left(\frac{1}{2}\int \log|x+y|\tilde \rho_t(x)\tilde \rho_t(y)\rd x\rd y+\alpha \int \log |x|\tilde \rho_t(x)\rd x\right)\right|_0^1\\
&-\frac{\beta \al^2}{8}\int_0^1\int \frac{\tilde \rho_t(x)\rd x}{x^2}\rd t
+\oo_{\delta, \fa}(1)+\OO(|\al-\al_n|).
\end{split}\end{align}

%
%

In the following we prove that 
\begin{align}\label{e:signn}
\bQ(\{\tilde \nu^{n}_t\}_{0\leq t\leq 1}\in \bB(\{\tilde \nu_t\}_{0\leq t\leq 1}, \delta/2), \tilde s_n(t)\geq \fa, 0\leq t\leq 1)
=\exp\left\{-n^2S_{\tilde \nu_0}(\{\tilde \rho_t\}_{0\leq t\leq 1})+\oo(n^2)\right\},
\end{align}
where the rate function $S_{\tilde \nu_0}$ is from \eqref{e:rateSmu}.
Then \eqref{e:abt}, \eqref{e:haha}, \eqref{e:abb} and \eqref{e:signn} together imply
\begin{align*}\begin{split}
&\bP(\{\nu^{n}_t\}_{0\leq t\leq 1}\in \bB(\{\nu_t\}_{0\leq t\leq 1}, \delta))
\geq e^{-n^2(S_{\tilde \nu_0}(\{\tilde \nu_t\}_{0\leq t\leq 1})+\oo_n(1))}\\
&\quad \times e^{\frac{\beta n^2}{2}\left(\left.\left( \frac{1}{2}\int \log|x+y|\rd \tilde \nu_t(x)\rd \tilde \nu_t(y)+\alpha \int \log |x|\rd \tilde \nu_t(x)\right)\right|_0^1-\frac{ \al^2}{4}\int_0^1\int \frac{1}{x^2}\rd\tilde \rho_t(x)\rd t
+\oo_{\delta, \fa}(1)+\OO(|\al-\al_n|)\right)}\\
&=e^{-n^2(S_{\tilde \nu_0}^\alpha(\{\tilde \nu_t\}_{0\leq t\leq 1})+\oo_{\fa,\varepsilon}(1))}
\geq e^{-n^2(S_{\hat\mu_0}^\alpha(\{\hat\nu_t\}_{0\leq t\leq 1})+\oo_{\fa,\varepsilon}(1))},
\end{split}\end{align*}
which gives the large deviations lower bound \eqref{e:DBMlower2} by taking $\fa,\varepsilon, \delta\rightarrow 0$.

The estimate \eqref{e:signn} can be proven essentially the same as  \eqref{e:replace2}. 
Let $\bQ^{\beta \tilde k}$ be the law of 
\begin{align*}
\bQ^{\beta\tilde k}=e^{n^2S^{n}(\{\tilde \nu^n_t, \beta\tilde k_t)\}_{0\leq t\leq 1}} \bQ,
\end{align*}
where $S^n(\{\tilde \nu^n_t, \beta\tilde k_t)\}_{0\leq t\leq 1})$ is as defined \eqref{e:defSn}.
Then under $\bQ^{\beta\tilde k}$, the measure valued process $\{\tilde \nu^{n}_t\}_{0\leq t\leq 1}$ has the same law as 
\begin{align}\label{e:DBMaa}
\rd \tilde x_i=\frac{\rd W_i}{\sqrt {\beta n}}+\frac{1}{2n}\left(\sum_{j: j\neq i}\frac{1}{\tilde x_i-\tilde x_j}\right)\d t+\del_x \tilde k_t(\tilde x_i)\rd t.
\end{align}
Thanks to Proposition \ref{p:lowerbound2},  it holds with probability $1-\oo(1)$
\begin{align}\label{e:xkbba}
\sup_{t\in [0,1]}\max_{1\le i\le n} | x_i(t)-\gamma_i(t)|\leq  e^{Kt}\left(\max_{1\le j\le n} (| x_j(0)-\gamma_{j}(0)|)+\frac{M}{\sqrt n}\right)=\oo_n(1),
\end{align}
where the constant $K$ depends on $\tilde \del_x k_t(x)$ and $M$ is stochastically bounded. Especially, \eqref{e:xkbba} implies that
$\bQ^{\beta\tilde k}(\{\tilde \nu^{n}_t\}_{0\leq t\leq 1}\in \bB(\{\tilde \rho_t\}_{0\leq t\leq 1}, \delta/2), \tilde s_n(t)\geq \fa, 0\leq t\leq 1)$ with probability $1-\oo(1)$.
We can conclude that
\begin{align}\begin{split}\label{e:replace2a}
&\phantom{{}={}}\bQ(\{\tilde \nu^{n}_t\}_{0\leq t\leq 1}\in \bB(\{\tilde \rho_t\}_{0\leq t\leq 1}, \delta/2),\tilde s_n(t)\geq \fa, 0\leq t\leq 1)\\
&=\bQ^{\beta\tilde k}( e^{-n^2 S^n(\{\tilde \nu^n_t, \beta\tilde k_t\}_{0\leq t\leq 1})}\bm1(\{\tilde \nu^{n}_t\}_{0\leq t\leq 1}\in \bB(\{\tilde \rho_t\}_{0\leq t\leq 1}, \delta/2),\tilde s_n(t)\geq \fa, 0\leq t\leq 1))\\
&=\exp\{-n^2(S_{\tilde \rho_0}(\tilde \rho_t)+\oo_\delta(1))\}\bQ^{\beta\tilde k}( \{\tilde \nu^{n}_t\}_{0\leq t\leq 1}\in \bB(\{\tilde \rho_t\}_{0\leq t\leq 1}, \delta/2, \tilde s_n(t)\geq \fa, 0\leq t\leq 1)\\
&=\exp\{-n^2(S_{\tilde \rho_0}(\tilde \rho_t)+\oo_\delta(1))\}(1-\oo(1)),
\end{split}\end{align}
where in the third line, we used that $S^n(\{\tilde \nu^n_t, \beta\tilde k_t\}_{0\leq t\leq 1})=S_{\tilde \rho_0}(\tilde \rho_t)+\oo_\delta(1)$ for $\{\tilde \nu^{n}_t\}_{0\leq t\leq 1}\in \bB(\{\tilde \rho_t\}_{0\leq t\leq 1}, \delta/2)$.
This finishes the proof of the  large deviations lower bound.

\subsection{Proof of Theorem \ref{main2}}\label{s:pmain2}
In this section, we prove Theorem \ref{main2}. The statement that $S_{\hat\nu_0}^\al$ is a good rate function follows from Proposition \ref{p:rate2}. The weak large deviations upper bound and lower bound \eqref{e:ulbb} are proven in Sections \ref{s:ldup} and \ref{s:lddown} respectively. In this section we show  that the distribution of $\{\hat\nu^n_t\}_{t\in [0,1]}$ 
satisfying $S_{\hat\nu_0}^\al(\{\hat\nu^n_t\}_{t\in [0,1]})$ is exponentially tight. Then the full large deviation principle follows from the weak large deviations upper bound and lower bound \eqref{e:ulbb}.

 The arguments are very similar to those of  \cite{GZ3} and \cite[Section 2.3]{CDG1} and we therefore only outline them. We see $\{\hat\nu^n_t\}_{0\leq t\leq 1}$ as a continuous process with values on the space of symmetric  probability measures $\mathbb M_1^s(\mathbb R)$ on $\mathbb R$.
We denote by $C([0,1],\mathbb M_1^s(\mathbb R))$ this set.  Because $\mathbb M_1^s(\mathbb R)$ is a closed subset of $\mathbb M_1(\mathbb R)$, its 
  compact sets have the same form and we consider the following compact sets:
\begin{align*}
\mathcal K_{M,\delta }:=&\bigcap_{p\in \mathbb N}\left\{\{\hat\nu_t\}_{0\leq t\leq 1}:\sup_{0\le t\le 1}\hat\nu_t([-M_p,M_p]^c)\le \frac{1}{p}\right\}\\
&\bigcap_{i\in  \N}\bigcap_{m\in \N} \left\{\{\hat\nu_t\}_{0\leq t\leq 1}:  \sup_{|s-t|\le \delta_{m,i}}|\hat\nu_t(f_i)-\hat\nu_s(f_i)|\le \frac{1}{m}\right\},
\end{align*}
where $f_i$ is a dense set of bounded continuous functions on $\mathbb R$ and $\delta_{m,i}$ and $M_p$ are sequences of positive real numbers. We need to show that we can choose the functions $f_i$,  such that for each $L>0$, there exists $\delta=\delta(L)$ and $M_p=M_p(L)$ such that
\begin{equation}\label{exptight}
\bP\left( \{\hat\nu_t\}_{0\leq t\leq 1}\notin \mathcal K_{M,\delta }\right)\le e^{-Ln^2}.
\end{equation}
We first show that for any positive real number $L$ and integer number $p$, we can find $M_{p}(L)$ such that 
\begin{equation}\label{compactM}
\sum_{p\ge p_0} \bP\left(\sup_{t\in [0,1]}\hat\nu^n_t([-M_p(L),M_p(L)]^c)\ge  \frac{1}{p}\right)\le e^{-n^2 L}\,.
\end{equation}

The proof from \cite{CDG1} uses the eigenvalue matrix  representation of the Dyson Brownian motion of the special cases $\beta=1$ or $2$. We therefore show how to extend this proof to all $\beta\ge 1$ and $\alpha_n\ge 0$. 
To this end we use for $\varepsilon>0$, the smooth function $f_\varepsilon(x)=   x^2/(1+\varepsilon x^2)$ in \eqref{e:df22} and notice as in the proof of \eqref{e:L2norm} that
$$\frac{1}{2n}L^{f_\varepsilon}_t\ge \hat\nu^n_t(f_\varepsilon) -\fC, $$
with a constant $\fC$ independent of $\varepsilon$.
Moreover $\langle L^{f_\varepsilon}_t,L^{f_\varepsilon}_t\rangle_t\le 4 \int_0^t \hat\nu_s(f_\varepsilon) \rd s$. On the other hand, for any $L>0$,
 the set $A^n_{R,\varepsilon}=\{\sup_{0\le t\le 1} e^{\frac{n}{2} L^{f_\varepsilon}_t-\frac{n^2}{8} \langle L^{f_\varepsilon}_t\rangle_t}\le e^{n^2 R}\}$ satisfies by Doob's inequality
$$\mathbb P\left( (A^n_{R,\varepsilon})^c\right)\le  \E[e^{n L^{f_\varepsilon}_1-\frac{n^2}{2} \langle L^{f_\varepsilon}_1, L^{f_\varepsilon}_1\rangle_1}]
e^{-n^2 R}=e^{-n^2 R}\,.$$
But on $A^n_{R,\varepsilon}$, we have for all $t\in [0,1]$
$$  \hat\nu_t^n(f_\varepsilon) -\fC- 2\int_0^t \hat\nu_s(f_\varepsilon) \rd s\le R,$$ 
and therefore by Gronwall's lemma 
$$\sup_{t\in [0,1]} \hat\nu_t^n(f_\varepsilon)\le 2(\fC+R)\,.$$
Finally, Tchebyshev's inequality yields, since $f_\varepsilon\ge 1/(2\varepsilon)$ on $[-\varepsilon^{-1/2}, \varepsilon^{-1/2}]^c$,
$$\sup_{t\in [0,1]} \hat\nu_t^n([-\varepsilon^{-1/2}, \varepsilon^{-1/2}]^c)\le 4\varepsilon( \fC+R)\,.$$
Hence, taking $R=L+p$,  $\varepsilon=(4p(\fC+L+p))^{-1}$ and $M_p(L)=\varepsilon^{-1/2}$, yields
$$\bP\left(\sup_{t\in [0,1]}\hat\nu^n_t([-M_p(L),M_p(L)]^c)\ge  \frac{1}{p}\right)\le \mathbb P\left( (A^n_{R,\varepsilon})^c\right)\le  e^{-n^2 (L+p)},$$
which completes the proof of \eqref{compactM} after summing over $p$. 
The proof that for any twice continuously differentiable function $f$ for any $L>0$ and $m\in\mathbb N$ we can find $\delta_{m,i}>0$ such that
$$\mathbb P\left(  \sup_{|s-t|\le \delta_{m,i}}|\hat\nu_t^n(f_i)-\hat\nu_s^n(f_i)|\ge \frac{1}{m} \right)\le e^{-L n^2},$$
follows exactly the proof of \cite[Lemma 2.5]{CDG1}. We therefore omit it.

\section{Applications}
As consequences of the large deviation principle of the Dyson Bessel process, we derive the asymptotics of the rectangular spherical integral in Section \ref{sph-sec}, and prove Theorem \ref{main1}. In Section \ref{s:jointlaw}, we characterize the limiting joint law of  $(A_n,UB_nV^*)$ which follows 
\begin{align*}
\rd\mu_{n,m}(U,V)=\frac{e^{\beta  n\Re[\Tr(A_n^*UB_nV^*)]}}{Z_{n,m}}\rd U\rd V.
\end{align*}

\subsection{Asymptotics of rectangular spherical integral}\label{sph-sec}
As the first application of our large deviation principle for the Dyson Bessel process, we prove Theorem \ref{main1} the asymptotics of rectangular spherical integral,
\begin{align*}
\lim_{n}\frac{1}{n^{2}}\log I_{n,m}(A_n,B_n)=\frac{\beta}{2} I^{\alpha}(\hat \nu_{A},\hat\mu_{B}),\quad I_{n,m}(A_n,B_n)=\int e^{\beta  n\Re[\Tr(A_n^*UB_nV^*)]}\rd U\rd V,
\end{align*}

\begin{proof}[Proof of Theorem \ref{main1}]
We recall from \eqref{e:defX}, $X_n$ is an $n\times m$ rectangular random matrix
\begin{align*}
X_n=A_n+\frac{1}{\sqrt n}G_n.
\end{align*}
where $G_n$ is an $n\times m$ rectangular matrix with independent real ($\beta=1$) or complex ($\beta=2$) Gaussian entries.
We denote the singular value decomposition of $X_n$ as $X_n=UB_nV^*$. Then the law of $X_n$ is
\begin{align}\begin{split}\label{e:lawXX}
\frac{1}{Z_{n,m}}\prod_i b_i^{\beta(m-n+1)-1}
\prod_{i<j}|b^2_i-b^2_j|^\beta e^{-\frac{\beta n}{2}(\sum_i b^2_i+\sum a^2_i)+\beta n\Re[\Tr(A_n^*UB_nV^*)]}\rd U\rd V\rd B_n.
\end{split}\end{align}
The large deviations principle of Dyson Bessel process gives 
\begin{align*}
\lim_{n\rightarrow \infty}\frac{1}{n^2}\log \bP(\hat\nu^n_{B} \in \bB(\hat\nu_B,\delta))
=\inf_{\hat\nu_{1}=\hat\nu_{B}}S_{\hat\nu_A}^\al(\{\hat\nu_{t}\}_{0\leq t\leq 1})+\oo_\delta(1),
\end{align*}
where $\oo_\delta(1)$ goes to zero as $\delta$ goes to zero.
By integrating \eqref{e:lawXX} over the ball $\bB(\hat\nu_B,\delta)$, we have
\begin{align*}
&\phantom{{}={}}\int_{\hat\nu^n_{B}\in \bB(\hat\nu_B,\delta)}\frac{1}{Z_{n,m}}\prod_i b_i^{\beta(m-n+1)-1}
\prod_{i<j}|b^2_i-b^2_j|^\beta e^{-\frac{\beta n}{2}(\sum_i b^2_i+\sum a^2_i)+\beta n\Re[\Tr(A^*UBV^*)]}\rd U\rd V\rd B_n\\
&=\frac{1}{Z_{n,m}} 
e^{\frac{\beta n^2}{2} (2  \al \int \log |x|\rd \hat\nu_B+2\Sigma(\hat\nu_B)-(\hat\nu_A( x^2) +\hat\nu_B(x^2))+\oo_\delta(1))}
\int_{\hat\nu^n_{B}\in \bB(\hat\nu_B,\delta)} \int e^{\beta n\Re[\Tr(A^*UBV^*)]}\rd U\rd V\rd B_n
\end{align*}
where we use similar techniques than in \cite{BAG} to prove that even though the logarithm is singular, $ \int \log |x|\rd \hat\nu^n_B$ is close to $\int \log |x|\rd \hat\nu_B$ on the ball (and similarly for the non-commutative entropy term). 
By rearranging, we obtain the following asymptotics of the spherical integral
\begin{align}\begin{split}\label{e:sphi}
&\phantom{{}={}}\lim_{n\rightarrow \infty}\frac{1}{n^{2}}\log I_{ n,m}(A_n,B_n)
=-\inf_{\nu_{1}=\hat\nu_{B}}S_{\mu_A}^\al(\{\nu_{t}\}_{0\leq t\leq 1})\\
&\qquad -\frac{\beta}{2}\left(2\al\int \log x \rd\hat\nu_{B}(x)+2\Sigma(\hat\nu_B)-(\hat\nu_{A}(x^{2})+\hat\nu_{B}(x^{2})) \right)+\const.
\end{split}\end{align}
Thanks to Proposition \ref{p:rate2}, if $S_{\hat \nu_A}^\al(\{\hat\nu_{t}\}_{0\leq t\leq 1})<\infty$, then $\hat\nu_t$ has a density, i.e. $\{\hat\nu_t\}_{0\leq t\leq 1}=\{\hat\rho_t(x)\rd x\}_{0\leq t\leq 1}$ is a symmetric measure valued process, satisfying the weak limits
\begin{align}\label{e:bbterm}
\lim_{t\rightarrow 0}\hat\rho_t(x)\rd x=\hat\nu_A,\quad
\lim_{t\rightarrow 1}\hat\rho_t(x)\rd x=\hat\nu_B.
\end{align}
Let $u_s$ be the weak solution of the following conservation of mass equation
\begin{align}\label{e:cmass}
\del_s\hat \rho_s+\del_x(\hat \rho_s u_s)=0.
\end{align}
We recall the following formula for the dynamical entropy $S_{\hat\nu_A}^\al(\{\hat\nu_{t}\}_{0\leq t\leq 1})$ from \eqref{Sent2}, 
\begin{align}\begin{split}\label{e:largeupb3}
S_{\mu_0}^\al(\{\hat\nu_t\}_{0\leq t\leq 1})&=\frac{\beta}{2}\left(
\int_0^1 \int u_s^2  \hat\rho_s(x)\rd x \rd s+\frac{\pi^2}{3}\int_0^1\int \hat \rho^3_s(x) \rd x  \rd s
+\frac{\alpha^2}{4}\int \frac{\hat\rho_s(x)}{x^2}\rd x\rd s\right.\\
&-\left.\left.\left(\Sigma(\hat\nu_t)+ \al \int \log |x|\rd\hat\nu_t(x)\right)\right|_{t=0}^1
\right).
\end{split}\end{align}
By plugging \eqref{e:largeupb3} into \eqref{e:sphi}, we obtain the following theorem on the asymptotics of rectangular spherical integral,
\begin{align}\begin{split}\label{e:aratecopy}
&I^{\alpha}(\mu_{A},\mu_{B})=-\inf_{\{\hat\rho_t\}_{0\leq t\leq1} \atop \text{satisfies \eqref{e:bbterm}}}\left\{
\int_0^1 \int u_s^2\hat\rho_s \rd x\rd s+\frac{\pi^2}{3}\int_0^1\int  \hat\rho^3_s \rd x\rd s
+\frac{\alpha^2}{4}\int \frac{\hat\rho_s(x)}{x^2}\rd x\rd s\right\}\\
&+(\hat\nu_{A}(x^{2}-\al\log |x|)+\hat\nu_{B}(x^{2}-\al\log |x|))-(\Sigma(\hat\nu_A)+\Sigma(\hat\nu_B)) +\const.
\end{split}\end{align}
This finishes the proof of Theorem \ref{main1}.
\end{proof}

In the remaining of this section, we give an informal characterization of the minimizer in \eqref{e:aratecopy}, by the complex Burger's equation.
We denote the minimizer of \eqref{e:aratecopy} as $\{\hat\rho_t^*\}_{0\leq t\leq 1}$, then it satisfies the following  Euler equation gives
\begin{align}\label{e:euler}
\del_t u_t+\frac{1}{2}\del_x(u_s^2-\pi^2(\hat\rho^*_s)^2-\frac{\alpha^2}{4x^2})=0.
\end{align}
We define the function 
\begin{align*}
f_t(x)=u_t(x)+{\rm i} \pi \hat\rho^*_t(x).
\end{align*}
Then thanks to the relations \eqref{e:cmass} and \eqref{e:euler}, $f_t(x)$ satisfies the following complex burger's equation
\begin{align}\label{e:burgeq}
\del_t f_t(x)+\del_x f_t(x) f_t(x)=\frac{\alpha^2}{4x^3}.
\end{align}
The complex burger's equation can be solved by characteristic flow formally. Let
\begin{align*}
\del_t z_t=f_t(z_t)=p_t,
\end{align*}
then 
\begin{align*}
\del_t f_t(z_t)=\del_t p_t=\frac{\al^2}{4z_t^3}
\end{align*}
There are two quantities conserved:
\begin{align*}
&\del_t\left(p_t^2+\frac{\al^2}{4z_t^2}\right)=2\del_t p_t p_t-\del_t z_t\frac{\al^2}{2z_t^3}=0,\\
&\del_t\left(z_tp_t-t\left(p_t^2+\frac{\al^2}{4z_t^2}\right)\right)=\del_t z_t p_t+z_t\del_t p_t-p_t^2-\frac{\al^2}{4z_t^2}=0.
\end{align*}
Therefore, we have
\begin{align*}
&f^2_t(z_t)+\frac{\al^2}{4z_t^2}=f^2_0(z)+\frac{\al^2}{4z^2},\\
&z_tf_t(z_t)
=zf_0(z)+t\left(f^2_0(z)+\frac{\al^2}{4z^2}\right).
\end{align*}
Solving them we get
\begin{align*}
z_t=\sqrt{\frac{f^2_0(z)+\frac{\al^2}{4z^2}}{\al^2/4+\left(zf_0(z)+t\left(f^2_0(z)+\frac{\al^2}{4z^2}\right)\right)^2}},\quad
f_t(z_t)=\frac{zf_0(z)+t\left(f^2_0(z)+\frac{\al^2}{4z^2}\right)}{z_t}.
\end{align*}

\subsection{Joint law of $A, UBV^*$}\label{s:jointlaw}
Let $A_n, B_n\in \bR^{n\times m}$ and  $U\in \cO(n),V\in \cO(m)$ following Haar distribution over orthogonal group for $\beta=1$;  $A_n, B_n\in \bC^{n\times m}$ and $U\in \cU(n),V\in \cU(m)$ following Haar distribution over unitary group, for $\beta=2$,  where $m\geq n$ and $m/n\rightarrow 1+\alpha, \alpha\ge 0$. We assume that  the symmetrized empirical singular values  $\hat\nu_A^n$ and $\hat \nu_B^n$ of $A_n$ and $B_n$ converges to $\hat\nu_A$ and $\hat\nu_B$ respectively.
In this section we consider the non-commutative joint distribution of $(A_n,UB_nV^*)$ under
\begin{align}\label{e:lawUb}
\rd\mu_{n,m}(U,V)=\frac{e^{\beta  n\Re[\Tr(A_n^*UB_nV^*)]}}{Z_{n,m}}\rd U\rd V.
\end{align}
To do it, we construct $\mathcal A_n,\mathcal B_n$ to be the hermitized version of these operators:
\begin{align*}
\mathcal A_n=\left(\begin{array}{cc}
0&A_n\cr
A_n^{*}&0\cr\end{array}\right),
\quad 
\mathcal B_n= \left(\begin{array}{cc}
0&UB_nV^* \cr
VB_n^{*}U^{*}&0\cr\end{array}\right)
=\left(\begin{array}{cc}
U&0\cr
0&V\cr\end{array}\right)\left(\begin{array}{cc}
0&B_n\cr
B_n^{*}&0\cr\end{array}\right)
\left(\begin{array}{cc}
U^*&0\cr
0&V^*\cr\end{array}\right).
\end{align*}
With the hermitized  operators $\cA_n, \cB_n$, we can rewrite the law $\rd\mu_{{n,m}}(U,V)$ from \eqref{e:lawUb} as
\begin{align}\label{e:lawUb2}
\rd\mu_{{n,m}}(U,V)=\frac{e^{\beta  n\Re[\Tr(\cA_n\cB_n)]/2}}{Z^\beta_{m,n}}\rd U\rd V,
\end{align}
where $\d U$ denotes the Haar measure on the Unitary (resp. orthogonal) group when $\beta=2$ (resp. $\beta=1$). 
We denote by $\mu_{\mathcal A_n,\mathcal B_n}$ the non-commutative distribution of $(\mathcal A_n,\mathcal B_n)$ given by 
\begin{equation}\label{defher}\mu_{\mathcal A_n,\mathcal B_n}(P)=\frac{1}{ n} \tr(P(\mathcal A_n,\mathcal B_n))\,.\end{equation}
where $P$ belongs to the set $\mathbb C\langle X_1,X_2\rangle$ of non-commutative polynomials in two  self-adjoint variables. 
We recall that $\mathbb C\langle X_1,X_2\rangle$ is the linear span of words in $X_1,X_2$ endowed with the convolution 
$$\left(z X_{i_1}\cdots X_{i_p}\right)^*=\bar zX_{i_p}\cdots X_{i_1},$$ for any $i_j\in \{1,2\}$ and $z\in \mathbb C$. 
We denote the space of non-commutative laws as
\begin{align*}
\cM=\{\tau\in \mathbb C\langle X_1,X_2\rangle^*: \tau(I)=1, \tau(PP^*)\geq 0, \tau(PQ)=\tau(QP),\forall P,Q\in \mathbb C\langle X_1,X_2\rangle\}\,.
\end{align*}
We recall that for any $L>0$, the subset 
$$\cM_L=\{\tau\in \cM: \max_{i=1,2 \atop n\in \mathbb N} L^{-2n}\tau(X_i^{2n})\le 1\},$$
of $\cM$  is a compact metric space. Hereafter, we will concentrate on non-commutative laws with given marginal distributions
$\hat\nu_A,\hat\nu_B$ compactly supported on $[-L.L]$ for some finite $L$:
$$\cM_{\hat\nu_A,\hat\nu_B}=\{ \tau\in \cM_L: \tau(X_1^k)=\hat\nu_A(x^k),\tau(X_2^k)=\hat\nu(x^k), \forall k\in \mathbb N\}\,,$$ 
which is also compact.

\begin{proposition} \label{p:uniquelimit}Let $A_n, B_n\in \bR^{n\times m}$ and  $U\in \cO(n),V\in \cO(m)$ following Haar distribution over orthogonal group for $\beta=1$;  $A_n, B_n\in \bC^{n\times m}$ and $U\in \cU(n),V\in \cU(m)$ following Haar distribution over unitary group, for $\beta=2$,  where $m\geq n$ and $m/n\rightarrow 1+\alpha, \alpha \ge 0$. We assume that  the symmetrized empirical singular values  $\hat\nu_A^n$ and $\hat \nu_B^n$ of $A_n$ and $B_n$ converges to $\hat\nu_A$ and $\hat\nu_B$ respectively.
We further assume that $A_n,B_n$ are uniformly bounded for the operator norm. Then $\mu_{{\mathcal A_n,\mathcal B_n}}$ as defined in \eqref{e:lawUb2} converges almost surely towards a tracial state $\tau$ which depends only on $\hat\nu_{A}$ and $\hat\nu_{B}$.
\end{proposition}
The proof starts by noticing that  the convergence of the trace of powers of linear combinations of $\mathcal A_n,\mathcal B_n$ follows from  the large deviations of Dyson Bessel proces. We then show that the non-commutative law $\mu_{{\mathcal A_n,\mathcal B_n}}$ is tight for the weak topology since the variables are uniformly bounded  (and therefore $\mu_{{\mathcal A_n,\mathcal B_n}}\subset \cM_L$ for some finite $L$ and all $n\in\mathbb N$) and that any limit point satisfies the so-called 
 loop equation. The convergence of powers of linear combinations of $\mathcal A_n,\mathcal B_n$ and the loop equation will then be shown to uniquely characterize the limit.

\begin{proof}

We recall the  real/complex Brownian motions starting from $A_n$ from \eqref{e:defBB}: 
\begin{align*}
H(t)=A_n+\frac{1}{\sqrt{n}}G(t). 
\end{align*} 
If we condition on that
the singular values of $H(1)$ are given by $B_n$, i.e. $H(1)=UB_nV^*$, then the joint law of $U,V$ is given by \eqref{e:lawUb}. 
If we further condition on $U,V$, i.e. we condition on that $H(1)=UB_nV^*$, then the law of $\{H_{ij}(t)\}_{0\leq t\leq 1}$ is the same as a Brownian bridge from $H_{ij}(0)$ to $H_{ij}(1)$. Therefore
\begin{align}\label{e:sumABW}
H(t)\stackrel{d}{=} (1-t)A_n +tUB_nU^*+\sqrt{t(1-t)}W_n/\sqrt n,
\end{align}
where $W_n$ is an $n\times  m$ matrix with entries given by independent  real or complex Gaussian random variables. Each entry has mean zero and variance one.
We denote the Hermitized version of $H(t)$ as
\begin{align*}
\mathcal H(t)= \left(\begin{array}{cc}
0&H(t) \cr
H(t)^*&0\cr\end{array}\right).
\end{align*}
The above discussion implies that if we
condition on $\cH(1)$ 
\begin{align*}
\mathcal \cH(1)=
\mathcal B_n= \left(\begin{array}{cc}
0&UB_nV^* \cr
VB_n^{*}U^{*}&0\cr\end{array}\right),
\end{align*}
Then the limiting law of the spectral measure of $\cH(t)$ as $n$ goes to infinity is characterized by the rectangular convolution of  the limiting law of $\hat\nu^n_{(1-t)\cA_n+t \cB_n}$ using the relation \eqref{e:recst}. For any limiting joint law $\tau\in \cM$ of $\cA_n, \cB_n$, it is necessary that
\begin{align}\label{e:boundary}
\tau(\sfa^k)= \int x^{k}\rd\hat\nu_A(x),\quad \tau(\sfb^k)= \int x^k\rd\hat\nu_B(x),
\end{align}
and we claim that we also know for $t\in [0,1]$ and $k\in\mathbb N$ the value of
\begin{align}\label{e:moment}
\tau(((1-t)\sfa+t\sfb)^k).
\end{align}
To see this point, for $\tau\in \cM$ a non-commutative joint law of $\mathsf a,\mathsf b$ and 
 denote $\{\hat\nu^\tau_t\}_{0\leq t\leq 1}$ the measure valued process  such that $\hat\nu^\tau_t$ is the law of $(1-t)\mathsf a+t\mathsf b+\sqrt{t(1-t)} \mathsf w$ where $\mathsf w$ is a symmetrized Pastur-Marchenko law (the limit distribution of nonzero eigenvalues of $\mathcal H(1)$). 
 
 By a tightness argument as in \cite[Section 4.1]{MR2034487}, we have
\begin{align*}
\inf_{\{\hat\rho_t\}_{0\leq t\leq1} \atop \text{satisfies \eqref{e:bbterm}}} S_{\hat\nu_A}^\al(\{\hat\rho_t\}_{0\leq t\leq 1})
=\inf_{\tau\in \cM,\atop 
\text{satisfies }\eqref{e:boundary}}S_{\hat\nu_A}({\{\hat\mu_t^\tau\}_{0\leq t\leq 1}}),
\end{align*}
and the infimum is achieved at some $\tau^*$ (later we will show that such $\tau^*$ is unique.)
It follows that for all $t\in [0,1]$, and $k\geq 1$,
\begin{align}\label{e:taum}
\tau^*\left(((1-t)\mathsf a-t\mathsf b-\sqrt{t(1-t)}\mathsf s)^k\right)= \int x^k \hat\rho^*_{t}(x) \d x,
\end{align}
and $\hat\rho_t^*$ is analytic for $\hat\rho_t^*(x)>0$, and bounded by $\OO(1/\sqrt{t(1-t)})$.

The relation \eqref{e:taum} is enough to deduce the distribution of $\hat\nu_{(1-t)\mathsf a+t\mathsf b }$ of $(1-t)\mathsf a+t\mathsf b$ thanks to the rectangular free convolution relation \eqref{e:recst}. In fact, thanks to Theorem \ref{t:ctransform}, the rectangular $R$-transform of the measure $\hat\nu_{(1-t)\sfa+t \sfb}$, and $\hat\rho_t^*$ are related by 
$$C_{\hat\nu_{(1-t)\mathsf a+t \mathsf b}}(z)=C_{\hat\rho_t^*}(z)-\frac{t(1-t) z}{\la}\,.$$
The rectangular $R$-transform of $\mu_t$ can be solved in terms of the rectangular $R$-transform of $\hat\rho_t^*$, and it uniquely characterizes $\mu_t$. It gives us the moments \eqref{e:moment} for $\tau^*$.
Next we derive the loop equations for the measure \eqref{e:lawUb2} : they will together with \eqref{e:moment} describe uniquely the non-commutative law $\tau^*$. Let $\mathbb C\langle A,B,U,V\rangle$ denote the set of $*$ polynomials for non-commutative operators $A,B,U,V$.
Under the assumptions of Proposition \ref{p:uniquelimit}, let  $W\in \mathbb C\langle A,B,U,V\rangle$ be a self-adjoint polynomial.
We recall from \cite{CGM,GN}, that for any measure 
\begin{align}\label{e:wmeasure}
\frac{1}{Z_n}e^{\frac{n\beta}{2}\tr W (A_n,B_n,U,V)} \rd U\rd V,
\end{align} 
and any polynomial $P$ in $ \bC\langle A,B,U,V\rangle$,
$$\frac{1}{n}\Tr\otimes\frac{1}{n} \Tr(\partial_U P)+\frac{1}{n}\Tr( P\mathcal D_U W),$$
goes to zero almost surely, where for any monomial $P$ in $A,B,A^*,B^*,U,U^*,V,V^*$
$$\partial_U P=\sum_{P=P_1UP_2} P_1 U\otimes P_2-\sum_{P=P_1U^*P_2} P_1 \otimes U^* P_2\,,$$
and $\mathcal D_U =m\circ \partial_U$ with $m(P\otimes Q)=QP$. Similar statements hold for $V$. 
We denote the normalized trace $\tau_n$ as 
$$\tau_n(P)=\frac{1}{n} \Tr( P(A_n,B_n,U,V,U^*,V^*)).$$
Then $\tau_n$ is tight almost surely, thanks to the uniform boundedness  of $(A_n,B_n,U,V)$. Hence, any limit point $\tau^*$ of $\tau_n$  satisfies
\begin{align}
\label{e:loopeq}\tau^*\otimes \tau^*(\partial_U P)={\color{blue}{-}}\tau^*(P\mathcal D_U W)\,.
\end{align}
To get the rectangular spherical integral \eqref{e:lawUb}, we take  $$W= A^*UBV^*+VB^{*}U^* A,$$ in \eqref{e:wmeasure},   so that
$$\mathcal D_U W=\left(BV^*A^*U-U^*AVB^{*}\right)\,.$$
We take non-commutative polynomial $P$ in the form $P=U^* Q(A,A^*, UBV^*,VB^*U^*)U$. Then, we get 
$$\partial_U P=- 1\otimes P+P\otimes 1+\sum_{Q=Q_1 UBV^* Q_2} U^* Q_1 U \otimes BV^*Q_2 U- \sum_{Q=Q_1 VB^*U^* Q_2} U^* Q_1 VB^* \otimes U^*Q_2 U.$$
Hence, since $U$ is unitary and $\tau_n$ tracial
$$ \tau_n\otimes  \tau_n(\partial_U P)= \sum_{Q=Q_1 UBV^* Q_2} \tau_n( Q_1) \tau_n(UBV^*Q_2)- \sum_{Q=Q_1 VB^*U^* Q_2} \tau_n( Q_1 VB^*U^*)  \tau_n(Q_2 ),$$
whereas
$$\tau_n( P\mathcal D_U W)=\tau_n(U^* QU (BV^*A^*U-U^*AVB^*))=\tau_n(Q(UBV^*A^*-AVB^*U^*))\,.$$
We conclude that $\tau_n$  satisfies the loop equation such that for any polynomial $Q$ in $ \bC\langle A,B,U,V\rangle
$
\begin{align}
\begin{split}\label{e:prelimit}
&\phantom{{}={}}\sum_{Q=Q_1 UBV^* Q_2} \tau_n( Q_1)  \tau_n(UB_nV^*Q_2)- \sum_{Q=Q_1 VB^*U^* Q_2} \tau_n( Q_1 VB^*U^*)  \tau_n(Q_2 )\\
&+\tau_n(Q(UBV^*A^*-AVB^*U^*))=\oo_n(1),
\end{split}
\end{align}
with overwhelming probability.
We denote the the limit of $UB_nV^*$ as $b$,  the limit of $A_n$ as $a$, and for any non-commutative polynomial $p(a,a^*,b,b^*)$,  $\partial_{b^*} p(a,a^{*},b,b^*)= \sum_{p=P_1b^{*}P_2} P_1b^{*} \otimes P_2-\sum_{p=P_1b P_2} P_1 \otimes b P_2$.
Therefore, any limit point $ \tau^*$ of $\tau_n$ satisfies
\begin{equation}
\label{loop}
 \tau^*(p(a b^*-ba^{*}))+\tau^*\otimes \tau^*(\partial_{b^*} p)=0.
\end{equation}
We can proceed similarly with the unitary matrix $V$ leading to a second  equation: with $ \partial_{b}p=\sum_{p=p_{1}b p_{2}}p_{1 }b\otimes p_{2}-\sum_{p=p_{1}b^*p_{2}}p_{1 }\otimes b^* p_{2}$,
\begin{align}\label{bloop}
\tau^*(p(a^*b-b^{*}a))+ \tau^*\otimes  \tau^*( \partial_b p)=0.
\end{align}
We can lift finally these equations at the level of the hermitised matrices $(\mathsf a,\mathsf b)$ and let $\tau^*$ be a limit point for $\mu_{\mathcal A_n,\mathcal B_n}$ from \eqref{defher}. Then observe that if $P$ is a monomial of $(\mathsf a,\mathsf b)$ with odd degree then $\tau^*(P(\sfa,\sfb))=0$; if $P$ has even total degree 
$$P(\mathsf a,\mathsf b)=\left(\begin{array}{cc}
p(a,a^{*},b,b^{*})&0\cr
0&p(a^{*},a,b^{*},b)\cr
\end{array}\right),$$
where $p$ is obtained by replacing in $P$ every even letter by its adjoint. We can then define $\partial_{\mathsf b}$ by putting
$$\partial_{\mathsf b} P=\sum_{P=P_1\mathsf b P_2} ( P_1\sfb\otimes    P_2 - P_1   \otimes \mathsf bP_2).$$
The two loop equations \eqref{loop} and \eqref{bloop} for $ \tau^*$ then summarize into an equation for any limit point $\tau^*$ of $\mu_{\cA_n,\cB_n}$ which reads 
\begin{equation}\label{lop}
\tau^*\otimes \tau^*(\partial_{\mathsf b} P)+\tau^* (P (\mathsf a\mathsf b-\mathsf b\mathsf a))=0.
\end{equation}

We can then proceed as in \cite[Theorem 2.11]{BGHldp} to see that the loop equations \eqref{lop} allows us to commute $\sfa$ and $\sfb$. We can use the loop equations to express the trace of any polynomial in terms of the trace of monomials in the form $\sfa^k \sfb^\ell$. In particular by applying the loop equations 
to \eqref{e:moment}, $\tau^*(\sfa^{k'} \sfb^{k-k'})$ are uniquely determined from the moments . Then the trace of any polynomials are uniquely determined. This gives the uniqueness of $\tau^*$ and completes the proof. 

\end{proof}

\bibliography{references.bib}{}
\bibliographystyle{plain}

\end{document}